\documentclass[a4paper,twoside,10pt]{article}

\usepackage[a4paper,left=3cm,right=3cm, top=3cm, bottom=3cm]{geometry}
\usepackage[latin1]{inputenc}
\usepackage{mathrsfs}
\usepackage{mathtools}
\usepackage{graphicx}
\usepackage{cancel}
\usepackage{epstopdf}
\graphicspath{{figs/}}
\usepackage{amsmath}
\usepackage{footmisc}
\usepackage{amsthm}
\usepackage{amssymb}
\usepackage{stmaryrd}
\usepackage{float}
\usepackage{bigints}
\usepackage{cite}
\usepackage{color}
\usepackage[abs]{overpic}
\usepackage[font=footnotesize,labelfont=bf]{caption}
\usepackage{cases}
\usepackage{tikz}
\usepackage{algorithmic}
\usepackage{rotating}
\usepackage{blkarray}
\usetikzlibrary{matrix,calc,arrows}
\usepackage[author={Lorenzo}]{pdfcomment}
\usepackage{soul,xcolor}
\usepackage{enumitem}  
\usepackage{verbatim}
\usepackage{graphicx}
\usepackage{subcaption}
\usepackage{mwe}
\usepackage{algorithm}
\hypersetup{
    colorlinks=true,
    pdfborder={0 0 0}
} 
\numberwithin{equation}{section} 

\restylefloat{table}
\theoremstyle{plain}
\newtheorem{theorem}{Theorem}[section]

\newtheorem{lemma}[theorem]{Lemma}
\newtheorem{proposition}[theorem]{Proposition}
\theoremstyle{definition}

\theoremstyle{remark}
\newtheorem{remark}{Remark} 

\def\R{\mathbb{R}}
\def\N{\mathbb{N}}

\newcommand{\sg}[1]{{\color{blue}{#1}}} 

\newcommand{\eremk}{\hbox{}\hfill\rule{0.8ex}{0.8ex}}

\def\QT{Q_T}

\def\Rbb{\mathbb R}
\def\Nbb{\mathbb N}
\def\Pbb{\mathbb P}
\newcommand{\Pp}[2]{\Pbb_{#1}\left(#2\right)}
\newcommand{\dpt}[1]{\partial_t {#1}}
\newcommand{\dptt}[1]{\partial_{tt} {#1}}
\def\Rbb{\mathbb{R}}
\def\x{\mathbf x}
\def\nablax{\nabla_{\x}}

\def\hx{\h_{\x}}
\def\Deltax{\Delta_{\x}}
\newcommand{\Norm}[1]{{\left\|{#1} \right\|}}
\newcommand{\SemiNorm}[1]{{\left|{#1} \right|}}
\def\dt{\mbox{dt}}
\def\dx{\mbox{d}\x \ }
\def\taun{\mathcal T_h}
\def\taunx{\taun^{\x}}
\def\E{K}

\def\Ex{\E_{\x}}
\def\tnmo{t_{n-1}}
\def\tn{t_n}
\def\h{h}
\def\hE{\h_\E}
\def\hEx{\h_{\Ex}}
\def\hExo{\h_{\E_{1,\x}}}
\def\hExt{\h_{\E_{2,\x}}}

\def\F{F}
\def\Fx{\F_\x}

\def\hFx{\h_{\Fx}}
\def\Fcalh{\mathcal F_\h}
\def\Fcalxh{\mathcal F_\h^\x}
\def\FcalE{\mathcal F^{\E}}
\def\FcalEx{\mathcal F^{\Ex}}
\def\Fcalxh{\mathcal F^{\x}_\h}
\def\Vh{V_\h}
\def\VhE{\Vh(\E)}
\def\nbf{\mathbf n}
\def\nbfF{\nbf_\F}
\def\nbfFx{\nbf_{\Fx}}
\def\malphaE{m_{\alpha}^\E}
\def\malphaEx{m_{\gamma}^{\Ex}}
\def\malphaF{m_{\beta}^\F}
\def\dS{\mbox{d}S}

\DeclareMathOperator{\DoF}{DoF}
\DeclareMathOperator{\DoFs}{DoFs}
\def\uh{u_h}
\def\uI{u_I}

\def\vh{v_h}
\def\wh{w_h}

\def\Yh{Y_\h}
\def\YE{Y(\E)}
\def\Xhtaun{X(\taun)}
\def\Ytaun{Y(\taun)}
\newcommand{\jump}[1]{\left[\!\left[#1\right]\!\right]}
\def\PiN{\Pi^N_\p}

\def\Pistar{\Pi^\star_\p}
\def\ah{a_\h}
\def\bE{b^\E}
\def\bh{b_\h}
\def\bhE{\bh^\E}
\def\ahE{\ah^\E}
\def\aE{a^\E}
\def\S{S}
\def\SE{\S^\E}

\def\Newton{\mathfrak N}
\def\Newtonh{\Newton_\h}
\def\NCh{\mathcal N_h}

\def\gammaI{\gamma_I}
\def\qp{q_\p}

\def\qpt{\qp^t}
\def\qpto{\qp^{t,1}}
\def\qptt{\qp^{t,2}}
\def\qptj{\qp^{t,j}}
\def\qpmo{q_{\p-1}}

\def\Scalptaun{\mathcal S_\p(\taun)}
\def\Scalelltaun{\mathcal S_\ell(\taun)}

\def\Pizpmo{\Pi^{0,\QT}_{\p-1}}
\def\PizIn{\Pi^{0,\In}_{\p-1}}
\def\PizoEx{\Pi^{0,\Ex}_0}
\def\PizE{\Pi^{0,\E}_{\p-1}}
\def\PizEx{\Pi^{0,\Ex}_{\p}}
\def\PizExz{\Pi^{0,\Ex}_{0}}
\def\PizF{\Pi^{0,\F}_{\p}}
\def\p{p}
\def\ctilde{\widetilde c}
\def\CP{C_P}
\def\CPE{\CP^\E}

\def\phih{\phi_\h}

\def\cH{c_H} 
\def\ctildeHE{\widetilde c_H^\E}
\def\nutildeE{\widetilde \nu^\E}
\def\Ehat{\widehat \E}
\def\cPiN{c_{\PiN}}
\def\cPistar{c_{\Pistar}}
\def\ctrace{c_{tr}}
\def\CPB{c_{PB}}
\def\Jcal{\mathcal J}

\def\JcalE{\Jcal^\E}

\def\In{I_n}
\def\Inmo{I_{n-1}}
\def\Inpo{I_{n+1}}
\def\htime{\h_{\In}}
\def\EcalY{\mathcal{E}^Y}
\def\EcalU{\mathcal{E}^U}

\def\EcalN{\mathcal{E}^N}
\def\EcalL{\mathcal{E}^L}

\date{}
\title{Space-time virtual elements for the heat equation}
\author{Sergio G\'omez\thanks{Department of Mathematics, University of Pavia, 27100 Pavia, Italy (sergio.gomez01@universitadipavia.it, 
andrea.moiola@unipv.it)},\thanks{Faculty of Informatics, Universit\`a della Svizzera italiana, Lugano, Switzerland (gomezs@usi.ch)}\;
Lorenzo Mascotto\thanks{Department of Mathematics and Applications, University of Milano Bicocca, 20125 Milan, Italy (lorenzo.mascotto@unimib.it)}
\thanks{Faculty of Mathematics, University of Vienna, 1090 Vienna, Austria (lorenzo.mascotto@univie.ac.at, ilaria.perugia@univie.ac.at)}
\thanks{IMATI-CNR, Pavia, Italy},\;
Andrea Moiola\footnotemark[1],\;
Ilaria Perugia\footnotemark[4]}

\begin{document}
\maketitle

\centerline{\today}

\begin{abstract}
\noindent We propose and analyze a space-time virtual element method for the discretization of the heat equation in a space-time cylinder,
based on a standard Petrov-Galerkin formulation.
Local discrete functions are solutions to a heat equation problem with polynomial data.
Global virtual element spaces are nonconforming in space,
so that the analysis and the design of the method
are independent of the spatial dimension.
The information between time slabs is transmitted by means of upwind terms
involving polynomial projections of the discrete functions.
We prove well posedness and optimal error estimates for the scheme,
and validate them with several numerical tests.
	

\end{abstract}

\section{Introduction} \label{section:introduction}

The virtual element method (VEM) was introduced in~\cite{Beirao-Brezzi-Cangiani-Manzini-Marini-Russo:2013} as an extension of the finite element method to general polytopic meshes
for the approximation of solutions to the Poisson equation.
Trial and test spaces consist of functions that are solutions to local problems related to the PDE problem to be approximated.
Moreover, they typically contain polynomials of a given maximum degree,
together with nonpolynomial functions allowing for the enforcement of the desired type of conformity in the global spaces.
Such functions are not required to be explicitly known.
Suitable sets of degrees of freedom ($\DoFs$) are chosen
so that projections from local VE spaces onto polynomial spaces can be computed out of them.
Such polynomial projectors and certain stabilizing bilinear form
are used to define the discrete bilinear forms.
A nonconforming version of the VEM was proposed in~\cite{nonconformingVEMbasic}.
Unlike its conforming counterpart, the nonconforming VEM can be presented in a unified framework for any dimension,
which significantly simplifies its analysis and implementation. 

In the VEM literature,
time dependent problems have always been tackled by combining
a VE discretization in space
with a time-stepping scheme
for the solution to the resulting ODE system.
The prototypical example is~\cite{VEMparabolic2015}, where the heat equation was considered.
On the other hand, space-time Galerkin methods are based on discretizing the space and time variables of a PDE at once.
These methods provide a natural framework where high-order accuracy can be obtained in both space and time,
and an approximate solution is available on the whole space-time domain.

In this paper, we design and analyze the first space-time VEM for the solution to a time-dependent PDE, namely, the heat equation;
we can consider spatial domains in one, two, and three dimensions.
We employ prismatic-type elements.
This allows us to distinguish two types of mesh facets:
\emph{space-like} facets, i.e.,
facets lying on hyperplanes in 
space-time
that are perpendicular to the time axis;
\emph{time-like} facets, i.e., facets 
whose normals are perpendicular to the time axis.
The method we propose is based on a standard space-time variational formulation of the heat equation
in the space-time cylinder $Q_T = \Omega \times (0, T)$ with trial space $L^2(0, T; H^1_0(\Omega)) \cap H^1(0, T; H^{-1}(\Omega)$ and test space $L^2(0, T; H_0^1(\Omega))$; see \cite[Ch.~XVIII, Sec.~4.1]{Dautray-Lions:1992}. 

For a recent survey of space-time  discretizations of parabolic problems, we refer to~\cite{Langer_Steinbach_2019}.
In particular,
a continuous finite element discretization of the standard Petrov-Galerkin variational formulation is presented and analyzed in~\cite{steinbach2015CMAM}.
Additionally, we refer to~\cite{Andreev:2013}, \cite{Schwab_Stevenson_2009}, and~\cite{Stevenson_Westerdiep_2021}
for wavelet- or finite element-type discretizations based on a minimal residual Petrov-Galerkin formulation, and to~\cite{NeumuellerPhD} and~\cite{Cangiani_Dong_Georgoulis_2016} for discontinuous Galerkin approaches.
Motivated by the boundary integral operator analysis,
a continuous finite element method based on a fractional-order-in-time variational formulation was studied in~\cite{Steinbach_Zank_2020}.
Recently, the space-time first order system least squares (FOSLS)
formulation of~\cite{Bochev_Gunzburger_2009} has been revisited and analyzed;
see~\cite{Fuehrer_Karkulik_2021}, \cite{Gantner_Stevenson_2021},
and~\cite{Gantner_Stevenson_2022}. \medskip

\noindent
We summarize the main features of the proposed VEM.
\medskip

\begin{itemize}
\item Local VE spaces consist of functions that solve
a heat equation with polynomial data on each space-time element;
this makes the method particularly suitable for further extensions,
e.g., to its Trefftz variant. 
\item We consider tensor-product in time (prismatic) meshes but the VE spaces are not of tensor-product type.
Even for prismatic elements with simplicial bases,
the proposed VE spaces do not coincide with their standard tensor-product finite element counterparts.
\item Global VE spaces involve approximating continuity constraints across mesh facets.
More precisely,
we impose nonconformity conditions on time-like facets
analogous to those in~\cite{nonconformingVEMbasic} for the Poisson problem,
and allow for discontinuous functions in time.
Across space-like facets,
we transmit the information between consecutive time slabs by upwinding.
In the present VEM context, the upwind terms are defined
by means of a polynomial projection.
\item To keep the presentation and the analysis of the method as simple as possible, the details are presented for the particular case of space-time tensor-product meshes.
However, as discussed in Subsection~\ref{subsection:more-general-meshes} below,
the method can handle nonmatching time-like or space-like facets, which is greatly advantageous for space-time adaptivity. 
\end{itemize}
\medskip

\noindent
We summarize the advantages of the proposed
space-time VEM over standard space-time conforming finite element methods.
\medskip

\begin{itemize}
\item The nonconforming VEM setting is of arbitrary order
and its design is independent of the spatial dimension.
\item Nonmatching space-like and time-like facets, which
naturally stem from mesh adaptive procedures,
can be handled easily.
\item As the discrete spaces are discontinuous in time,
we can solve the global (expensive) problem
as a sequence of local (cheaper) problems on time slabs.
\item The definition of the local spaces allows for the construction of space-time discrete Trefftz spaces.
\end{itemize}
\medskip

\noindent The main advancements of this paper are the following.\medskip

\begin{itemize}
    \item We design a novel space-time VEM for the heat equation in any spatial dimension.
    \item We prove its well posedness and
    optimal \textit{a priori} error estimates.
    \item We validate numerically the theoretical results on some test cases. 
\end{itemize}
\medskip

\paragraph*{Notation}
We denote the first and second partial derivatives with respect to the time variable $t$ by~$\dpt{}$ and~$\dptt{}$, respectively,
and the spatial gradient and Laplacian operators by~$\nablax{}, \ \Deltax{}$, respectively.
Throughout the paper, standard notation for Sobolev spaces will be employed.
For a given bounded Lipschitz domain~$D \subset \R^d$ ($d \in \N$), $H^s(D)$ represents the standard Sobolev space of order $s \in \N$,
endowed with the standard inner product $(\cdot, \cdot)_{s,D}$, the seminorm $\SemiNorm{\cdot}_{s,D}$, and the norm~$\Norm{\cdot}_{s, D}$.
In particular, $H^0(D) := L^2(D)$, where $L^2(D)$ is the space of Lebesgue square integrable functions over~$D$
and $H_0^1(D)$ is the closure of $C_0^{\infty}(D)$ in the $H^1(D)$ norm.
Whenever~$s$ is a fractional or negative number, the Sobolev space $H^{s}(D)$ is defined by means of interpolation and duality.
The Sobolev spaces on~$\partial D$ are defined analogously
and denoted by~$H^s(\partial D)$, $s < 1$.

As common in space-time variational problems, we shall also use Bochner spaces of functions mapping a time interval $(a, b)$ into a Banach space $(Z, \Norm{\cdot}_Z)$, which
we denote by $H^s(a, b; Z)$, $s \in \N$.\medskip

\paragraph*{Structure of the paper}
In the remainder of this introduction,
we introduce the model problem (Subsection~\ref{subsection:model-problem}),
and regular sequences of meshes (Subsection~\ref{subsection:meshes}).
The new space-time VEM method is presented in Section~\ref{section:VEM}. Section~\ref{section:well-posedness} is dedicated to the well-posedness of the method,
while in Section~\ref{section:convergence-analysis} we present
an \emph{a priori} error analysis and prove quasi-optimal estimates for the $h$-version of the method.
We conclude this work with some numerical experiments in Section~\ref{section:numerics}
and some concluding remarks in Section~\ref{section:conclusion}.

\subsection{The model problem and its weak formulation} \label{subsection:model-problem}

We are interested in the approximation of
solutions to heat equation initial-boundary value problems
on the space-time domain~$\QT := \Omega \times (0, T)$,
where~$\Omega \subset \Rbb^d \ (d = 1,\ 2,\ 3)$ and~$T > 0$ 
denote a bounded Lipschitz spatial domain and a final time,
respectively.

Let~$f:\QT\to\Rbb$
denote the prescribed right-hand side.
We consider a positive constant volumetric heat capacity~$\cH$
and a positive constant scalar-valued thermal 
conductivity~$\nu$.
The strong formulation of the initial-boundary value problem for the heat equation reads:
Find a function~$u:\QT \to \Rbb$ (temperature) such that
\begin{equation} \label{continuous-strong}
\begin{cases}
\cH \dpt{u} - \nu \Deltax u = f \quad\text{ in } \QT;\\
u = 0  \quad\text{ on } \Omega \times \left\{0\right\};\qquad 
u = 0 \quad\text{ on } \partial \Omega \times (0, T).
\end{cases}
\end{equation}
See Remark~\ref{remark:initial-condition} below for more general initial and boundary conditions.

Introduce the function spaces
\begin{equation}\label{X-Y}
Y := L^2\left(0, T; H_0^1(\Omega)\right),
\quad
X := \left\{v \in Y \cap H^1\left(0, T; H^{-1}(\Omega)\right) 
        \mid v=0\ \text{in}\ \Omega\times\{0\}
        \right\},
\end{equation}
endowed with the norms
\begin{equation*}
\Norm{v}_{Y}^2 := 
\Norm{\nu^{1/2}\nablax v}_{0,\QT}^2,
\qquad\qquad
\Norm{v}_{X}^2 := 
\Norm{\cH\dpt{v}}_{L^2(0,T;H^{-1}(\Omega))}^2
                  + \Norm{v}_{Y}^2,
\end{equation*}
respectively. Here, we have used the following definition:
\begin{equation*}
\text{for any $\phi$ in ${L^2}(0,T;H^{-1}(\Omega))$,}\qquad
\Norm{\phi}_{L^2(0,T;H^{-1}(\Omega))}
:= \sup_{0\ne v \in Y} \frac{
\int_0^T\langle \phi,v\rangle\dt
}{\Norm{v}_{Y}},
\end{equation*}
where $\langle\cdot,\cdot\rangle$ denotes the duality between $H^1_0(\Omega)$ and $H^{-1}(\Omega)$.
Next, we define the space-time bilinear form
$b(\cdot,\cdot):X\times Y \to \Rbb$ as
\begin{equation} \label{continuous:bf:global}
b(u, v) := \int_0^T  \left(\langle\cH\dpt u, v\rangle
                                + 
\int_\Omega\nu \nablax u \cdot \nablax v\ \dx\right) \dt.
\end{equation}
The weak formulation of~\eqref{continuous-strong},
see, e.g., \cite{Dautray-Lions:1992}, reads as follows:
\begin{equation} \label{continuous-weak}
\text{Find } u \in X \text{ such that }
b(u,v)=
\int_0^T\langle f,v\rangle\dt 
\quad \quad \forall v \in Y.
\end{equation}
The following well-posedness result is valid;
see e.g., \cite[Cor.~2.3]{steinbach2015CMAM}.

\begin{proposition} \label{proposition:well-posedness-continuous}
If~$f$ belongs to~$L^2(0,T; H^{-1}(\Omega))$,
then the variational formulation~\eqref{continuous-weak} is well posed
with the \textit{a priori} bound
\begin{equation*}
\Norm{u}_{X}
\le 2\sqrt2 \Norm{f}_{L^2(0,T;H^{-1}(\Omega))}.
\end{equation*}
\end{proposition}

\begin{remark}[Inhomogeneous initial and boundary conditions] \label{remark:initial-condition}
Given~$(f,u_0)$ in $L^2(0,T;H^{-1}(\Omega))\times L^2(\Omega)$,
consider the following problem:
find $u\in Y\cap H^1(0,T;H^{-1}(\Omega))$ such that
\begin{equation}\label{eq:Non0IC}
\begin{cases}
\displaystyle{\int_0^T  \left(\langle\cH\dpt u, v\rangle 
+ \int_\Omega\nu \nablax u \cdot \nablax v\ \dx\right) \dt=\int_0^T\langle f,v\rangle\dt}
& \forall v \in L^2(0,T;H^{-1}(\Omega)) \\[0.2cm]
\displaystyle{\int_\Omega u(\cdot,0) w\,dx=\int_\Omega u_0 w\,dx}
& \forall w \in L^2(\Omega).
\end{cases}
\end{equation}
The well-posedness of problem~\eqref{eq:Non0IC}
is discussed, e.g., in~\cite[Sect.~5]{Schwab_Stevenson_2009}.

The case of inhomogeneous Dirichlet boundary conditions~$u=g$ on $\partial\Omega\times(0,T)$
can be dealt with assuming~$g$ in~$H^1(0,T;H^{1/2}(\partial\Omega))$. 
Denote by~$G:\QT\to\Rbb$ the solution to the family
of elliptic problems $-\nu\Deltax G (\cdot,t)=0$ in $\Omega$
with $G(\cdot,t)=g(\cdot,t)$ on $\partial\Omega$ for all $0\le t\le T$.
The function~$G$ belongs to~$H^1(0,T;H^1(\Omega))$,
since $\partial_t G$ solves a similar family of elliptic problems
with boundary data $\partial_t g$
in~$L^2(0,T;H^{1/2}(\partial\Omega))$).

For the case of inhomogeneous initial and boundary conditions, denote by~$w$ the solution to problem~\eqref{eq:Non0IC}
with source term $f-c_H\partial_t G$
and initial condition $u_0-g(\cdot,0)$.
Then, $u=G+w$ solves the inhomogeneous initial-boundary value problem with data $(f,u_0,g)$.
In particular, $u$ belongs to~$L^2(0,T;H^1(\Omega))\cap H^1(0,T; H^{-1}(\Omega))$.
\eremk
\end{remark}

\subsection{Mesh assumptions} \label{subsection:meshes}
For the sake of presentation, we stick to tensor-product-in-time meshes.
We postpone possible generalization to Subsection~\ref{subsection:more-general-meshes} below,
which are important, e.g., for an adaptive version of the scheme.

We consider a sequence of polytopic meshes~$\{\taun\}_h$ of~$\QT$.
We require that
\begin{itemize}
\item[(\textbf{G1})]
the space domain~$\Omega$ is split into a mesh~$\taunx$ of
non-overlapping $d$-dimensional polytopes
with straight facets;
the time interval~$(0,T)$ is split into~$N$ subintervals $\In:=(\tnmo,\tn)$ with knots~$0=t_0 < t_1 < \dots < t_N =T$;
each element~$\E$ in~$\taun$ can be written as $\Ex \times \In$, for
some $\Ex$ in~$\taunx$ and~$1\le n\le N$.
\end{itemize}
Essentially, assumption (\textbf{G1}) states that
(\emph{i}) each element is the tensor-product of a~$d$-dimensional polytope with a time interval;
(\emph{ii}) each element belongs to a time slab out of the~$N$ identified by the partition~$\{ t_n \}_{n=0}^N$;
(\emph{iii}) each time slab is partitioned by the same space mesh;
(\emph{iv}) all elements within the same time slab have the same extent in time.
\medskip

Given an element~$\E$ in~$\taun$, $\E=\Ex\times\In$,
we denote its diameter by~$\hE$ and
the diameter of~$\Ex$ by~$\hEx$,
and set $\htime:=\tn-\tnmo$.
We let~$\h := \max_{\E \in \taun} \hE$
and~$\hx := \max_{\E \in \taunx} \hEx$.
Furthermore, the set of all $(d-1)$-dimensional facets of~$\Ex$ is denoted by~$\FcalEx$, and
for any~$\Fx\in \FcalEx$ we define
\begin{equation*}
\hFx := 
\begin{cases}
\min\{\hEx, h_{\widetilde{\E}_{\x}}\} & \text{if}~\Fx = \Ex \cap \widetilde{\E}_{\x}\ \text{for some }\widetilde{\E}_{\x} \in \taunx, \\
\hEx & \text{if}~\Fx \subset \partial \Omega.
\end{cases}
\end{equation*}
For~$d=1$, $\Fx$ is a point and~$\int_{\Fx} v(x,t) \dS$
is equal to~$v(\Fx,t)$.
For each spatial facet~$\Fx $ in~$\FcalEx$, 
we introduce the time-like facet $\F:=\Fx \times \In$;
we collect all these time-like facets into the set~$\FcalE$.

We fix one of the two unit normal $d$-dimensional vectors associated with~$\Fx$
and denote it by~$\nbfFx$.
For~$d\ge1$, each time-like facet~$\F=\Fx\times\In$
lies in a $d$-dimensional hyperplane
with unit normal vector~$\nbfF := (\nbfFx, 0)$.

Next, we require further assumptions on the spatial mesh:
there exists $\gamma>0$ independent of the meshsize such that
\begin{itemize}
\item[(\textbf{G2})] each spatial element~$\Ex$ in~$\taunx$ is star-shaped with respect to a ball of radius~$\rho_{\Ex}$ with $\hEx \leq \gamma \rho_{\Ex}$
and the number of ($d-1$)-dimensional facets of~$\Ex$ is uniformly bounded
with respect to the meshsize;
\item[(\textbf{G3})] 
given two neighbouring elements~$\Ex$ and $\widetilde{\E}_{\x}$
of~$\taunx$,
we have that $\gamma^{-1} \h_{\widetilde{\E}_{\x}}
\le \hEx 
\le \gamma \h_{\widetilde{\E}_{\x}}$.
\end{itemize}
For a given space-time element $K \subset \R^{d + 1}$ and any space-like or
time-like facet $F \subset \partial K$,
we denote the space of polynomials of total degree at most~$p \in \mathbb N$ on~$K$ and~$F$
by~$\Pp{p}{K}$ and~$\Pp{p}{F}$, respectively.
For a given time interval $I$,
$\Pp{p}{I}$ denotes the space of polynomials in~$I$ 
of total degree at most~$\p$ in~$\mathbb N$.

Henceforth, for a positive natural number $k$, we define the spaces of broken $H^k$
functions over~$\taunx$ and $\taun$, respectively, by
\begin{equation*}
\begin{split}
H^k(\taunx) & :=
\left\{
v\in L^2(\QT) \mid  v_{|\Ex} \in H^k(\Ex) \ \; \forall \Ex \in \taunx
\right\}; \\
H^k(\taun) & :=
\left\{
v\in L^2(\QT) \mid \ v_{|\E} \in H^k(\E) \ \; \forall \E \in \taun
\right\}.
\end{split}
\end{equation*}
We denote the broken Sobolev~$k$ seminorm on~$\taun$ by~$\SemiNorm{\cdot}_{k, \taun}$
and the space of piecewise polynomials of degree at most~$\ell$ in~$\Nbb$ on~$\taun$ by~$\Scalelltaun$.

\section{The virtual element method} \label{section:VEM}
In this section, we introduce a VEM for the discretization of problem~\eqref{continuous-weak}
based on the regular meshes introduced in Subsection~\ref{subsection:meshes}.
First, local VE spaces are introduced in Subsection~\ref{subsection:VE-spaces} together with their degrees of freedom ($\DoFs$).
Based on the choice of such $\DoFs$, in Subsection~\ref{subsection:polynomial-projection},
we show that we can compute different polynomial projections
of the VE functions.
Such polynomial projections are instrumental in the design of the global VE spaces;
see Subsection~\ref{subsection:global-spaces}.
Likewise, in Subsection~\ref{subsection:discrete-bf}, we design computable discrete bilinear forms
and require sufficient properties that will allow us to prove the well posedness of the scheme,
introduced in Subsection~\ref{subsection:method},
as well as convergence estimates.
Finally, in Subsection~\ref{subsection:more-general-meshes}, we present more general types of meshes that can be used, e.g., in an adaptive framework.

\subsection{Local virtual element spaces} \label{subsection:VE-spaces}
We present a VE discretization of the infinite
dimensional spaces~$X$ and~$Y$ introduced in~\eqref{X-Y}.

Given an approximation degree~$p \in \Nbb$
and an element~$\E=\Ex\times I_n$ in~$\taun$,
we define the following local VE spaces:
\begin{equation} \label{local-VE-space}
\begin{split}
\VhE := 
& \Big\{v_h \in L^2(\E) \mid  
            \ctildeHE \dpt \vh - \nutildeE \Deltax \vh  \in \Pp{p-1}{\E},
            \quad \vh{}_{|\Ex \times \{\tnmo\}} \in \Pp{p}{\Ex};\\[0.2em]
& \qquad\qquad \nbfFx \cdot \nablax \vh{}_{|\F} \in \Pp{\p}{\F} \, \forall \F:=\Fx \times \In\text{ with } \Fx \in \FcalEx \Big\},
\end{split}
\end{equation}
where~$\ctildeHE := \htime$ and~$\nutildeE := \hEx^2$.

The space~$\VhE$ contains~$\Pp{p}{\E}$.
The degree~$\p$ in the Neumann boundary conditions is not necessary for this inclusion to be valid,
as~$\p - 1$ would be sufficient.
Nevertheless, the degree~$\p$ is crucial in the proof of the Poincar\'e-type inequality in Proposition~\ref{proposition:YE-seminorm-Poincare} below.

\begin{remark} \label{remark:Lions-Magenes}
Functions in~$\VhE$ solve a heat equation problem with polynomial source,
initial condition, and Neumann boundary conditions.
For this reason, $\VhE \subset L^2(\In; H^1(\Ex))$; see~\cite[Thm.~$4.1$ and Sect.~$4.7.2$ in Ch.~$3$]{Lions-Magenes:1972}
with standard modifications to deal with the inhomogeneous Neumann data.
\eremk
\end{remark}

\begin{remark} \label{remark:scaling-local-spaces}
As opposite to the standard VE setting~\cite{Beirao-Brezzi-Cangiani-Manzini-Marini-Russo:2013},
in definition~\eqref{local-VE-space}
we consider solutions to local problems involving
some scaling factors ($\ctildeHE$ and~$\nutildeE$).
The reason 
is that these local problems
involve differential operators of different orders.
By using a scaling argument and mapping the element~$\E$
into a ``reference'' element~$\Ehat = \widehat{I} \times \widehat{K}_{\widehat{\x}}$,
with $|\widehat{I}_n| = \text{diam}(\widehat{\E}_{\widehat{\x}}) = 1$,
the resulting reference space consists of solutions to a heat equation with both coefficients equal to~$1$.
This allows us to use equivalence of norms results when proving the stability of the scheme.
\eremk
\end{remark}

Let~$\{ \malphaE \}_{\alpha=1}^{\dim(\Pbb_{\p-1}(\E))}$,
$\{ \malphaF{} \}_{\beta=1}^{\dim(\Pbb_{\p}(\F))}$,
and~$\{ \malphaEx \}_{\gamma=1}^{\dim(\Pbb_{\p}(\Ex))}$
be any bases of~$\Pbb_{p-1}(\E)$, $\Pbb_{\p}(\F)$, and~$\Pbb_{\p}(\Ex)$.
We introduce the following set of linear functionals on $\VhE$:\medskip

\begin{itemize}
\item the \emph{bulk} moments
\begin{equation} \label{bulk-dofs}
\frac{1}{\vert \E \vert}
\int_{\In} \int_{\Ex} \vh \ \malphaE \dx \dt
\qquad \forall \alpha=1,\dots, \dim(\Pbb_{p-1}(\E));
\end{equation}
\item for all space-time facets~$\Fx \times \In = \F \in \FcalE$, the \emph{time-like} moments
\begin{equation} \label{vertical-dofs}
\frac{1}{\vert\F\vert}
\int_{\In}\int_{\Fx} \vh \ \malphaF \dS\ \dt
\qquad \forall \beta=1,\dots,\dim(\Pbb_{\p}(\F));
\end{equation}
\item the \emph{space-like} moments
\begin{equation} \label{bottom-dofs} 
\frac{1}{\vert\Ex\vert} 
\int_{\Ex} \vh(\cdot, \tnmo) \malphaEx \dx
\qquad \forall \gamma=1,\dots,\dim(\Pbb_{p}(\Ex)).
\end{equation}
\end{itemize}
Since functions $\vh\in\VhE$ are polynomials at time $t_{n-1}$,
then the integrals in~\eqref{bottom-dofs} are well defined. 
Moreover, the inclusion $\VhE\subset L^2(\In; H^1(\Ex))$, see Remark~\ref{remark:Lions-Magenes}, 
implies that the integrals in~\eqref{bulk-dofs} and~\eqref{vertical-dofs} are well defined as well.

We introduce the number of the functionals in~\eqref{bulk-dofs}--\eqref{bottom-dofs} as
\begin{equation*}
\# \DoFs
:= \dim(\Pbb_{\p-1}(\E))
+ \sum_{\F \in \FcalE}  \dim(\Pbb_{\p}(\F))
+ \dim(\Pbb_{\p}(\Ex)).
\end{equation*}
In the following lemma, we prove that the linear functionals~\eqref{bulk-dofs}--\eqref{bottom-dofs}
actually define a set of $\DoFs$ for~$\VhE$.
For convenience, we
denote the set of these
linear functionals by~$\{ \DoF_i \}_{i=1}^{\# \DoFs}$.

\begin{lemma} \label{lemma:unisolvence}
The linear functionals~\eqref{bulk-dofs}--\eqref{bottom-dofs} are a set of unisolvent $\DoFs$ for the space~$\VhE$.
\end{lemma}
\begin{proof}
Since the right-hand side and the initial and Neumann boundary conditions in~\eqref{local-VE-space} are independent of each other,
the dimension of~$\VhE$ is equal to the number of the linear functionals~\eqref{bulk-dofs}--\eqref{bottom-dofs}.
Thus, it suffices to prove that the set of these linear functionals is unisolvent.
In other words, we prove that, whenever~$\vh\in\VhE$ satisfies~$\DoF_i(\vh)=0$
for all $i=1,\dots,\# \DoFs$, then~$\vh=0$.

Thanks to the definition of the $\DoFs$~\eqref{bulk-dofs} and~\eqref{vertical-dofs}, we have
\begin{equation*}
\begin{split}
0
& = \int_{\In} \int_{\Ex} \vh \underbrace{\left( \ctildeHE \dpt{\vh} - \nutildeE \Deltax \vh \right)}_{\in \Pp{p-1}{K}} \dx\ \dt  
    + \nutildeE \sum_{\Fx \in \FcalEx} \int_{\In} \int_{\Fx} \vh \underbrace{\nbfFx \cdot \nablax \vh}_{\in \Pp{\p}{\F}} \dS\ \dt  \\
& = \frac{\ctildeHE}{2} \left( \Norm{\vh(\cdot, \tn)}_{0,\Ex}^2 - \Norm{\vh(\cdot, \tnmo)}_{0,\Ex}^2  \right)  
    + \nutildeE \Norm{\nablax \vh}_{0,\E}^2.
\end{split}
\end{equation*}
Furthermore, using the definition of the $\DoFs$~\eqref{bottom-dofs},
we have $\Norm{\vh(\cdot, \tnmo)}_{0,\Ex}^2=0$ and deduce
\begin{equation*}
\nutildeE \Norm{\nablax \vh}_{0,\E}^2 = 0 \quad\Rightarrow\quad \nablax \vh =0 \text{ in } \E
\quad\Rightarrow\quad \vh=\vh(t).
\end{equation*}
From the definition of the space~$\VhE$,
this implies that~$\partial_t \vh$ belongs to~$\Pbb_{\p-1}(I_n)$;
equivalently, $\vh$ belongs to~$\Pbb_\p(I_n)$.
On the other hand, we know that the moments~\eqref{vertical-dofs} are zero,
in particular when they are taken with respect to monomials up to degree $p$ in time only.
This implies~$\vh=0$.
\end{proof}

\subsection{Polynomial projections} \label{subsection:polynomial-projection}
Functions in the local VE space~$\VhE$ are not known in closed form.
However, if we have at our disposal the $\DoFs$ of a function~$\vh$ in~$\VhE$,
then we can compute projections onto polynomial spaces with given maximum degree.

First, for all~$\E=\Ex \times \In$ in~$\taun$ and~$\varepsilon>0$,
we define the operator
$\PiN:   H^{\frac12+\varepsilon}(\In; L^2(\Ex))    \cap L^2(\In; H^1(\Ex)) \to \Pbb_p(\E)$
as follows: for any
$v$ in~$H^{\frac12+\varepsilon}(\In; L^2(\Ex)) \cap L^2(\In; H^1(\Ex))$,
\begin{subequations} \label{PiN}
\begin{align}
\int_{\In}\int_{\Ex} \nablax \qp^\E \cdot \nablax \left(\PiN{v} - v \right) \ \dx \dt &= 0 \quad \forall \qp^\E \in \Pp{p}{\E}; 
\label{PiN-1} \\
\int_{\In} \int_{\Ex} \qpmo(t) \left(\PiN{v} - v \right) \dx \dt &= 0 \quad \forall \qpmo \in \Pp{p-1}{\In}; \label{PiN-2} \\
 \int_{\Ex} \left(\PiN{v}(\x,\tnmo) - v(\x, \tnmo) \right) \dx &= 0.    \label{PiN-3} 
\end{align}
\end{subequations}
We have~$\VhE \subset L^2(\In;H^1(\Ex))$;
see Remark~\ref{remark:Lions-Magenes}.
This and the fact that functions in~$\VhE$ restricted to the time~$\tnmo$ are polynomials
entail that we can define~$\PiN v$ also for~$v$ in~$\VhE$.

\begin{lemma} \label{lemma:PiN}
The operator~$\PiN$ is well defined. Moreover, for any~$\vh$ in~$\VhE$,
$\PiN{\vh}$ is computable via the $\DoFs$~\eqref{bulk-dofs}--\eqref{bottom-dofs}.
\end{lemma}
\begin{proof}

In order to prove that~$\PiN$ is well defined,
we need to show that the number of (linear) conditions
in~\eqref{PiN-1}--\eqref{PiN-3} is equal to~$\dim(\Pbb_p(\E))$. 
As~\eqref{PiN-1} is void for all $\qp^\E\in\Pp{p}{\In}$, 
we have that the number of conditions in~\eqref{PiN-1}--\eqref{PiN-3}
is equal to~$\dim(\Pbb_p(\E))$.
We only need to show that they are linearly independent.

To this aim, assume that~$v=0$.
Conditions~\eqref{PiN-1} imply that~$\nablax \PiN v =0$, i.e.,
$\PiN v$ belongs to~$\Pbb_{p}(\In)$.
Let~$L_p(\cdot)$ be the Legendre polynomial of degree~$p$ over~$[-1,1]$.
Using conditions~\eqref{PiN-2}, we deduce that
there exists a constant~$c$ such that
\begin{equation*}
\PiN v = c L_p \left( \frac{2t-\tnmo-\tn}{\tn-\tnmo} \right).
\end{equation*}
Since condition~\eqref{PiN-3} entails~$\PiN v(\cdot,\tnmo) = 0$ and~$L_p(-1)\ne 0$,
we deduce~$c=0$, whence~$\PiN v=0$.
Therefore, the conditions are linearly independent and so~$\PiN$ is well defined.

As for the computability of~$\PiN\vh$ for~$\vh$ in~$\VhE$,
conditions~\eqref{PiN-1} and~\eqref{PiN-2} are available via the bulk moments~\eqref{bulk-dofs} (up to order~$\p-2$)
and the time-like moments~\eqref{vertical-dofs} (up to order~$\p-1$); 
condition~\eqref{PiN-3} is available via the lowest-order space-like moment in~\eqref{bottom-dofs}.
\end{proof}

Next, for all~$\E$ in~$\taun$,
we define the operator $\Pistar: 
\mathcal{C}^0(\In; L^2(\Ex)) \to \Pbb_p(\E)$
as follows:
for any~$v$ in~$
\mathcal{C}^0(\In; L^2(\Ex))$,
\begin{subequations} \label{Pistar}
\begin{align}
\int_{\In} \int_{\Ex} \qpmo^\E \left(\Pistar{v} - v \right) \dx \dt & = 0\quad \forall \qpmo^\E \in \Pp{p-1}{\E};\label{Pistar-1} \\
\int_{\Ex} \qp^{\Ex} \left(\Pistar{v}(\x, \tnmo)  - v(\x, \tnmo) \right) \dx & = 0 \quad \forall \qp^{\Ex} \in \Pp{p}{\Ex} \label{Pistar-2}.
\end{align}
\end{subequations}
Again, we have~$\VhE \subset L^2(\In;H^1(\Ex))$;
see Remark~\ref{remark:Lions-Magenes}.
This and the fact that functions in~$\VhE$ restricted to the time~$\tnmo$ are polynomials
entail that we can define~$\PiN v$ also for~$v$ in~$\VhE$.

\begin{lemma} \label{lemma:Pistar}
The operator~$\Pistar$ is well defined. 
Moreover, for any~$\vh$ in~$\VhE$,
$\Pistar{\vh}$ is computable via the $\DoFs$~\eqref{bulk-dofs}--\eqref{bottom-dofs}.
\end{lemma}
\begin{proof}
As in the proof of Lemma~\ref{lemma:PiN}, 
we
observe that the number of (linear) conditions
in \eqref{Pistar-1}--\eqref{Pistar-2}
is equal to
$\dim(\Pbb_p(\E))$.
Thus, it suffices to show that they are linearly independent.

Assume that $v=0$. Then,
taking $\qp^{\Ex} = \Pistar v(\x, \tnmo)$ in \eqref{Pistar-2}, we get $\Pistar v(\x, \tnmo) = 0$.
On the other hand, taking $\qpmo^\E = \dpt{\Pistar v}$ in \eqref{Pistar-1}, we get
\begin{equation*}
0 
= \frac{1}{2} \left(\Norm{\Pistar v(\cdot, \tn)}_{0,\Ex}^2 -  \Norm{\Pistar v(\cdot, \tnmo)}_{0,\Ex}^2  \right)
= \frac{1}{2} \Norm{\Pistar v(\cdot, \tn)}_{0,\Ex}^2.
\end{equation*}
In addition, we observe that
\begin{equation*}
\Norm{\dpt{\Pistar v}}_{0,\E}^2 
= \int_{\Ex} \Pistar v(\x, t) \,\dpt{\Pistar v}(\x, t) \dx \Big|_{t = \tnmo}^{\tn} - \int_{\In} \int_{\Ex} \Pistar v 
\underbrace{\dptt{\Pistar v}}_{\in \Pp{p - 2}{\E}} \dx \dt = 0.
\end{equation*}
This implies that~$\dpt{\Pistar v}=0$, which, together with~$\Pistar v(\cdot, \tn)=0$,
gives that $\Pistar v=0$.
Therefore, the conditions are linearly independent
and so~$\Pistar$ is well defined.

As for the computability of $\Pistar\vh$ for $\vh\in\VhE$, conditions~\eqref{Pistar-1} are available via the bulk $\DoFs$~\eqref{bulk-dofs}, and
conditions~\eqref{Pistar-2} are at disposal via the bottom space-like
$\DoFs$~\eqref{bottom-dofs}.
\end{proof}

We introduce other polynomial projectors:
for all~$\E$ in~$\taun$ and~$v$ in~$L^2(\E)$,
\begin{equation*}
\PizE: L^2(\E) \to \Pbb_{\p-1}(\E),\qquad
(\qpmo^\E, v-\PizE v)_{0,\E} = 0 \quad
\forall \qpmo^\E \in \Pbb_{\p-1}(\E);
\end{equation*}
for each temporal interval $\In$ and $v \in L^2(\In)$, 
\begin{equation*}
\PizIn: L^2(\In) \to \Pbb_{\p-1}(\In),\qquad
(\qpmo^{\In}, v-\PizIn v)_{0,\In} = 0 \quad
\forall \qpmo^{\In} \in \Pbb_{\p-1}(\In);
\end{equation*}
for each spatial element $\Ex$ and $v \in L^2(\Ex)$, 
\begin{equation*}
\PizoEx: L^2(\Ex) \to \R,\qquad
(q_0, v-\PizoEx v)_{0,\Ex} = 0 \quad
\forall q_0 \in \R;
\end{equation*}
for all time-like facet~$\F$ and~$v$ in~$L^2(\F)$,
\begin{equation*}
\PizF: L^2(\F) \to \Pbb_{\p}(\F), \qquad
(\qp^{\F}, v-\PizF v)_{0,\F} = 0 \quad
\forall \qp^{\F} \in \Pbb_{\p}(\F).
\end{equation*}
Given~$\vh$ in~$\VhE$,
the computability of the above projectors applied to $\vh$ follows from the definition of the $\DoFs$~\eqref{bulk-dofs}--\eqref{bottom-dofs}.
The projector~$\PizE$ induces the global piecewise $L^2$ projector~$\Pizpmo$ over~$\taun$.

The following polynomial inverse inequalities are valid.
\begin{lemma}\label{lemma:polynomial-inverse-projectors}
For any $p \in \N$, there exist positive constants~$\cPiN$ and~$\cPistar$ independent of $\htime$ and $\hEx$ such that,
for all~$\qp$ in~$\Pbb_\p(\E)$,
\begin{equation} \label{inverse:PiN}
\begin{split}
& \Norm{\qp}_{0,\E}^2 
+ \hEx^2    \Norm{\nablax\qp}_{0,\E}^2
+ \htime^2  \Norm{\dpt\qp}_{0,\E}^2\\
& \le \cPiN \left(  \hEx^2 \Norm{\nablax \qp}_{0,\E} ^2
            + \Norm{\PizIn \qp}_{0,\E} ^2
            + \htime \Norm{\Pi^{0,\Ex}_0 \qp(\cdot,\tnmo)}_{0,\Ex} ^2 \right)
\end{split}
\end{equation}
and
\begin{equation}  \label{inverse:Pistar}
\begin{split}
& \Norm{\qp}_{0,\E}^2 
+ \hEx^2    \Norm{\nablax\qp}_{0,\E}^2
+ \htime^2  \Norm{\dpt\qp}_{0,\E}^2
 \le \cPistar \left( \Norm{\PizE \qp}_{0,\E} ^2
                      +\htime \Norm{\qp(\cdot,\tnmo)}_{0,\Ex}^2 \right).
\end{split}
\end{equation}
\end{lemma}
\begin{proof}
The assertion follows from the regularity
of the spatial mesh in assumption 
(\textbf{G2}),
the fact that the functionals on the right-hand side of~\eqref{inverse:PiN} and~\eqref{inverse:Pistar}
are norms for~$\Pbb_\p(\E)$,
and the equivalence of norms for spaces of polynomials with fixed maximum degree.
\end{proof}

The presence of the subscripts appearing in the inverse estimate constants~$\cPiN$ and~$\cPistar$
is to remind that the norms on the right-hand side of~\eqref{inverse:PiN} and~\eqref{inverse:Pistar}
are induced by the definition of the operators~$\PiN$ and~$\Pistar$.
\medskip

\subsection{Global virtual element spaces} \label{subsection:global-spaces}
We construct the global VE spaces in a nonconforming fashion.
To this aim, we introduce a jump operator on the time-like facets.
Each internal time-like facet~$\F$ is shared
by two elements~$\E_1$ and~$\E_2$
with outward pointing unit normal vectors~$\nbf_{\E_1}$ and~$\nbf_{\E_2}$,
whereas each boundary time-like facet belongs to the boundary of a
single element~$\E_3$ with outward pointing unit normal vector~$\nbf_{\E_3}$.
We denote the $d$-dimensional vector containing the spatial components
of the restriction of~$\nbf_{\E_j}$ to the time-like facet~$\F$ by~$\nbf_{\E_j}^F$.
Then, the 
normal
jump on each time-like facet~$\F$ is defined as
\begin{equation} \label{jump-vector}
\jump{v}_{\F}:=
\begin{cases}
v_{|\E_1} \nbf_{\E_1}^\F + v_{|\E_2} \nbf_{\E_2}^\F & \text{if } \F \text{ is an internal face};\\
v_{|\E_3} \nbf_{\E_3}^\F                            & \text{if } \F \text{ is a  boundary face}.
\end{cases}     
\end{equation}
On each time slab~$\In$, 
we introduce the nonconforming
Sobolev space of order~$p$ associated with the mesh~$\taunx$:
\begin{equation} \label{nonconforming-condition}
H^{1, nc}(\taunx; \In) :=
\Big\{
v\in L_2\left(\In; H^1(\taunx)\right) 
\Big|
\int_{\In} \int_{\Fx} \qp^\F \jump{v}_{\F} \cdot \nbfFx \dS\ \dt = 0 \quad \forall \qp^\F \in \Pp{\p}{\F} \Big\}.
\end{equation}
This allows us to define the VE discretization~$\Yh$ of
the space~$Y$ in~\eqref{X-Y}
as the space of functions that are possibly discontinuous in time across space-like facets and nonconforming as above in space:
\begin{equation*}
\begin{split}
\Yh 
:= \Big\{ \vh \in L^2(\QT)  \mid 
\vh{}_{|\E} \in \Vh(\E) \ \forall\E \in \taun,\
\vh{}_{|\taunx \times \In} \in H^{1, nc}(\taunx; \In) \ \forall n=1,\dots,N \Big\}.
\end{split}
\end{equation*}
The functions in the space~$X$ in~\eqref{X-Y} are continuous in time, namely,
\begin{equation}\label{eq:Xincl}
X \hookrightarrow \mathcal{C}^0([0, T]; L^2(\Omega));
\end{equation}
see e.g,~\cite[Thm.~25.5]{Wloka:1987}.
Nevertheless, we discretize it with $\Yh$ as well,
and impose the time continuity weakly through upwinding.
As functions in the local VE space~$\VhE$ are not known at the local final time~$\tn$,
the upwind fluxes are defined in terms of the traces of their polynomial projections~$\Pistar$; see~\eqref{eq:upwind} below.

\begin{remark}\label{rem:Xhspace1}
Due to the choice of the $\DoFs$,
one cannot define a continuous-in-time discretization of~$X$ with the local spaces $\VhE$.
If this were possible, then each VE function on~$\E=\Ex\times\In$
would be a polynomial of degree~$p$ at the local final time~$\tn$.
For general choices of the right-hand side, initial condition, and boundary conditions in~\eqref{local-VE-space}, this cannot be true.
\eremk
\end{remark}


\subsection{Discrete bilinear forms} \label{subsection:discrete-bf}
On each element~$\E$, define the local continuous bilinear form in $\VhE \times \VhE$ and seminorm
\begin{equation*}
\aE(\uh,\vh) := \nu (\nablax \uh, \nablax \vh)_{0,\E},
\qquad \qquad
\SemiNorm{\vh}_{\YE}^2 := \aE(\vh, \vh).
\end{equation*}
Next, we prove a local Poincar\'e-type inequality.
\begin{proposition} \label{proposition:YE-seminorm-Poincare}
If~$\vh$ belongs to~$\VhE$, $\E=\Ex\times\In$,
then~$\SemiNorm{\vh}_{\YE}=0$ if and only if~$\vh=\vh(t)$ belongs to~$\Pbb_\p(\In)$.
Moreover, there exists a positive constant~$\CPE$ independent of~$\htime$ and~$\hEx$ such that
\begin{equation} \label{particular-Poincare}
\inf_{\qpt \in \Pbb_\p(\In)} \Norm{\vh-\qpt}_{0,\E}
\le \CPE \hEx \Norm{\nablax \vh}_{0,\E}
\qquad \forall \vh \in \VhE.
\end{equation}
\end{proposition}
\begin{proof}
If~$\vh$ belongs to~$\VhE$ with~$\Norm{\nablax \vh}_{\YE}=0$, then~$\vh=\vh(t)$.
The definition of~$\VhE$ in~\eqref{local-VE-space} implies that~$\partial_t \vh$ belongs to~$\Pbb_{\p-1}(\In)$
or, equivalently, that $\vh$ belongs to~$\Pbb_\p(\In)$. The converse
is obviously true.

Inequality~\eqref{particular-Poincare} follows from the equivalence of
seminorms with the same kernel on finite dimensional spaces
and the scaling argument in Remark~\ref{remark:scaling-local-spaces}.
\end{proof}

We define $Y(\taun) := L^2(0, T; H^1(\taunx))$ and introduce the global broken seminorms
\begin{equation*}
\begin{split}
& \text{ for almost all } t, \quad \SemiNorm{v(\cdot,t)}_{1,\taunx}^2
:= \sum_{\Ex \in \taunx} \Norm{\nablax v(\cdot,t)}_{0,\Ex}^2;\\
& \SemiNorm{v}_{\Ytaun}^2
:= \int_0^T \nu \SemiNorm{v(\cdot,t)}_{1,\taunx}^2 \dt
= \sum_{\E \in \taun} \SemiNorm{v}_{\YE}^2.
\end{split}
\end{equation*}
\begin{proposition} \label{proposition:seminorm-is-norm}
The seminorm~$\SemiNorm{\cdot}_{\Ytaun}$ is a norm in~$\Yh$.
\footnote{In fact, $\SemiNorm{\cdot}_{\Ytaun}$ is a norm on~$Y + \Yh$.
So, for arguments in~$Y+\Yh$,
we shall denote it by~$\Norm{\cdot}_{\Ytaun}$.}
\end{proposition}
\begin{proof}
Given~$\vh$ in~$\Yh$,
we only have to prove that~$\SemiNorm{\vh}_{\Ytaun}=0$ implies~$\vh=0$.
The identity~$\SemiNorm{\vh}_{\Ytaun}=0$ implies that~$\SemiNorm{\vh}_{\YE}=0$ for all elements~$\E=\Ex\times\In$.
Using Proposition~\ref{proposition:YE-seminorm-Poincare},
we deduce that~$\vh{}_{|\E}$ only depends on time and belongs to~$\Pbb_\p(\In)$.
The assertion follows using the spatial nonconformity of the space~$\Yh$,
see~\eqref{nonconforming-condition}, which is up to order~$\p$.
\end{proof}

On each element~$\E$ in~$\taun$, $\E=\Ex\times\In$, let
\[
\SE: [\VhE+L^2(\In;H^1(\Ex))\cap \mathcal{C}^0(\In;L^2(\Ex))]^2
\to \Rbb
\]
be any symmetric positive semidefinite bilinear form
that is computable via the $\DoFs$
and satisfies the following properties:
\medskip

\begin{itemize}
\item for any~$\vh$ in~$\VhE \cap \ker(\PiN)$, we have that
\begin{equation} \label{property-1:stab}
\SE(\vh,\vh) = 0 \qquad \Longrightarrow \qquad \vh =0;
\end{equation}
\item the following bound is valid
with a positive constant~$\ctilde^*>0$ independent of $\htime$, $\hEx$, and~$\E$:
\begin{equation} \label{property-2:stab}
\SE(v,v) 
\le \ctilde^*
\left(\hEx^{-2} \Norm{v}_{0,\E}^2
 + \Norm{\nablax v}_{0,\E}^2
 + \hEx^{-2} \htime^2 \Norm{\dpt v}_{0,\E}^2 \right)
\qquad \forall v \in H^1(\E).
\end{equation}
\end{itemize}
Property~\eqref{property-1:stab} implies that $\SE(\cdot,\cdot)$ induces a norm in $\VhE \cap \ker(\PiN)$.
Another consequence of~\eqref{property-1:stab}
and the scaling argument in Remark~\ref{remark:scaling-local-spaces},
is that there exist two constants~$0 < c_* < c^*$ independent of~$\E$ such that
\begin{equation} \label{stab-prop-0}
    c_* \SemiNorm{\vh}_{\YE}^2
    \le \nu \SE(\vh,\vh)
    \le c^* \SemiNorm{\vh}_{\YE}^2
    \qquad \forall \vh \in \VhE \cap \ker(\PiN).
\end{equation}
In fact, the functional~$\SemiNorm{\cdot}_{\YE}$ is a norm
on~$\VhE \cap \ker(\PiN)$.

We define the discrete counterpart of the local bilinear forms~$\aE(\cdot,\cdot)$:
\begin{equation} \label{ahtildeE}
\ahE(\uh,\vh)
:=  \aE(\PiN\uh, \PiN \vh) 
    + \nu \SE( (I-\PiN)\uh, (I-\PiN)\vh ).
\end{equation}
\begin{lemma}
Property~\eqref{stab-prop-0} implies that there exist two constants~$0<\alpha_*<\alpha^*$ independent of~$\E$ such that
the following local stability bounds are valid:
\begin{equation} \label{stability-bounds:local}
\alpha_* \SemiNorm{\vh}_{\YE}^2
\le \ahE(\vh,\vh)
\le \alpha^* \SemiNorm{\vh}_{\YE}^2
\qquad\forall \vh\in \VhE.
\end{equation}
\end{lemma}
\begin{proof}
We only show the upper bound as the lower bound follows analogously leading to~$\alpha_* := \min(1,c_*)$.
We have
\begin{equation*}
\begin{split}
\ahE(\vh,\vh)
& = \nu \Norm{\nablax \PiN \vh}_{0,\E}^2 
        + \nu \SE( (I-\PiN)\vh, (I-\PiN)\vh )\\
& \le \SemiNorm{\PiN \vh}_{\YE}^2 
      + c^* \SemiNorm{(I-\PiN) \vh}_{\YE}^2
  \le \max(1, c^*) \left(\SemiNorm{\PiN \vh}_{\YE}^2
        +\SemiNorm{(I-\PiN) \vh}_{\YE}^2 \right).
\end{split}
\end{equation*}
Pythagoras' theorem implies
\begin{equation*}
\ahE(\vh,\vh)
\le  \max(1, c^*) \SemiNorm{\vh}_{\YE}^2.
\end{equation*}
This proves the upper bound in~\eqref{stability-bounds:local}
with $\alpha^* = \max(1, c^*)$.
\end{proof}

The global discrete bilinear form associated with the spatial Laplace operator reads
\begin{equation*}
\ah(\uh,\vh) 
:= \sum_{\E \in \taun} \ahE(\uh,\vh)
\qquad \forall \uh,\ \vh \in \Yh.
\end{equation*}
Taking into account Proposition~\ref{proposition:seminorm-is-norm},
an immediate consequence of~\eqref{stability-bounds:local} are
the global stability bounds
\begin{equation} \label{stability-bounds-1:global}
\alpha_* \Norm{\vh}_{\Ytaun}^2
\le \ah(\vh,\vh)
\le \alpha^* \Norm{\vh}_{\Ytaun}^2 \qquad\forall \vh\in \Yh.
\end{equation}
For sufficiently smooth functions,
we have the following upper bounds.

\begin{proposition} \label{proposition:weak-continuity}
For all~$v$ in~$H^1(\taun)$, the following local and global bounds are valid:
for all~$\E$ in~$\taun$,
\begin{equation} \label{continuity:continuous:0}
\ahE(v,v)
\le 3 \max(1,\ctilde^*) \nu
    \left(1 + (1 + \ctrace) \cPiN\right)
    \left( 
        \hEx^{-2} \Norm{v}_{0,\E}^2
        + \Norm{\nablax v}_{0,\E}^2
        + \hEx^{-2} \htime^{2} \Norm{\dpt v}_{0,\E}^2
    \right)
\end{equation}
and
\begin{equation} \label{continuity:continuous}
\ah(v,v)
\!\le\! 3 \max(1,\ctilde^*) \nu \left(1 \!+\! (1 \!+\! \ctrace) \cPiN \right)\!\! \sum_{\E\in\taun} \!
    \left( 
        \hEx^{-2} \Norm{v}_{0,\E}^2
        \!+\! \Norm{\nablax v}_{0,\E}^2
        \!+\! \hEx^{-2} \htime^{2} \Norm{\dpt v}_{0,\E}^2
    \right),
\end{equation}
where~$\ctilde^*$ is the stability constant in~\eqref{property-2:stab},
$\nu$ is the thermal conductivity,
$\ctrace$ is the constant appearing in the elemental trace (in time) inequality,
and~$\cPiN$ is the inverse estimate constant in~\eqref{inverse:Pistar}.
\end{proposition}
\begin{proof}
The stability of the~$\PiN$ projector entails
\begin{equation*}
\aE(\PiN v, \PiN v)
= \nu \Norm{\nablax \PiN v}_{0,\E}^2
\le \nu \Norm{\nablax v}_{0,\E}^2.
\end{equation*}
Using definition~\eqref{ahtildeE}
and bound~\eqref{property-2:stab},
we deduce
\begin{align}
\label{a-bound} 
& \ahE(v,v)
= \aE(\PiN v, \PiN v) + \nu \SE((I-\PiN)v, (I-\PiN)v) \le \max(1,\ctilde^*) \nu 
\\
\nonumber
&  \times \Big( \Norm{\nablax v}_{0,\E}^2
    + \hEx^{-2} \Norm{(I-\PiN)v}_{0,\E}^2 
        + \Norm{\nablax(I-\PiN)v}_{0,\E}^2
        + \hEx^{-2}\htime^{2} \Norm{\dpt(I-\PiN)v}_{0,\E}^2 \Big).
\end{align}
Using the polynomial inverse estimate~\eqref{inverse:PiN} with~$\qp = \PiN v$,
we can write
\begin{equation*}
\begin{split}
& \hEx^{-2} \Norm{\PiN v}_{0,\E}^2
  + \Norm{\nablax \PiN v}_{0,\E}^2
  + \hEx^{-2} \htime^{2} \Norm{\dpt \PiN v}_{0,\E}^2\\
& = \hEx^{-2}
    \left(\Norm{\PiN v}_{0,\E}^2
        + \hEx^2 \Norm{\nablax \PiN v}_{0,\E}^2
         + \htime^{2} \Norm{\dpt \PiN v}_{0,\E}^2 \right)\\
& \le \cPiN \hEx^{-2}
    \left(
        \hEx^2 \Norm{\nablax \PiN v}_{0,\E}^2
        + \Norm{\PizE \PiN v}_{0,\E}^2
        + \htime \Norm{\PizExz \PiN v (\cdot,\tnmo)}_{0,\E}^2
    \right).
\end{split}
\end{equation*}
The definition of~$\PiN$,
and the stability of the $L^2$ and~$\PiN$ projectors entail
\begin{equation*}
\begin{split}
& \hEx^{-2} \Norm{\PiN v}_{0,\E}^2
  + \Norm{\nablax \PiN v}_{0,\E}^2
  + \hEx^{-2} \htime^{2} \Norm{\dpt \PiN v}_{0,\E}^2\\
& \le \cPiN \hEx^{-2}
    \left( 
        \hEx^2 \Norm{\nablax v}_{0,\E}^2
        + \Norm{v}_{0,\E}^2
        + \htime \Norm{v (\cdot,\tnmo)}_{0,\E}^2 \right).
\end{split}
\end{equation*}
Applying a trace inequality along the time variable (with constant~$\ctrace$) on the last term yields
\begin{equation*}
\begin{split}
& \hEx^{-2} \Norm{\PiN v}_{0,\E}^2
  + \Norm{\nablax \PiN v}_{0,\E}^2
  + \hEx^{-2} \htime^{2} \Norm{\dpt \PiN v}_{0,\E}^2\\
& \le (1+\ctrace) \cPiN 
    \left( 
        \hEx^{-2} \Norm{v}_{0,\E}^2
        + \Norm{\nablax v}_{0,\E}^2
        + \hEx^{-2} \htime^{2} \Norm{\dpt v}_{0,\E}^2
    \right).
\end{split}
\end{equation*}
We insert this bound into~\eqref{a-bound} after applying the triangle inequality and obtain \eqref{continuity:continuous:0}. Adding over all elements gives \eqref{continuity:continuous}.
\end{proof}
Here and in the following, for a given~$v$ in~$L^2(Q_T)$,
we shall write 
\begin{equation*}
v^{(n)}:=v{}_{|_{\Omega\times I_{n}}}
\qquad
\text{for all } n=1,\dots,N.
\end{equation*}
For all $\uh,\,\vh$ in~$\Yh$ and $\E$ in~$\taun$,
$\E=\Ex\times\In$, we set
\small{\begin{equation}\label{eq:upwind}
\bhE(\uh,\vh):=\begin{cases}
&\!\!\!\!\! \cH(\dpt \Pistar \uh,\vh)_{0,\E}
    + \ahE(\uh,\vh)
    +\cH\left(\Pistar \uh^{(1)}(\cdot,0),\vh^{(1)}(\cdot,0)\right)_{0,\Ex}
    \hfill \text{if}\ n=1;\\[0.2cm]
&\!\!\!\!\! \cH(\dpt \Pistar \uh,\vh)_{0,\E}
    + \ahE(\uh,\vh)\\[0.2cm]
    & + \cH\left(\Pistar \uh^{(n)}(\cdot,\tnmo) - \Pistar\uh^{(n-1)}(\cdot,\tnmo),\vh^{(n)}(\cdot,\tnmo) \right)_{0,\Ex} \quad
        \hfill \text{if}\ 2\le n\le N.                    
\end{cases}
\end{equation}}\normalsize{}
The bilinear form $\bhE(\cdot, \cdot)$ is computable through the $\DoFs$.
Actually, $\Pistar \uh^{(n)}(\cdot,\tnmo) =  \uh^{(n)}(\cdot,\tnmo)$
for~$1\le n\le N$ by the  definition of~$\Pistar$ in~\eqref{Pistar} and the definition of the local VE spaces.

We define the discrete counterpart of the global bilinear
form~$b(\cdot,\cdot)$
introduced in~\eqref{continuous:bf:global} as follows:
\begin{equation} \label{bh}
\bh(\uh,\vh)
:= \sum_{\E \in \taun} \bhE(\uh, \vh)
\qquad \forall \uh,\ \vh \in \Yh.
\end{equation}
The third terms in the definition of $\bhE(\uh,\vh)$  in~\eqref{eq:upwind}
stand for upwind fluxes for the weak imposition of the zero initial condition for~$n=1$,
or of time continuity for $2\le n\le N$.

\subsubsection{An admissible stabilization}
Consider the following stabilization, for~$\E=\Ex\times\In$:
\begin{equation} \label{explicit:stabilization}
\begin{split}
\SE(\uh,\vh)
& := \hEx^{-2} (\PizE \uh, \PizE \vh)_{0,\E}
   + \hEx^{-1} \sum_{\F \in \FcalE} (\PizF \uh, \PizF \vh)_{0,\F}\\
&\quad + \hEx^{-2}\htime \left(\uh(\cdot,\tnmo), \vh(\cdot,\tnmo) \right)_{0,\Ex}.
\end{split}
\end{equation}
This bilinear form is computable via the $\DoFs$.

\begin{proposition} \label{proposition:properties-explicit-stabilization}
The stabilization in~\eqref{explicit:stabilization} satisfies properties~\eqref{property-1:stab} and~\eqref{property-2:stab}.
\end{proposition}
\begin{proof}
Property~\eqref{property-1:stab} follows from the fact that~$\SE(\vh,\vh)$
involves the squares of~\emph{all} the $\DoFs$.
Furthermore, property~\eqref{property-2:stab} follows from the stability of the~$L^2$ projectors and the trace inequality applied to the time-like and space-like facet terms.
\end{proof}

As pointed out in Subsection~\ref{subsection:discrete-bf}, property~\eqref{property-1:stab}
and the scaling argument in Remark~\ref{remark:scaling-local-spaces}
imply that property~\eqref{stab-prop-0} is satisfied as well.

\subsection{The method} \label{subsection:method}
The VEM that we propose reads as follows:
\begin{equation} \label{VEM}
\text{Find } \uh \in \Yh \text{ such that }
\bh(\uh,\vh) = (f, \Pizpmo \vh)_{0,\QT} \qquad \forall \vh \in \Yh.
\end{equation}
The projector~$\Pizpmo$ makes the right-hand side computable and is~$L^2$ stable,
which is used in the proof of the well posedness of~\eqref{VEM} in Theorem~\ref{theorem:well-posedness} below.

Under assumption (\textbf{G1}),
the method can be solved in a time-marching fashion
by solving the counterpart of~\eqref{VEM} restricted to the time-slab~$\In$, for $n=1,\ldots, N-1$,
and then transmitting the information to the subsequent time-slab~$\Inpo$ through
upwinding.

\subsection{A glimpse on more general meshes} \label{subsection:more-general-meshes}
The reasons why we required assumption (\textbf{G1})
is that it is easier to present the construction of the VE spaces.
We refer to Figure~\ref{figure:meshes} (a) for an example of an admissible mesh in the sense of~(\textbf{G1}).
We can weaken this assumption along two different avenues:
we can allow for
\begin{itemize}
\item nonmatching time-like facets, see Figure~\ref{figure:meshes} (b);
\item nonmatching space-like facets, see Figure~\ref{figure:meshes} (c).
\end{itemize}

\begin{figure}[H]
\centering
\begin{minipage}{0.30\textwidth}
\begin{center}
\begin{tikzpicture}[scale=.7]
\draw[black, ->] (-0.5,0) -- (5,0);
\draw[black, ->] (0,-0.5) -- (0,5);
\draw (5,-.2) node[black, left] {\small x}; 
\draw (-.1,4.5) node[black, left] {\small t}; 
\draw[black, very thick, -] (0,0) -- (4,0) -- (4,4) -- (0,4) -- (0,0);
\draw[black, very thick, -] (1,0) -- (1,4);
\draw[black, very thick, -] (2,0) -- (2,4);
\draw[black, very thick, -] (3,0) -- (3,4);
\draw[black, very thick, -] (0,1) -- (4,1);
\draw[black, very thick, -] (0,2) -- (4,2);
\draw[black, very thick, -] (0,3) -- (4,3);
\end{tikzpicture}
\end{center}
\end{minipage}
\begin{minipage}{0.30\textwidth}
\begin{center}
\begin{tikzpicture}[scale=.7]
\draw[black, ->] (-0.5,0) -- (5,0);
\draw[black, ->] (0,-0.5) -- (0,5);
\draw (5,-.2) node[black, left] {\small x}; 
\draw (-.1,4.5) node[black, left] {\small t}; 
\draw[black, very thick, -] (0,0) -- (4,0) -- (4,4) -- (0,4) -- (0,0);
\draw[black, very thick, -] (1,0) -- (1,4);
\draw[black, very thick, -] (2,0) -- (2,4);
\draw[black, very thick, -] (3,0) -- (3,4);
\draw[black, very thick, -] (0,1) -- (3,1);
\draw[black, very thick, -] (0,2) -- (4,2);
\draw[black, very thick, -] (0,3) -- (3,3);
\end{tikzpicture}\end{center}
\end{minipage}
\begin{minipage}{0.30\textwidth}
\begin{center}
\begin{tikzpicture}[scale=.7]
\draw[black, ->] (-0.5,0) -- (5,0);
\draw[black, ->] (0,-0.5) -- (0,5);
\draw (5,-.2) node[black, left] {\small x}; 
\draw (-.1,4.5) node[black, left] {\small t}; 
\draw[black, very thick, -] (0,0) -- (4,0) -- (4,4) -- (0,4) -- (0,0);
\draw[black, very thick, -] (1,0) -- (1,2);
\draw[black, very thick, -] (1,3) -- (1,4);
\draw[black, very thick, -] (2,0) -- (2,4);
\draw[black, very thick, -] (3,0) -- (3,2);
\draw[black, very thick, -] (3,3) -- (3,4);
\draw[black, very thick, -] (0,1) -- (4,1);
\draw[black, very thick, -] (0,2) -- (4,2);
\draw[black, very thick, -] (0,3) -- (4,3);
\end{tikzpicture}
\end{center}
\end{minipage}
\caption{(a) Left panel: a mesh satisfying assumption (\textbf{G1}).
(b) Central panel: a mesh with nonmatching time-like facets.
(c) Right panel: a mesh with nonmatching space-like facets.}
\label{figure:meshes}
\end{figure}
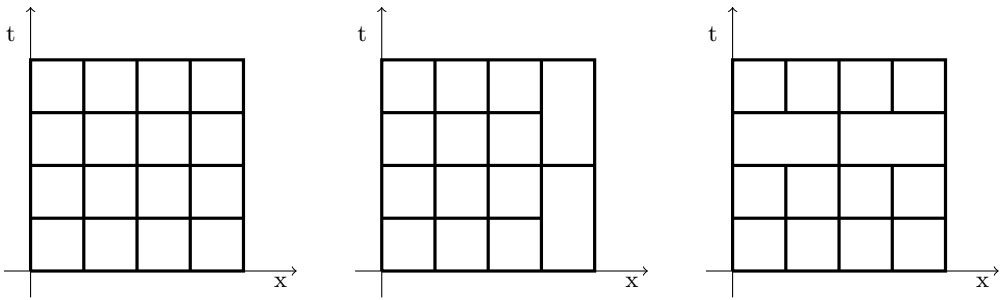

These generalizations are particularly convenient for space-time adaptivity, where nonmatching time-like and space-like facets typically occur.
In the definition of the corresponding VE spaces,
a few modifications would take place.
In the case of nonmatching time-like facets (as in Figure~\ref{figure:meshes}(b)),
VE functions have piecewise polynomial Neumann traces.
Nonmatching space-like facets
(as in Figure~\ref{figure:meshes}(c)) have no effect in the definition of the local VE spaces, see~\cite{Gomez-Mascotto-Perugia:2023} for more details.

\section{Well posedness of the virtual element method} \label{section:well-posedness}
In this section, we prove well posedness of the method in~\eqref{VEM}.
To this aim, we endow the trial space with a suitable norm, which is defined by means
of a VE Newton potential, see Subsection~\ref{subsection:VEM-Newton}.
In Subsection~\ref{subsection:discrete:inf-sup},
we prove a discrete inf-sup condition.
This proof extends that of~\cite[Theorem~$2.1$]{steinbach2015CMAM} to our setting,
where multiple variational crimes have to be taken into account.

Before that, we prove a global Poincar\'e-type inequality for functions in the space~$\Yh$.
\begin{proposition}\label{prop:global-Poincare}
Let assumptions (\textbf{G1})-(\textbf{G2}) be valid.
Then, there exists a positive constant~$\CP$ independent of the mesh size~$\h$ such that
\begin{equation} \label{global-Poincare}
\Norm{\vh}_{0,\QT} \le \CP \Norm{\vh}_{\Ytaun}
\qquad \forall \vh \in \Yh.
\end{equation}
\end{proposition}
\begin{proof}
It suffices to prove the counterpart of~\eqref{global-Poincare} over each time slab~$\In$.
On any time-like face $F$, we define the scalar jump~$\jump{\vh}$ as~$\jump{\vh}_\F \cdot \nbfF$.
\footnote{We have that~$\jump{\cdot}$ is a scalar function \label{footnote-jump}
whereas~$\jump{\cdot}_{\F}$ defined in~\eqref{jump-vector}
is a vector field.}
We start from the spatial Poincar\'e inequality 
in~\cite[Eqn.~(1.3) for~$d\ge2$ and Sect.~8 for~$d=1$]{brenner2003poincare} with constant~$\CPB$
and integrate it in time over the time slab~$\In$:
\begin{equation*}
\Norm{\vh}_{0, \Omega\times \In}^2
\le \CPB \Big( 
\sum_{\E \in \taun, \E \subset \Omega\times\In} 
\Norm{\nablax \vh}_{0,\E}^2
+ \sum_{\Fx \in\Fcalxh} 
 \hFx^{-1} \int_{\In} \Big(\int_{\Fx} \llbracket \vh \rrbracket \dS \Big)^2 \dt \Big).
\end{equation*}
For~$d=1$, the integral over the point~$\Fx$ is the evaluation at~$\Fx$.

To conclude, we have to estimate the second term on the right-hand side.
To this aim, we prove estimates on each time-like facet and then collect them together.
For simplicity, we further assume that~$\F =\Fx\times\In$ is an
internal (time-like) facet shared by two elements~$\E_1={\Ex}_{,1}\times\In$ and~$\E_2={\Ex}_{,2}\times\In$.
The case of a boundary time-like facet can be dealt with similarly.
Recall from the nonconformity of the space~$\Yh$, see~\eqref{nonconforming-condition},
that~$\PizF \llbracket \vh \rrbracket =0$.

Using Jensen's inequality, we write
\begin{equation*}
\hFx^{-1}
\int_{\In} \Big(\int_{\Fx} \llbracket \vh \rrbracket \dS \Big)^2 \dt
\le \int_{\In} \int_{\Fx} \llbracket \vh \rrbracket^2 \dS\ \dt
= \Norm{\llbracket \vh \rrbracket - \PizF \llbracket \vh \rrbracket}_{0,\F}^2.
\end{equation*}
Denote the~$L^2$ projection onto~$\Pbb_{\p}(\In)$
of the restriction of $\vh$ to~$\E_j$
by~$\qptj$, $j=1,2$.
Let~$\qpt$ be defined on~$\E_1 \cup \E_2$ piecewise as~${\qpt}_{|_{\E_j}}=\qptj$, $j=1,2$.
A standard trace inequality in space with constant~$\ctrace$,
and the local Poincar\'e inequality~\eqref{particular-Poincare}
give
\begin{equation*}
\begin{split}
& \hFx^{-1} \int_{\In} \Big(\int_{\Fx} \llbracket \vh \rrbracket \dS \Big)^2 \dt
\le \ctrace \left( \hExo^{-1} \Norm{\vh-\qpt}_{0,\E_1}^2
+ \hExo \Norm{\nablax(\vh-\qpt)}_{0,\E_1}^2  \right.\\
& \qquad\qquad\quad\qquad\quad\qquad\quad\qquad\quad\; \left. + \hExt^{-1}\Norm{\vh-\qpt}_{0,\E_2}^2 
       + \hExt \Norm{\nablax (\vh-\qpt)}_{0,\E_2}^2 \right)\\
& \le 2\ctrace (\max_{\E\in\taun} \CPE)^2
  \left( \hExo \Norm{\nablax(\vh-\qpt)}_{0,\E_1}^2
         + \hExt \Norm{\nablax(\vh-\qpt)}_{0,\E_2}^2 \right) \\
& \le 2\ctrace (\max_{\E\in\taun} \CPE)^2 \max(\hExo, \hExt)
 \sum_{j=1}^2 \Norm{\nablax \vh}_{0,\E_j}^2.
\end{split}
\end{equation*}
Summing over all the time-like facets of the $n$-th time slab
and recalling that the number of ($d-1$)-dimensional facets of each~$\Ex$
is uniformly bounded with respect to the meshsize,
see assumption (\textbf{G2}),
we get the assertion.
\end{proof}

\subsection{A virtual element Newton potential} \label{subsection:VEM-Newton}
We define a VE Newton potential~$\Newtonh: \Scalptaun \to \Yh$ as follows:
for any~$\phih$ in~$\Scalptaun$, $\Newtonh \phih$ in~$\Yh$ solves
\begin{equation}\label{Newton-VEM}
\begin{split}
\ah(\Newtonh \phih,\vh)  
&= \bh(\phih, \vh)-\ah(\phih, \vh)\\
&= \cH \Big[( \partial_t\phih, \vh )_{0,\QT} 
+\left(\phih^{(1)}(\cdot,0),\vh^{(1)}(\cdot,0)\right)_{0,\Omega}\\
&\qquad +\sum_{n=2}^{N}
\left(\phih^{(n)}(\cdot,\tnmo)
-\phih^{(n-1)}(\cdot,\tnmo),\vh^{(n)}(\cdot,\tnmo)\right)_{0,\Omega} \Big]
 \qquad \forall \vh \in \Yh.
\end{split}
\end{equation}
Thanks to the stability bounds~\eqref{stability-bounds-1:global},
the bilinear form~$\ah(\cdot, \cdot)$ is continuous and coercive,
and the continuity in the~$\Ytaun$ norm of the functional
on the right-hand side
of~\eqref{Newton-VEM}
follows from Proposition~\ref{prop:global-Poincare}.
Therefore, the VE Newton potential is well defined.

We introduce the following norm on the sum space $X + \Yh$:
for all~$v$ in $X + \Yh$,
\begin{equation} \label{norm:Xh}
\begin{split}
 \Norm{v}_{\Xhtaun}^2
& := \Norm{v}_{\Ytaun}^2
   + \Norm{\Newtonh (\Pistar v)}_{\Ytaun}^2
   + \frac{\cH}{2}\Big( \Norm{\Pistar v^{(1)}(\cdot,0)}_{0,\Omega}^2 \\
 & \quad 
+\sum_{n=2}^N\Norm{\Pistar v^{(n)}(\cdot,\tnmo) - \Pistar v^{(n-1)}(\cdot,\tnmo)}_{0,\Omega}^2
+\Norm{\Pistar v^{(N)}(\cdot,T)}_{0,\Omega}^2
 \Big).
 \end{split}
\end{equation}
Recalling the embedding $X \hookrightarrow \mathcal{C}^0(0,T; L^2(\Omega))$ in~\eqref{eq:Xincl},
we have that~$\Pistar$ in~\eqref{Pistar} is well-defined for functions in~$X$.
In Section~\ref{section:convergence-analysis} below, we shall present the convergence analysis of the method
with respect to the~$\Norm{\cdot}_{\Xhtaun}$ norm.

\subsection{A discrete inf-sup condition and well posedness of the method} \label{subsection:discrete:inf-sup}
In this section, we prove a discrete inf-sup condition
in the spaces~$(\Yh,\Norm{\cdot}_{\Xhtaun})$ for the trial functions
and~$(\Yh,\Norm{\cdot}_{\Ytaun})$ for the test functions.
\begin{proposition} \label{proposition:discrete:inf-sup}
There exists a positive constant~$\gammaI$ independent of~$\taun$
such that
\begin{equation} \label{discrete:inf-sup}
\sup_{0\ne \vh \in \Yh}
\frac{\bh(\uh,\vh)}{\Norm{\vh}_{\Ytaun}}
\ge
\gammaI\Norm{\uh}_{\Xhtaun}
\qquad\forall \uh \in \Yh.
\end{equation}
\end{proposition}
\begin{proof}
For any~$\uh$ in~$\Yh$, define~$\wh := \Newtonh (\Pistar \uh)$ in~$\Yh$.
It suffices to prove that
\begin{equation*}
  \frac{\bh(\uh,\uh + \delta \wh)}{\Norm{\uh + \delta \wh}_{\Ytaun}}\ge
\gammaI\Norm{\uh}_{\Xhtaun}
\end{equation*}
for a suitable real parameter~$\delta>0$, which will be fixed below.

The triangle inequality
and the definition of the norm~$\Norm{\cdot}_{\Xhtaun}$ in~\eqref{norm:Xh} imply
\begin{equation*}
\Norm{\uh + \delta \wh}_{\Ytaun}^2
\le
2\left(\Norm{\uh}_{\Ytaun}^2 \!+\! \delta^2 \Norm{\wh}_{\Ytaun}^2 \right)
\le
2 \max(1,\delta^2) \Norm{\uh}_{\Xhtaun}^2,
\end{equation*}
whence we deduce
\begin{equation} \label{step1:inf-sup}
\Norm{\uh + \delta \wh}_{\Ytaun}
\le \sqrt 2 \max(1,\delta) \Norm{\uh}_{\Xhtaun}.
\end{equation}
Next, recalling~\eqref{bh} and~\eqref{eq:upwind}, we write
\begin{equation}\label{eq:buu}
\begin{split}
\bh(\uh,\uh) =
&\sum_{\E \in \taun} 
    \left( \cH(\dpt \Pistar \uh, \uh)_{0,\E} + \ahE(\uh, \uh)\right)
    +\cH\Norm{\Pistar \uh^{(1)}(\cdot,0)}_{0,\Omega}^2 \\
& +\cH\sum_{n=2}^N\left(\Pistar \uh^{(n)}(\cdot,t_{n-1}) - \Pistar\uh^{(n-1)}(\cdot,t_{n-1}),\uh^{(n)}(\cdot,t_{n-1})\right)_{0,\Omega}.
\end{split}
\end{equation}
For $\E=\Ex\times\In$, we have
\begin{equation*}
(\dpt \Pistar \uh, \uh)_{0,\E}
\stackrel{\eqref{Pistar-1}}{=} 
(\dpt \Pistar \uh, \Pistar\uh)_{0,\E}
=
\frac12\left(\Norm{\Pistar \uh^{(n)}(\cdot,t_n)}_{0,\Ex}^2
-\Norm{\Pistar \uh^{(n)}(\cdot,t_{n-1})}_{0,\Ex}^2
\right).
\end{equation*}
By~\eqref{Pistar-2}, we have $\uh^{(n)}(\cdot,\tnmo)=\Pistar \uh^{(n)}(\cdot,\tnmo)$.
Simple calculations give
\begin{equation*}
\begin{split}
&\sum_{\E \in \taun}
(\dpt \Pistar \uh, \uh)_{0,\E} 
+\Norm{\Pistar  \uh^{(1)}(\cdot,0)}_{0,\Omega}^2
    \!\!\!+\sum_{n=2}^N\left( \Pistar  \uh^{(n)}(\cdot,t_{n-1}) - \Pistar \uh^{(n-1)}(\cdot,t_{n-1}),\uh^{(n)}(\cdot,t_{n-1})\right)_{0,\Omega}\\
&= 
\sum_{n=1}^N\left(\frac12\Norm{\Pistar \uh^{(n)}(\cdot,t_n)}_{0,\Omega}^2
-\frac12\Norm{\Pistar \uh^{(n)}(\cdot,t_{n-1})}_{0,\Omega}^2
\right)+\Norm{\Pistar\uh^{(1)}(\cdot,0)}_{0,\Omega}^2\\
&\qquad+\sum_{n=2}^N\left(
\Norm{\Pistar \uh^{(n)}(\cdot,t_{n-1})}_{0,\Omega}^2
-\left(\Pistar \uh^{(n-1)}(\cdot,t_{n-1}),
\Pistar \uh^{(n)}(\cdot,t_{n-1})\right)_{0,\Omega}
\right)\\
&=\frac12\Norm{\Pistar\uh^{(1)}(\cdot,0)}_{0,\Omega}^2
+\sum_{n=2}^N\frac12\Norm{\Pistar \uh^{(n)}(\cdot,t_{n-1})-
\Pistar \uh^{(n-1)}(\cdot,t_{n-1})
}_{0,\Omega}^2
+\frac12\Norm{\Pistar\uh^{(N)}(\cdot,T)}_{0,\Omega}^2. 
\end{split}
\end{equation*}
Therefore, from~\eqref{eq:buu} and~\eqref{stability-bounds-1:global}, we get
\begin{equation} \label{step2:inf-sup}
\begin{split}
\bh(\uh,\uh) \ge &\alpha_* \Norm{\uh}_{\Ytaun}^2
+ \frac{\cH}2 \Big(
\Norm{\Pistar \uh^{(1)}(\cdot,0)}_{0,\Omega}^2\\
&+\sum_{n=2}^N\Norm{\Pistar  \uh^{(n)}(\cdot,\tnmo)
- \Pistar \uh^{(n-1)}(\cdot,t_{n-1})}^{2}_{0,\Omega}
+\Norm{\Pistar\uh^{(N)}(\cdot,T)}_{0,\Omega}^2\Big).
\end{split}
\end{equation}
Moreover, the definition of $\bh(\cdot,\cdot)$ in~\eqref{eq:upwind} and~\eqref{bh}, and
the definition of the VE Newton potential in~\eqref{Newton-VEM} imply
\begin{equation} \label{bh_delta_w}
\bh(\uh,\delta\wh)
= \delta \left( \ah(\wh,\wh) + \ah(\uh,\wh)    \right).
\end{equation}
Since~\eqref{stability-bounds-1:global} gives
$
\ah(\wh,\wh) \ge \alpha_* \Norm{\wh}_{\Ytaun}^2$, then
Young's inequality entails,
for all positive~$\varepsilon$,
\begin{equation*}
\begin{split}
\ah(\uh,\wh)
& \ge - \left( \ah(\uh,\uh) \right)^{\frac12} 
            \left( \ah(\wh,\wh) \right)^{\frac12}
  \ge - \alpha^* \Norm{\uh}_{\Ytaun} \Norm{\wh}_{\Ytaun}\\
& \ge -\frac{\alpha^*}{2\varepsilon}  \Norm{\uh}_{\Ytaun}^2
        -\frac{\alpha^* \varepsilon}{2} \Norm{\wh}_{\Ytaun}^2.
\end{split}
\end{equation*}
Inserting the two above inequalities into \eqref{bh_delta_w} yields
\begin{equation} \label{step3:inf-sup}
\bh(\uh,\delta\wh)
\ge \delta \left( \alpha_* -\frac{\alpha^* \varepsilon}{2} \right) \Norm{\wh}_{\Ytaun}^2
    -\frac{\alpha^* \delta}{2\varepsilon} \Norm{\uh}_{\Ytaun}^2.
\end{equation}
As a final step, we sum~\eqref{step2:inf-sup} and~\eqref{step3:inf-sup}:
\begin{equation*}
\begin{split}
\bh(\uh,\uh+\delta\wh)
\ge &\left( \alpha_* - \frac{\alpha^*\delta}{2\varepsilon} \right) \Norm{\uh}_{\Ytaun}^2
+ \delta \left( \alpha_* - \frac{\alpha^* \varepsilon}{2} \right) \Norm{\wh}_{\Ytaun}^2+ \frac{\cH}2 \Big(
\Norm{\Pistar  \uh^{(1)}(\cdot,0)}_{0,\Omega}^2\\
&+\sum_{n=2}^N\Norm{\Pistar  \uh^{(n)}(\cdot,\tnmo)
-\Pistar \uh^{(n-1)}(\cdot,t_{n-1})}_{0,\Omega}^2
+\Norm{\Pistar\uh^{(N)}(\cdot,T)}_{0,\Omega}^2\Big).
\end{split}
\end{equation*}
Taking~$0<\varepsilon<(2\alpha_*)/\alpha^*$
and~$0<\delta<(2\varepsilon \alpha_*)/\alpha^*$,
defining
\begin{equation*}
\beta:= \min \left( \alpha_* - \frac{\alpha^*\delta}{2\varepsilon},
           \delta \left( \alpha_* - \frac{\alpha^* \varepsilon}{2} \right) \right) >0,
\end{equation*}
and recalling~\eqref{norm:Xh} and~\eqref{step1:inf-sup},
we can write
\begin{equation*}
\begin{split}
\bh(\uh,
& \uh + \delta \wh)
\ge \beta \left( \Norm{\uh}_{\Ytaun}^2 
                + \Norm{\wh}_{\Ytaun}^2 \right)
                +\frac{\cH}2 \Big(
\Norm{\Pistar \uh^{(1)}(\cdot,0)}_{0,\Omega}^2\\
&\qquad +\sum_{n=2}^N\Norm{\Pistar \uh^{(n)}(\cdot,\tnmo)
-\Pistar \uh^{(n-1)}(\cdot,t_{n-1})}_{0,\Omega}^2
+\Norm{\Pistar\uh^{(N)}(\cdot,T)}_{0,\Omega}^2\Big)\\
& \ge \min \left(1,\beta \right)
\Norm{\uh}_{\Xhtaun}^2
\ge \frac{\min \left(1,\beta \right)}{\sqrt2 \max(1,\delta)} \Norm{\uh}_{\Xhtaun} \Norm{\uh+\delta\wh}_{\Ytaun}.
\end{split}
\end{equation*}
The assertion follows with
$\gammaI:= \min \left(1,\beta \right) / (\sqrt2 \max(1,\delta)). $
\end{proof}

We are in a position to prove the well posedness of the method in~\eqref{VEM}.

\begin{theorem} \label{theorem:well-posedness}
There exists a unique solution $\uh$ to the method in~\eqref{VEM} with the following continuous dependence on the data:
\begin{equation*}
\Norm{\uh}_{\Xhtaun}
\le \gammaI^{-1} \CP \nu^{-1} \Norm{f}_{0,\QT},
\end{equation*}
where~$\gammaI$ is the discrete inf-sup constant in~\eqref{discrete:inf-sup},
$\CP$ is the Poincar\'e-type inequality constant in~\eqref{global-Poincare}
and~$\nu$ is the thermal conductivity.
\end{theorem}
\begin{proof}
The discrete inf-sup condition~\eqref{discrete:inf-sup} implies uniqueness of the solution.
The existence follows from the uniqueness,
owing to the finite dimensionality of~$\Yh$.
As for the stability bound,
we apply again the inf-sup condition~\eqref{discrete:inf-sup}
and recall the definition~\eqref{VEM} of the method:
\begin{equation*}
\Norm{\uh}_{\Xhtaun}
\le \frac{1}{\gammaI} \sup_{0\ne\vh\in\Yh}
\frac{\bh(\uh,\vh)}{\Norm{\vh}_{\Ytaun}}
=\frac{1}{\gammaI} \sup_{0\ne\vh\in\Yh}
\frac{(f, \Pizpmo \vh)_{0,\QT}}{\Norm{\vh}_{\Ytaun}}.
\end{equation*}
The Cauchy-Schwarz inequality, the $L^2$ stability of $\Pizpmo$,
and the global Poincar\'e-type inequality~\eqref{global-Poincare} give the assertion.
\end{proof}

\section{Convergence analysis} \label{section:convergence-analysis}
In this section, we analyze the convergence of the method in~\eqref{VEM}.
We start by introducing further technical tools
in Subsection~\ref{subsection:nc-preliminary},
which are typical of the nonconforming framework.
Then, in Subsection~\ref{subsection:Strang,subsection:convergence}, we develop an \textit{a priori} error analysis in two steps:
first, we prove a convergence result \emph{\`a la} Strang; next, we derive optimal convergence rates,
by using interpolation and polynomial approximation results,
assuming sufficient regularity on the solution.

\subsection{Technical results} \label{subsection:nc-preliminary}
Introduce the bilinear form
$
\NCh: L^2 \big(0,T;H^{\frac32+\varepsilon}(\Omega) \big)
\times \Yh \to \Rbb
$
given by
\begin{equation} \label{nc-bf}
\NCh(u,\vh):= \nu \sum_{n=1}^N \int_{\In}  
             \sum_{\Fx \in \Fcalxh} \int_{\Fx} \!\!\! \nablax u \cdot \jump{\vh}_{\F} \dS\ \dt.
\end{equation}
This bilinear form encodes information on the nonconformity of the
space~$\Yh$ across time-like facets.

On $\E = \In \times \Ex$, define the local bilinear form 
\begin{equation*}
\bE(w, v) := \int_{\In} \int_{\Ex} \left(\cH \dpt{} w v + \nu \nablax w \cdot \nablax v\right) \dx \dt .
\end{equation*}
\begin{lemma} \label{lemma:properties:nc-bf}
Assume that the solution $u$ to the continuous problem~$\eqref{continuous-weak}$
belongs to $L^2 (0,T;H^{\frac32+\varepsilon}(\Omega))$.
Then, for all~$\vh$ in~$\Yh$,
\begin{equation} \label{identity:nc}
\sum_{\E \in \taun} \bE(u,\vh) = (f,\vh)_{0,\QT} + \NCh(u,\vh).
\end{equation}
\end{lemma}
\begin{proof}
Integrating by parts in space and recalling the definition of~$\NCh$ in~\eqref{nc-bf}, we can write
\begin{equation*}
\begin{split}
\sum_{\E \in \taun} \bE(u,\vh)
& = \sum_{\E \in \taun} \int_{\In} 
    \Big( \int_{\Ex} (\cH \dpt u - \nu \Deltax u)\vh \dx 
    + \nu \sum_{\Fx \in \FcalEx} \int_{\Fx} \vh (\nbfFx \cdot \nablax u)  \dS \Big)\dt\\
& = (f,\vh)_{0,\QT} + \NCh(u,\vh),
\end{split}
\end{equation*}
which proves~\eqref{identity:nc}.
\end{proof}
We introduce another preliminary result,
which characterizes the polynomial inconsistency of the method in~\eqref{VEM}.
To this aim, we define the bilinear form
$\JcalE: H^1(\taun) \times \Yh \to \Rbb$
given on each~$\E = \Ex \times \In$ in~$\taun$ as follows:
for all~$w$ in~$H^1(\taun)$, $\vh$ in~$\Yh$,
\begin{equation} \label{time-inconsistency-operator}
\begin{split}
&
\JcalE(w,\vh) :=
\begin{cases}
\cH\left(\Pistar  w^{(1)}(\cdot,0), \vh^{(1)}(\cdot,0)\right)_{0,\Ex} 
& \text{if $n=1$};\\[0.2cm]
\cH\left(\Pistar  w^{(n)}(\cdot,\tnmo) - \Pistar w^{(n-1)}(\cdot,\tnmo), \vh^{(n)}(\cdot,\tnmo)\right)_{0,\Ex}  
& \text{if $2\le n\le N$}.\\
\end{cases}
\end{split}
\end{equation}
This bilinear form encodes the polynomial inconsistency of the method at space-like facets, as stated in the following lemma.

\begin{lemma} \label{lemma:consistency}
The local bilinear forms~$\bhE(\cdot,\cdot)$ satisfy
\begin{equation} \label{polynomial:consistency}
\begin{split}
& \bhE(\qp,\vh) = 
\bE(\qp,\vh)+ \JcalE(\qp, \vh)
\qquad
\forall \qp \in \Scalptaun,\; \forall \vh \in \VhE,
\; \forall \E \in \taun.
\end{split}
\end{equation}
\end{lemma}
\begin{proof}
Thanks to the definition of the bilinear form~$\aE(\cdot,\cdot)$,
the orthogonality properties of the projector~$\PiN$,
and the fact that the projectors~$\Pistar$ and~$\PiN$ preserve polynomials of degree~$p$,
we have
\begin{equation*}
\begin{split}
\bhE(\qp,\vh)
& = \cH (\dpt \Pistar \qp, \vh)_{0,\E} 
    + \aE(\PiN\qp, \PiN \vh) 
    + \nu \SE( (I \!-\! \PiN)\qp, (I \!-\! \PiN) \vh)
    + \JcalE(\Pistar \qp, \vh)\\
& = \cH (\dpt \Pistar\qp, \vh)_{0,\E}
+\aE(\qp,\PiN\vh)
+\JcalE(\qp, \vh)\\
&=\cH (\dpt \qp, \vh)_{0,\E}
+\aE(\qp,\vh) +\JcalE(\qp, \vh)
= \bE(\qp,\vh) +\JcalE(\qp, \vh).
\end{split}
\end{equation*}
This completes the proof.
\end{proof}

\subsection{A Strang-type result} \label{subsection:Strang}
We prove an \emph{a priori} estimate for the method in~\eqref{VEM}.

\begin{theorem} \label{theorem:Strang}
Let~$u$ and~$\uh$ be the solutions to~\eqref{continuous-weak} and~\eqref{VEM},
$u$ belong to~$X \cap L^2(0,T,H^{\frac32+\varepsilon}(\Omega))$
for some~$\varepsilon>0$,
$\uI$ in~$\Yh$ be the $\DoF$ interpolant of~$u$ in~$\Yh$,
and~$\gammaI$ be the discrete inf-sup constant appearing in~\eqref{discrete:inf-sup}.
Then, we have
\begin{equation} \label{Strang-estimate}
\begin{split}
\Norm{u-\uh}_{\Xhtaun}
& \le \Norm{u-\uI}_{\Ytaun}
+ \gammaI^{-1} \sup_{0\ne\vh\in\Yh} \Bigg[
   \frac{\vert (f- \Pizpmo f, \vh )_{0,\QT} \vert}{\Norm{\vh}_{\Ytaun}}
   + \frac{\vert \NCh(u,\vh) \vert}{\Norm{\vh}_{\Ytaun}} \\
& \quad + \inf_{\qp \in \Scalptaun}
    \frac{\sum_{\E\in\taun}\left( \bhE(u-\qp,\vh) - \bE(u - \qp,\vh) +\JcalE(\qp, \vh) \right)}{\Norm{\vh}_{\Ytaun}}  \Bigg].
\end{split}
\end{equation}
\end{theorem}
\begin{proof}
By the triangle inequality, we have
\begin{equation*}
\Norm{u-\uh}_{\Xhtaun}
\le 
\Norm{u-\uI}_{\Xhtaun} + \Norm{\uI-\uh}_{\Xhtaun}
=: T_1 + T_2.
\end{equation*}
Since~$\Pistar$ is computable from the $\DoFs$,
we have that $\Pistar(u-\uI)=0$ in each element.
Taking into account~\eqref{norm:Xh}, this yields
\begin{equation*}
\begin{split}
 T_1^2
 &= 
   \Norm{u-\uI}_{\Ytaun}^2
   +\Norm{\Newtonh \left(\Pistar(u-\uI)\right)}_{\Ytaun}^2
   +\frac{\cH}{2} \Big(\Norm{\Pistar (u-\uI)(\cdot,0)}_{0,\Omega}^2\\
&\qquad +\sum_{n=2}^{N}
\Norm{\Pistar (u-\uI^{(n)})(\cdot,\tnmo)
-\Pistar(u-\uI^{(n-1)})(\cdot,\tnmo)}_{0,\Omega}^2
+\Norm{\Pistar(u-\uI)(\cdot,T)}_{0,\Omega}^2\Big)\\
 &= \Norm{u-\uI}_{\Ytaun}^2.
\end{split}
\end{equation*}
The rest of this proof is devoted to estimate the term~$T_2$.
The definition of~$\uI$ implies
\begin{equation*}
\bh(\uI,\vh) = \bh(u,\vh) \qquad \forall \vh \in \Yh.
\end{equation*}
Using this property and the discrete inf-sup condition~\eqref{discrete:inf-sup},
we get
\begin{equation*}
\Norm{\uI-\uh}_{\Xhtaun}
\le \gammaI^{-1} \sup_{0\ne \vh \in \Yh} \frac{\bh(\uI-\uh,\vh)}{\Norm{\vh}_{\Ytaun}}
= \gammaI^{-1} 
\sup_{0\ne \vh \in \Yh} \frac{\bh(u-\uh,\vh)}{\Norm{\vh}_{\Ytaun}}.
\end{equation*}
We recall~\eqref{VEM},
add and subtract any~$\qp$ in~$\Scalptaun$,
use the inconsistency property~\eqref{polynomial:consistency},
add and subtract~$u$, recall the property of the nonconformity bilinear form~$\NCh(\cdot,\cdot)$ in~\eqref{identity:nc},
and deduce
\begin{equation*}
\begin{split}
& \bh(u-\uh,\vh)
 = \sum_{\E \in \taun} \bhE(u,\vh) - (f, \Pizpmo \vh)_{0,\QT}\\
& = \sum_{\E \in \taun} \left( \bhE(u-\qp,\vh) + \bhE(\qp,\vh) \right) - (f, \Pizpmo \vh)_{0,\QT}\\
& = \sum_{\E \in \taun}
\left( \bhE(u-\qp,\vh) + \bE(\qp,\vh) +\JcalE(\qp, \vh) \right)
- (f, \Pizpmo \vh)_{0,\QT}\\
& = \sum_{\E \in \taun} \left( \bhE(u-\qp,\vh) + \bE(\qp-u,\vh)
+\bE(u,\vh) +\JcalE(\qp, \vh) \right) 
- (f, \Pizpmo \vh)_{0,\QT}\\
& = \sum_{\E \in \taun} \left( \bhE(u-\qp,\vh) - \bE(u - \qp,\vh) +\JcalE(\qp, \vh) \right) 
+ (f- \Pizpmo f, \vh)_{0,\QT} + \NCh(u,\vh).
\end{split}
\end{equation*}
The assertion follows from
taking the infimum over all~$\qp$ in~$\Scalptaun$
and then the supremum over all~$\vh$ in~$\Yh$.
\end{proof}

\subsection{\textit{A priori} error estimate} \label{subsection:convergence}
The aim of this section is to prove optimal convergence rates
for the method in~\eqref{VEM}.
So far, we derived all estimates with explicit constants,
so as to track the use of different type of inequalities
(Poincar\'e, trace, inverse estimates, $\dots$).
Furthermore, we kept separated the contributions of~$\hEx$ and~$\htime$.
In this section, we shall not keep this level of detail.
As a matter of notation,
we henceforth write~$a \lesssim b$ meaning that there exists a positive constant~$c$
independent of the meshsize,
such that~$a \leq c b$.
We also write~$a \simeq b$ if~$a \lesssim b$ and~$b \lesssim a$ at once.

We prove error estimates under some regularity assumptions on the exact solution
and focus on the case of isotropic space-time meshes,
i.e., assume that
\begin{equation} \label{isotropic-properties}
\hEx \simeq \htime \simeq \hE
\qquad \forall \E = \Ex \times \In \in \taun.
\end{equation}

\noindent In Theorem~\ref{theorem:Strang},
we proved that the error of the method in~\eqref{VEM}
is bounded by the sum of four terms of different flavour:
\emph{i)} a VE interpolation error;
\emph{ii)} a term involving the discretization of the right-hand side~$f$;
\emph{iii)} a term measuring the spatial nonconformity of the discrete space;
\emph{iv)} a term involving polynomial error estimates,
which appears because of the temporal nonconformity and the polynomial inconsistency of the discrete bilinear form.
Based on that result, we prove the following theorem.

\begin{theorem} \label{theorem:a-priori:error-estimates}
Let 
assumptions (\textbf{G1})-(\textbf{G3}) be valid,
and~$\taun$ be isotropic
in the sense of~\eqref{isotropic-properties}.
Let~$u$, the solution of~\eqref{continuous-weak}, and $f$, the right-hand side
of~\eqref{continuous-weak},
belong to~$H^{\p+1}(\taun)$
and~$H^\p(\taun)$, respectively,
where~$p \geq 1$ denotes the degree of approximation of the method in~\eqref{VEM}.
Let~$\uh$ be the solution to~\eqref{VEM}.
Then, 
\begin{equation} \label{error:estimates}
\Norm{u-\uh}_{\Xhtaun} 
\lesssim \h^\p (\SemiNorm{u}_{\p+1,\taun} + \SemiNorm{f}_{\p,\taun} ).
\end{equation}
\end{theorem}
\begin{proof}
We estimate the four terms on the right-hand side of~\eqref{Strang-estimate} separately.
The assertion then follows by combining the four bounds we provide below.

\medskip
\noindent \textbf{Part \emph{i)} VE interpolation error.}
For any~$\qp$ in~$\Scalptaun$, the triangle inequality implies
\begin{equation} \label{bound_II}
\Norm{u-\uI}_{\Ytaun}
\le \SemiNorm{u-\qp}_{\Ytaun} + \SemiNorm{\qp-\uI}_{\Ytaun}.
\end{equation}
We focus on the second term on the right-hand side. 
For any $\E$ in~$\taun$,
$(\qp-\uI){}_{|_{\E}}$ belongs to~$\VhE$.
Therefore, the stability bounds~\eqref{stability-bounds:local} entail
\begin{equation*}
\SemiNorm{\qp-\uI}_{\YE}^2
\lesssim \ahE(\qp-\uI,\qp-\uI)
\qquad \forall \E \in \taun.
\end{equation*}
Since~$\uI$ is the $\DoFs$ interpolant of~$u$ and~$\ah(\cdot,\cdot)$ is computed via the $\DoFs$, the above inequality also implies
\begin{equation*}
\SemiNorm{\qp-\uI}_{\Ytaun}^2 
\lesssim \ah(\qp-\uI,\qp-\uI)
=  \ah(\qp-u,\qp-u).
\end{equation*}
Furthermore, using the discrete continuity property~\eqref{continuity:continuous},
we arrive at
\begin{equation*}
\SemiNorm{\qp-\uI}_{\Ytaun}^2
\lesssim \sum_{\E \in \taun}
\left( \hE^{-2} \Norm{u-\qp}_{0,\E}^2
+ \SemiNorm{u-\qp}_{1,\E}^2 \right).
\end{equation*}
Inserting this into \eqref{bound_II} and using standard polynomial approximation results yield
\begin{equation*}
\Norm{u-\uI}_{\Ytaun}
\lesssim \h^\p \SemiNorm{u}_{\p+1, \taun}.
\end{equation*}

\medskip 
\noindent \textbf{Part \emph{ii)} Handling the variational crime on the right-hand side~$f$.}
Using the definition of~$\Pizpmo$,
standard polynomial approximation estimates,
and the global discrete Poincar\'e inequality~\eqref{global-Poincare} entail
\begin{equation*}
\begin{split}
& (f- \Pizpmo f,\vh)_{0,\QT}
  \le \sum_{\E \in \taun} \Norm{f - \PizE f}_{0,\E} \Norm{\vh }_{0,\E}
\lesssim \h^\p \SemiNorm{f}_{\p,\taun} \Norm{\vh}_{0,\QT}
  \lesssim \h^\p \SemiNorm{f}_{\p,\taun} \Norm{\vh}_{Y(\taun)} .
\end{split}
\end{equation*}

\medskip
\noindent \textbf{Part \emph{iii)} Handling the variational crime of the time-like nonconformity.}
We estimate
\begin{equation*}
\sup_{0\ne \vh \in \Yh} \frac{\vert \NCh(u,\vh) \vert}{\Norm{\vh}_{\Ytaun}}
= 
\sup_{0\ne \vh \in \Yh} \frac{\vert \nu \sum_{\F \in \Fcalh} \int_{\In} \int_{\Fx} \nablax u \cdot \llbracket\vh\rrbracket_{\F}  \dS\, \dt\vert }{\Norm{\vh}_{\Ytaun}}.
\end{equation*}
We present estimates on a single facet~$\F=\Fx \times \In$.
For the sake of simplicity,
we assume that~$\F$ is an internal time-like facet shared by the elements~$\E_1$ and~$\E_2$.
Using the definition of the spatial nonconformity of the space~$\Yh$, see~\eqref{nonconforming-condition},
and the properties of $L^2$ projectors,
for all~$\qpto$, $\qptt$ in~$\Pbb_\p(\In)$,
we write
\footnote{Here we use the scalar normal jump~$\jump{\cdot}$ defined in the proof of Proposition~\ref{prop:global-Poincare}.}
\begin{equation*}
\begin{split}
& \left| \int_{\In}\int_{\Fx} \nablax u \cdot \llbracket\vh\rrbracket_{\F} \dS\, \dt \right| 
= \left| \int_{\In} \int_{\Fx} \nablax u \cdot \nbfFx \llbracket\vh\rrbracket \dS\, \dt \right| \\ 
& = \left| \int_{\In} \int_{\Fx} \left(\nablax u \cdot \nbfFx - \PizF\left(\nablax u \cdot \nbfFx\right)\right) \llbracket\vh\rrbracket \dS\, \dt \right| \\
& = \left| \int_{\In} \int_{\Fx} \left(\nablax u \cdot \nbfFx - \PizF\left(\nablax u \cdot \nbfFx\right)\right) \left((\vh{}_{|_{K_2}} - \qptt ) - (\vh{}_{|_{K_1}} - \qpto )\right) \dS\, \dt \right|.
\end{split}
\end{equation*}
Next, use the Cauchy-Schwarz inequality,
the triangle inequality,
standard properties of the $L^2$ projector,
a trace inequality, the local quasi-uniformity of the space-time mesh,
and arrive at
\begin{equation*}
\begin{split}
& \left| \int_{\In}\int_{\Fx} \nablax u \cdot \llbracket\vh\rrbracket_{\F} \dS\, \dt \right|  \\
& \leq \Norm{\nablax u \cdot \nbfFx - \PizF\left(\nablax u \cdot \nbfFx\right)}_{L^2(F)}
\left(\Norm{\vh{}_{|_{K_2}} - \qptt }_{L^2(F)} + \Norm{\vh{}_{|_{K_1}} - \qpto }_{L^2(F)}\right)\\
& \leq \Norm{\nablax \left(u - \Pi_{p+1}^{0, \E_1} u\right) \cdot \nbfFx}_{L^2(F)} 
\left(\Norm{\vh{}_{|_{K_2}} - \qptt }_{L^2(F)} + \Norm{\vh{}_{|_{K_1}} - \qpto }_{L^2(F)}\right)\\
& \lesssim \Big(\h_{\E_1}^{-\frac12} \SemiNorm{u - \Pi_{p+1}^{0, \E_1} u}_{Y(\E_1)}
+ \h_{\E_1}^{\epsilon}\SemiNorm{u-\Pi_{p+1}^{0, \E_1} u}_{\frac32+\varepsilon,\E_1} \Big)  \\
& \quad \times \Big( \sum_{j=1,2} 
         \Big(\h_{\E_j}^{-\frac12} \Norm{\vh-\qptj}_{0,\E_j}
         + \h_{\E_j}^{\frac12} \SemiNorm{\vh}_{Y(\E_j)} \Big)\Big).
\end{split}
\end{equation*}
An application of~\eqref{particular-Poincare} yields
\begin{equation*}
\begin{split}
& \left| \int_{\In}\int_{\Fx} \nablax u \cdot \llbracket\vh\rrbracket_{\F} \dS\, \dt \right| 
\lesssim \left( \SemiNorm{u - \Pi_{p+1}^{0, \E_1} u}_{Y(\E_1)} 
+ \h_{\E_1}^{\epsilon+\frac12} \SemiNorm{u-\Pi_{p+1}^{0, \E_1} u}_{\frac32+\varepsilon,\E_1} \right)
         \sum_{j=1,2}  \SemiNorm{\vh}_{Y(\E_j)}.
\end{split}
\end{equation*}
Summing up over all the elements and using approximation properties of the $L^2$ projector,
we eventually get
\begin{equation*}
\sup_{0\ne \vh \in \Yh} \frac{\vert \NCh(u,\vh) \vert}{\Norm{\vh}_{\Ytaun}}
\lesssim \h^{\p}
\SemiNorm{u}_{\p+1,\taun}.
\end{equation*}

\medskip
\noindent \textbf{Part \emph{iv.a)} Polynomial approximation error of~$\bE(\cdot,\cdot)$ type.}
Let~$\qp$ be in~$\Scalptaun$.
Using the Cauchy-Schwarz inequality twice
and the definition of the bilinear form~$\bE(\cdot,\cdot)$ give
\begin{equation*}
\bE(u \!-\! \qp,\vh)
= \cH\left(\dpt(u \!-\! \qp),\vh\right)_{0,\E}
   + \nu \left(\nablax(u \!-\! \qp),\nablax\vh\right)_{0,\E}
\lesssim \SemiNorm{u \!-\! \qp}_{1,\E} 
         ( \Norm{\vh}_{0,\E} + \SemiNorm{\vh}_{\YE}). 
\end{equation*}
Summing up over all the elements, using an~$\ell^2$ Cauchy-Schwarz inequality,
and recalling the global Poincar\'e-type inequality~\eqref{global-Poincare},
we can write
\begin{equation} \label{bound-IA}
\sum_{\E \in \taun} \bE(u-\qp,\vh)
\lesssim \SemiNorm{u-\qp}_{1,\taun}  \Norm{\vh}_{\Ytaun}.
\end{equation}

\medskip
\noindent \textbf{Part \emph{iv.b)} Polynomial approximation error of~$\bhE(\cdot,\cdot) + \JcalE(\cdot,\cdot)$ type.}
Thanks to definitions~\eqref{eq:upwind} and~\eqref{time-inconsistency-operator}
on each element~$\E=\Ex\times\In$ in~$\taun$,
for all~$\vh$ in~$\Yh$,
we have
\begin{equation} \label{Part-1-initial-step}
\bhE(u-\qp, \vh) + \JcalE(\qp,\vh)=
\cH(\dpt\Pistar(u-\qp),\vh)_{0,\E}
+ \ahE(u-\qp,\vh)
+\JcalE(u,\vh),
\end{equation}
where~$\qp$ is the same as in Part \emph{iv.a)}.
We first focus on the second term.
Using the stability bound~\eqref{stability-bounds:local}
and the continuity property~\eqref{continuity:continuous:0}
with~$\hEx\simeq\htime$,
we arrive at
\begin{equation*}
\ahE(u-\qp,\vh)
\le \ahE(u-\qp,u-\qp)^{\frac12}
      \ahE(\vh,\vh)^{\frac12}
 \lesssim \left( \hE^{-1} \Norm{u-\qp}_{0,\E}
            + \SemiNorm{u-\qp}_{1,\E} \right)
            \SemiNorm{\vh}_{\YE}.
\end{equation*}
Next, we deal with the first term on the right-hand side of~\eqref{Part-1-initial-step}.
The Cauchy-Schwarz inequality yields
\begin{equation*}
(\dpt\Pistar(u-\qp),\vh)_{0,\E}
\le \Norm{\dpt\Pistar(u-\qp)}_{0,\E}
    \Norm{\vh}_{0,\E}.
\end{equation*}
A polynomial inverse inequality gives
\begin{equation} \label{easy-inverse-estimate}
\Norm{\dpt\Pistar(u-\qp)}_{0,\E}
\lesssim \hE^{-1}
         \Norm{\Pistar(u-\qp)}_{0,\E}.
\end{equation}
By using~\eqref{inverse:Pistar},
the definition of~$\Pistar$,
the stability of the~$L^2$ orthogonal projection,
and the trace inequality,
we arrive at
\begin{equation} \label{main-step-Part-1b}
\begin{split}
\Norm{\Pistar(u-\qp)}_{0,\E}
& \lesssim \Norm{\PizE \Pistar(u-\qp)}_{0,\E}
    + \hE^{\frac12} \Norm{\PizEx \Pistar(u-\qp)(\cdot,\tnmo)}_{0,\Ex}\\
& = \Norm{\PizE (u-\qp)}_{0,\E}
    + \hE^{\frac12} \Norm{\PizEx (u-\qp)(\cdot,\tnmo)}_{0,\Ex}\\
& \le \Norm{u-\qp}_{0,\E}
+ \hE^{\frac12} \Norm{(u-\qp)(\cdot,\tnmo)}_{0,\Ex}
\lesssim \Norm{u-\qp}_{0,\E} + \hE \SemiNorm{u-\qp}_{1,\E}.
\end{split}
\end{equation}
Therefore, we obtain
\begin{equation*}
(\dpt\Pistar(u-\qp),\vh)_{0,\E}
\lesssim \left(\hE^{-1} \Norm{u-\qp}_{0,\E} + \SemiNorm{u-\qp}_{1,\E}\right)\Norm{\vh}_{0,\E}.
\end{equation*}
Finally, we estimate the third term on the right-hand side of~\eqref{Part-1-initial-step}.
Since the initial condition~$u(\cdot, 0)$ is zero,
$\JcalE(u,\vh)=0$ if~$n=1$.
So, we consider the case~$n\ge2$:
\begin{equation*}
\begin{split}
& \JcalE(u,\vh)
= \cH (\Pistar u^{(n)}(\cdot,\tnmo) - \Pistar u^{(n-1)}(\cdot,\tnmo), \vh^{(n)}(\cdot,\tnmo))_{0,\Ex} \\
& \!\lesssim \hE^{\frac12} \!\! \left( \Norm{u(\cdot,\tnmo) 
\!-\! \Pistar u^{(n)}(\cdot,\tnmo)}_{0,\Ex}
\!\!\!\!\!\!+ \Norm{u(\cdot,\tnmo) \!-\! \Pistar u^{(n-1)}(\cdot,\tnmo)}_{0,\Ex} \right)
 \hE^{-\frac12}\Norm{\vh^{(n)}(\cdot,\tnmo)}_{0,\Ex} \!\!.
\end{split}
\end{equation*}
Proceeding as in Proposition~\ref{proposition:properties-explicit-stabilization}, 
it is possible to show that
\begin{equation*}
\hE^{-\frac12}\Norm{\vh^{(n)}(\cdot,\tnmo)}_{0,\Ex}
\lesssim \SemiNorm{\vh}_{\YE}.
\end{equation*}
Thus, we can focus on the two terms involving~$u$.
As for the first one, we use a trace inequality along the time direction,
add and subtract the same~$\qp$ as above,
recall that~$\Pistar$ preserves polynomials of degree at most~$\p$,
use the triangle inequality,
apply the polynomial inverse estimate~\eqref{easy-inverse-estimate},
and get
\begin{equation*}
\begin{split}
& \hE^{\frac12} \Norm{u(\cdot,\tnmo) - \Pistar u^{(n)}(\cdot,\tnmo)}_{0,\Ex}
  \lesssim \Norm{u - \Pistar u}_{0,\E}
    + \hE \Norm{\dpt(u - \Pistar u)}_{0,\E}\\
& \le \Norm{u - \qp}_{0,\E}
    + \hE \Norm{\dpt(u - \qp)}_{0,\E}
    + \Norm{\Pistar (u-\qp)}_{0,\E}
    + \hE \Norm{\dpt(\Pistar (u-\qp)}_{0,\E}\\
& \le \Norm{u - \qp}_{0,\E}
    + \hE \Norm{\dpt(u - \qp)}_{0,\E}
    + \Norm{\Pistar (u-\qp)}_{0,\E}.
\end{split}
\end{equation*}
Next, we apply estimate~\eqref{main-step-Part-1b} and get
\begin{equation*}
\hE^{\frac12} \Norm{u(\cdot,\tnmo) - \Pistar u^{(n)}(\cdot,\tnmo)}_{0,\Ex}
\lesssim \Norm{u - \qp}_{0,\E}
    + \hE \SemiNorm{u - \qp}_{1,\E}.
\end{equation*}
For the second term involving $u$, we proceed analogously.
Setting
$\E' := \Ex \times \Inmo$
and using the local quasi-uniformity of the space-time mesh,
we get 
\begin{equation*}
\hE^{\frac12} \Norm{u(\cdot,\tnmo) - \Pistar u^{(n-1)}(\cdot,\tnmo)}_{0,\Ex}
\lesssim \Norm{u - \qp}_{0,\E'}
    + \h_{\E'} \SemiNorm{u - \qp}_{1,\E'}.
\end{equation*}
Summing over all the elements,
using standard manipulations (including~$\ell^2$ Cauchy-Schwarz inequalities),
and applying the global Poincar\'e type inequality~\eqref{global-Poincare} give
\begin{equation} \label{bound_IB}
\sum_{\E \in \taun} 
\left(\bhE(u-\qp,\vh) 
+ \JcalE(\qp,\vh)
\right)
\lesssim \sum_{\E \in \taun} \left(\hE^{-2} \Norm{u - \qp}_{L^2(K)}^2 + \SemiNorm{u - \qp}_{1, K}^2 \right)^{\frac{1}{2}} 
\Norm{\vh}_{Y(\taun)}.
\end{equation}
\medskip

\noindent \textbf{Conclusion of Part \emph{iv})} From~\eqref{bound-IA} and~\eqref{bound_IB},
which are valid for any $\qp\in\VhE$, and
standard polynomial approximation results, we obtain
\begin{equation*}
\sup_{0\ne\vh\in\Yh} \inf_{\qp\in\Scalptaun}
 \frac{ \sum_{\E \in \taun} 
 \left(\bhE(u-\qp,\vh) 
 - \bE(u-\qp,\vh) 
 + \JcalE(\qp,\vh) \right)}{\Norm{\vh}_{\Ytaun}} 
\lesssim \h^\p \SemiNorm{u}_{\p+1, \taun}.
\end{equation*}
This concludes Part \emph{iv}) and completes the whole proof.
\end{proof}

\section{Numerical results} \label{section:numerics}
In this section, we assess the error estimates proven in Theorem~\ref{theorem:a-priori:error-estimates}.
We developed an object-oriented MATLAB implementation to obtain high-order approximations of space-time $(1+1)$- and $(2+1)$-dimensional problems.
We briefly mention some relevant computational aspects regarding the numerical results below.
\begin{itemize}
\item In case of inhomogeneous initial and/or boundary conditions,
we set moments at $\Omega \times \{0\}$ and/or at $\partial \Omega \times (0, T)$ accordingly and modify the right-hand side.
This corresponds to a standard lifting procedure,
where the lifting has all the remaining moments equal to zero.
In this way,
in the presence of incompatible initial and boundary data,
no artificial compatibility condition is enforced on the discrete solutions.
\item In Theorem~\ref{theorem:a-priori:error-estimates},
error bounds are provided in the~$\Norm{\cdot}_{\Xhtaun}$ norm. 
Since the virtual element solution~$\uh$ to~\eqref{VEM} is not known in closed form and the error in the~$\Xhtaun$ norm is not computable,
we report the following associated error quantities:
\begin{subequations}
 \label{exact-errors}
\begin{equation}
\begin{split}
\EcalY
 & := \Norm{u-\PiN \uh}_{\Ytaun},
\quad 
 \EcalN := \Norm{\PiN(\Newtonh \Pistar (u - \uh))}_{\Ytaun}, \\
 (\EcalU)^2 & :=  \frac{\cH}{2} 
\Bigg(
\Norm{\Pistar(u - \uh)(\cdot, 0)}_{L^2(\Omega)}^2 + \Norm{\Pistar (u - \uh)(\cdot, T)}_{L^2(\Omega)}^2 
\\ 
& 
\quad +\sum_{n=2}^N\Norm{\Pistar (u - \uh)^{(n)}(\cdot,\tnmo) - \Pistar (u - \uh)^{(n-1)}(\cdot,\tnmo)}_{L^2(\Omega)}^2
\Bigg). 
\end{split}
\end{equation}
The~$\Xhtaun$ norm is related to the sum of $\EcalY$, $\EcalN$, and~$\EcalU$.
We also show the error in the~$L^2(\QT)$ norm, namely
\begin{equation}
\EcalL := \Norm{u - \Pistar \uh}_{L^2(\QT)},
\end{equation}
\end{subequations}
which is not covered by our theory. 
\item In all experiments, we take $\cH = \nu = 1$
and employ the stabilization in~\eqref{explicit:stabilization}.
\end{itemize}

\subsection{Results in (1+1)-dimension}  \label{subsection:1+1}
We use tensor-product meshes and uniform partitions along the space and time directions.

\subsubsection{Patch test} \label{subsection:patch-test-1D}
The discrete bilinear form~$\bh(\cdot,\cdot)$ in~\eqref{bh} is polynomial inconsistent;
see Lemma~\ref{lemma:consistency}.
However, thanks to the error estimates~\eqref{error:estimates},
the method in~\eqref{VEM} passes the patch test,
i.e., up to round-off errors,
polynomial solutions of order~$\p$ are approximated exactly.

We consider the following family of exact solutions
on 
$\QT = (0, 1) \times (0, 1)$:
\begin{equation} \label{exact-patch-1D}
u_p(x, t) = 
\begin{cases}
t^{p/2} x^{p/2}                                     & \text{if $p$ is even};\\
t^{(p-1)/2} x^{(p+1)/2} + t^{(p+1)/2} x^{(p-1)/2}   & \text{if $p$ is odd}.
\end{cases}
\end{equation}
For any $p \in \N$, $u_p$ belongs to~$\Pp{p}{\QT}$.
In Figure~\ref{patch-test-figs},
for $p = 1, \ldots, 5,$ we show the errors in the approximation of $u_p$ obtained using a sequence of meshes with~$\hEx = \htime = 5 \times 10^{-2} /  2^{i - 1}$, $i = 1, \ldots, 4$,
and approximation degree~$p$.
The scale of $10^{-10}$ in the figures validates the patch test.
The growth of the error observed while decreasing the mesh size
represents the actual effect of the condition number
when solving the linear systems stemming from~\eqref{VEM}.

\begin{figure}[!ht]
\centering
{
\includegraphics[width = 2.5in]{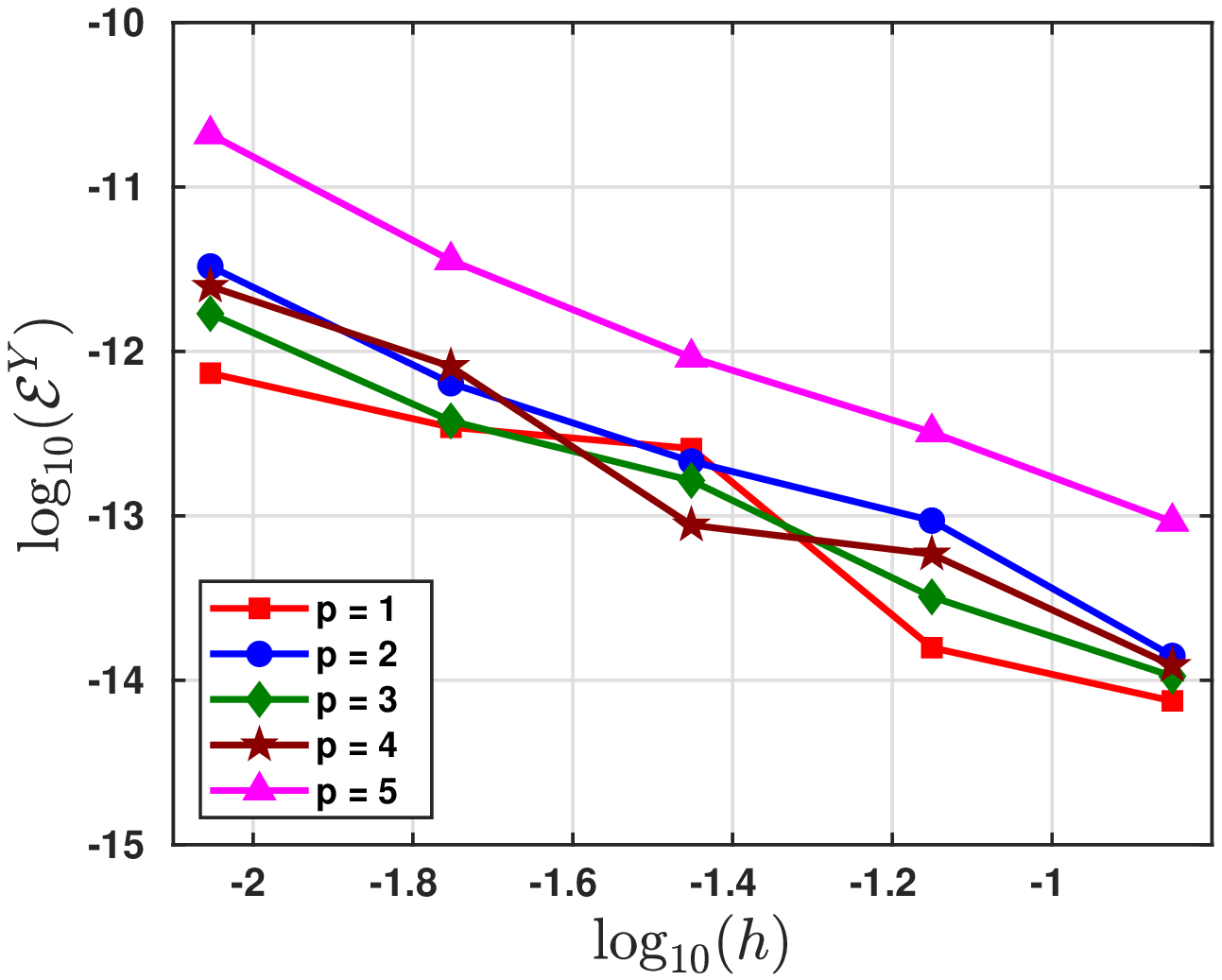}
}
\hspace{0.6cm} 
{
\includegraphics[width = 2.5in]{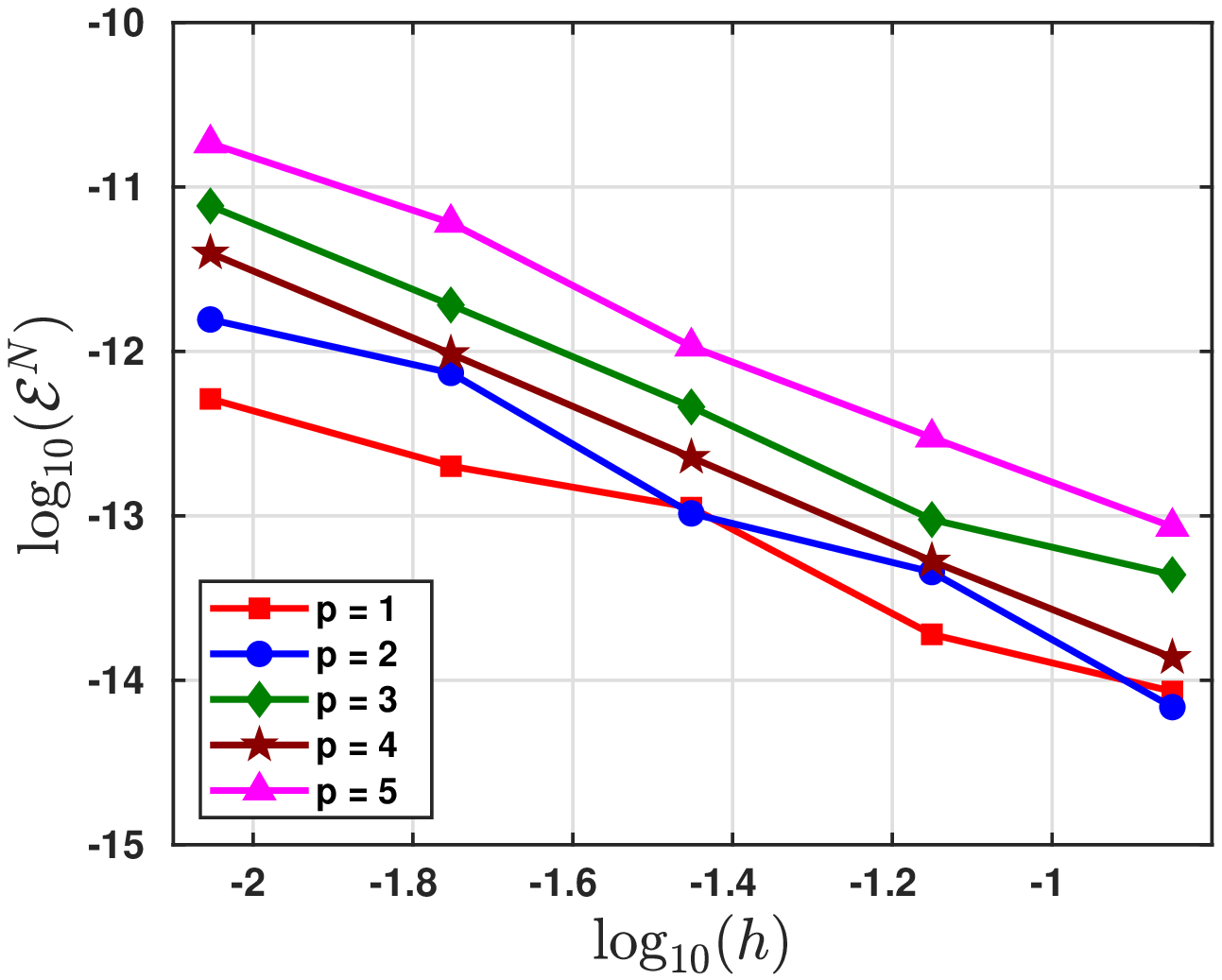}
}
\caption{$\h$-dependence of the errors in~\eqref{exact-errors} for the patch tests with solution~$u_p$ in \eqref{exact-patch-1D}.
\label{patch-test-figs}
}
\end{figure}

\subsubsection{Smooth solution} \label{exp:smooth-solution}
On the space-time domain~$\QT = (0, 1) \times (0, 1)$,
we consider
the 
problem with exact smooth solution
\begin{equation} \label{smooth-test-case}
u(x, t) = \sin(t)\sin(3\pi x).
\end{equation}
In Figure~\ref{fig:smooth-1D}, we show the rates of convergence \sg{of the errors in~\eqref{exact-errors}} obtained
using a sequence of meshes with~$\hEx = \htime = 0.2\times 2^{-i}$,
 for~$i =  1, \ldots, 5$,
and different approximation degrees~$p$.
We observe convergence of order~$\mathcal O (\h^{\p})$ for the error~$\EcalY$,
of order~$\mathcal{O}(\h^{\p+\frac12})$ for the error~$\EcalU$,
and of order~$\mathcal{O}(\h^{\p+1})$ for the errors~$\EcalN$
and~$\EcalL$.
Such rates of convergence are in agreement with estimate~\eqref{error:estimates}
and the approximation rates that might be expected from the norms in~\eqref{exact-errors}.

\begin{figure}[!ht]
\centering
{
\includegraphics[width = 2.5in]{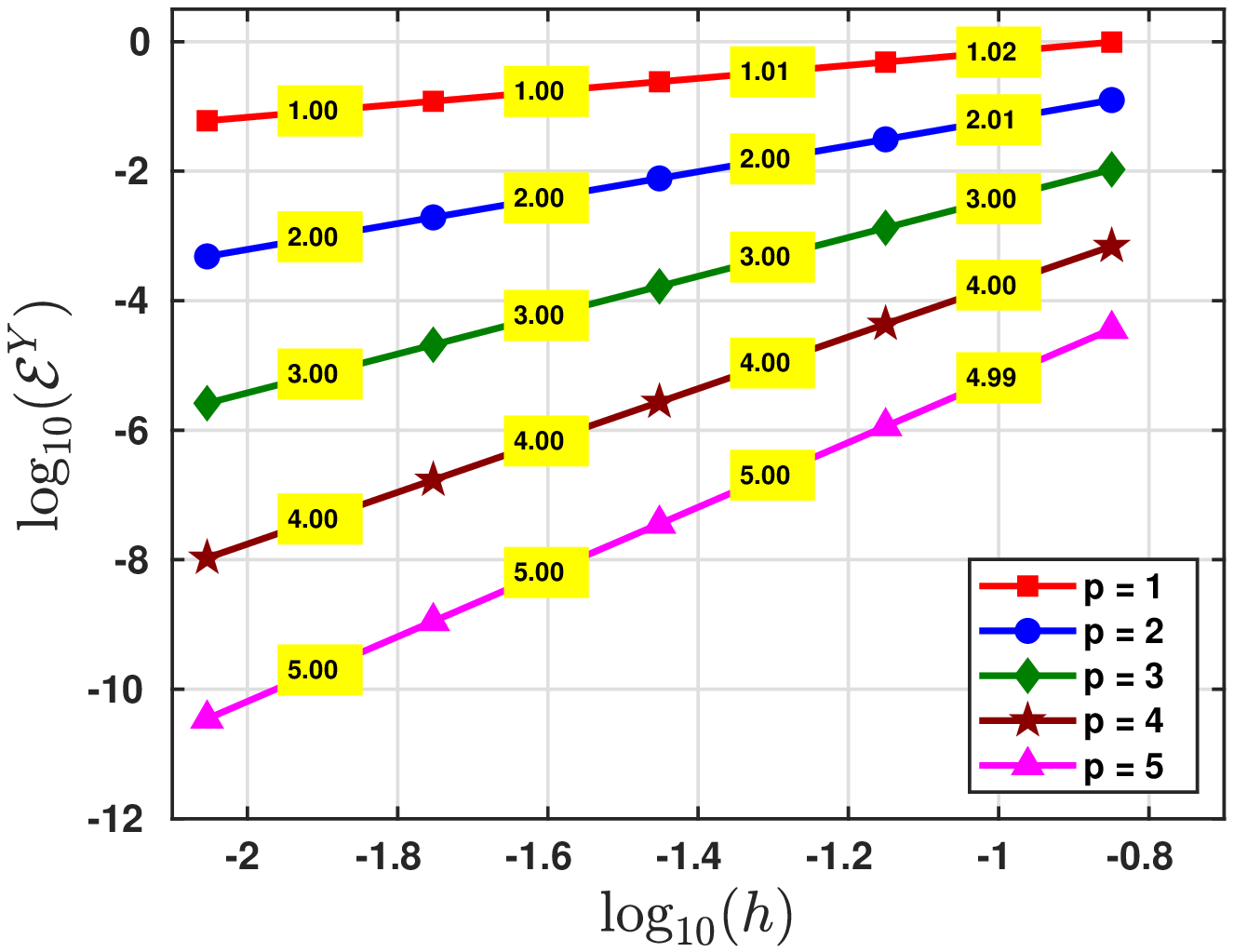}
}
\hspace{0.6cm}
{
\includegraphics[width = 2.5in]{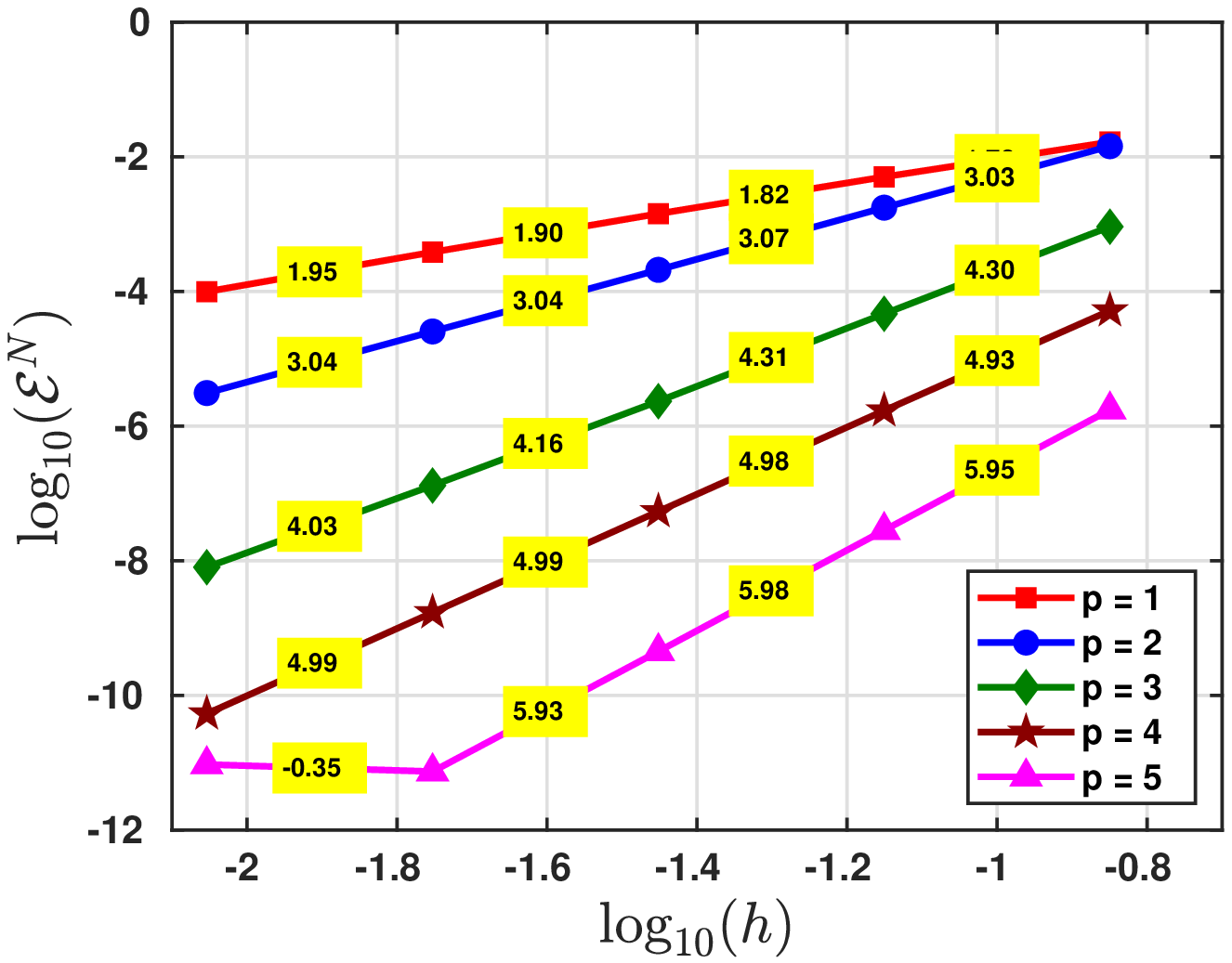}
}\\
{
\includegraphics[width = 2.5in]{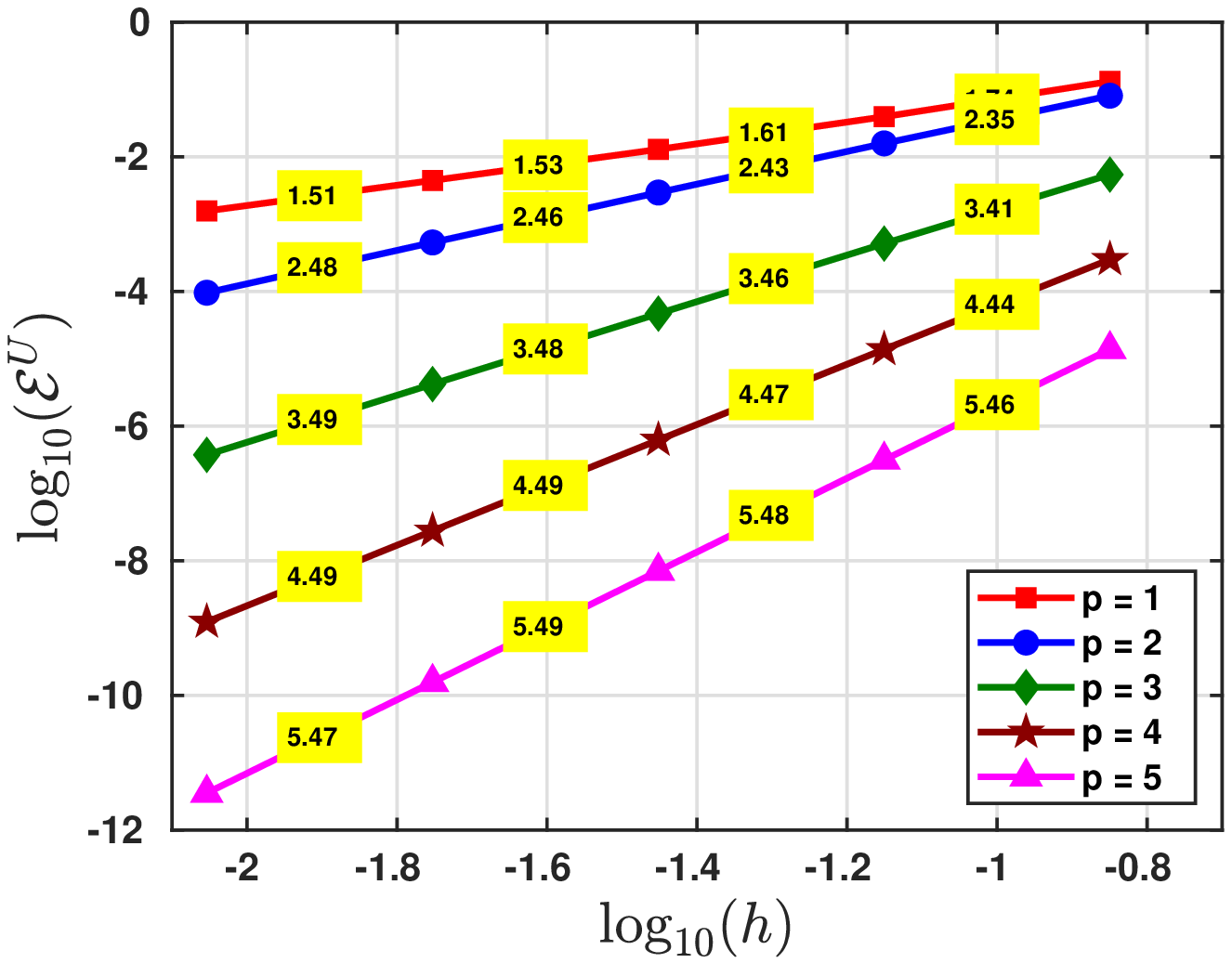}
}
\hspace{0.6cm}
{
\includegraphics[width = 2.5in]{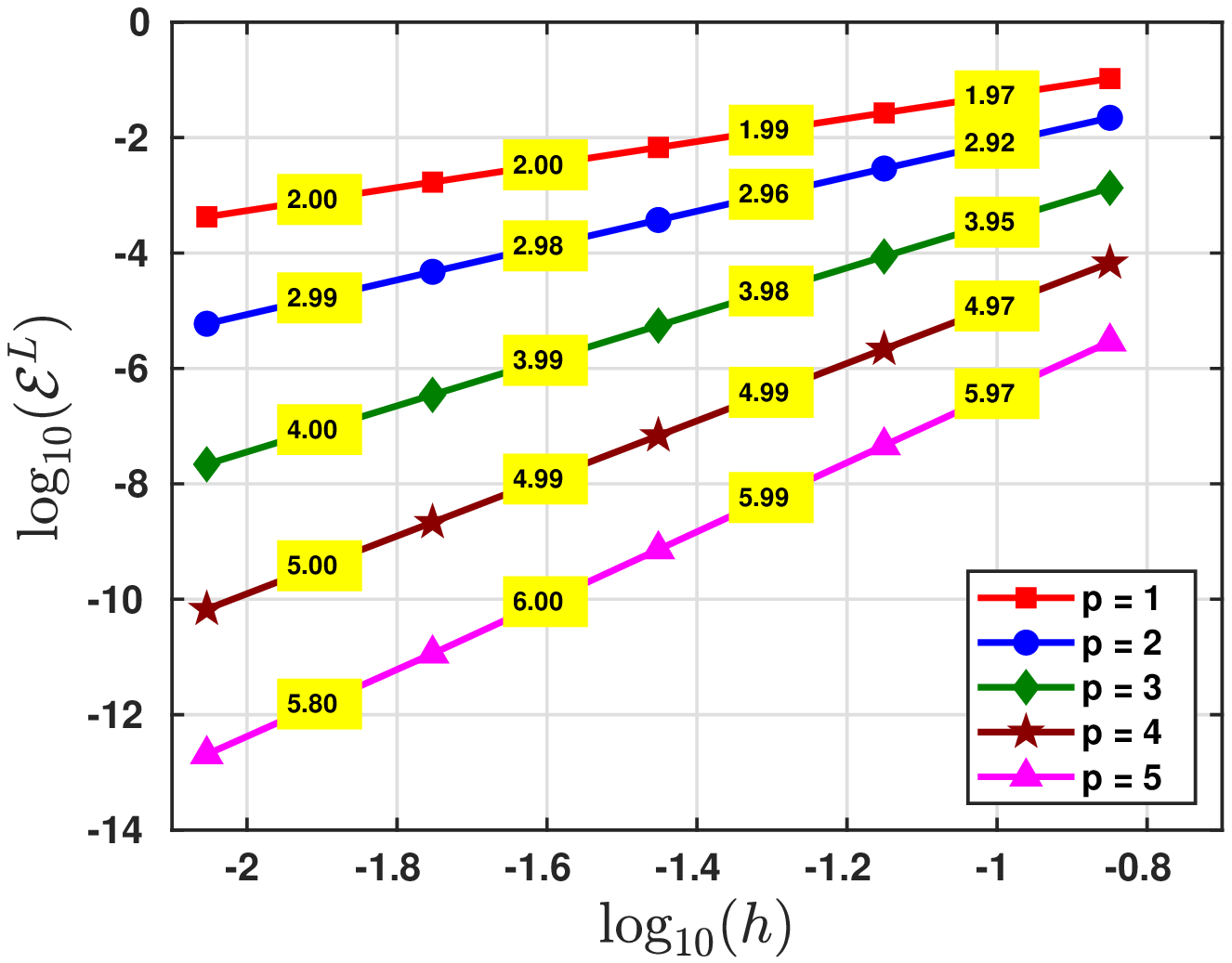}
}
\caption{$\h$-convergence of the errors in~\eqref{exact-errors} for the test case with smooth solution~\eqref{smooth-test-case}.
The numbers in the yellow rectangles denote the experimental orders of convergence.} \label{fig:smooth-1D}
\end{figure}

\subsubsection{Singular solutions}
We assess the convergence of the method for solutions with finite Sobolev regularity.
We use same sequence of meshes as in Subsection~\ref{subsection:patch-test-1D}.
For~$\QT = (0, 1)\times (0, 1)$ and~$\alpha > -1/2$,
we consider the singular solutions
\begin{equation} \label{singular-solution}
    u_\alpha (x,t) = t^{\alpha} \sin(\pi x).
\end{equation}
We have that~$u_\alpha$ and~$\partial_x u_\alpha$ belong to~$H^{\alpha + 1/2 - \epsilon}(0, 1; \mathcal{C}^{\infty}(0, 1))$
for any~$\epsilon > 0$.
The singularity occurs at the initial time.  The errors in~\eqref{exact-errors} are depicted in Figures~\ref{fig:non-smooth-1D-a55} and~\ref{fig:non-smooth-1D-a75} for~$\alpha=0.55$ and~$\alpha=0.75$.
We observe convergence of order $\mathcal{O}\left(h^{\min\{p, \alpha + 1/2\}}\right)$ for the error~$\EcalY$, of order~$\mathcal{O}(\h^{\alpha - \frac12})$ for the error~$\EcalN$, of order~$\mathcal{O}(h^{\alpha})$ for the error~$\EcalU$, and of order~$\mathcal{O}(h^{\alpha + \frac12})$ for the error~$\EcalL$.

For a continuous finite element discretization of formulation~\eqref{continuous-weak}, lower rates of convergence are obtained; see~\cite[Sect. 7.5.3]{Gomez:2023}.

\begin{figure}[ht]
\centering
{
\includegraphics[width = 2.5in]{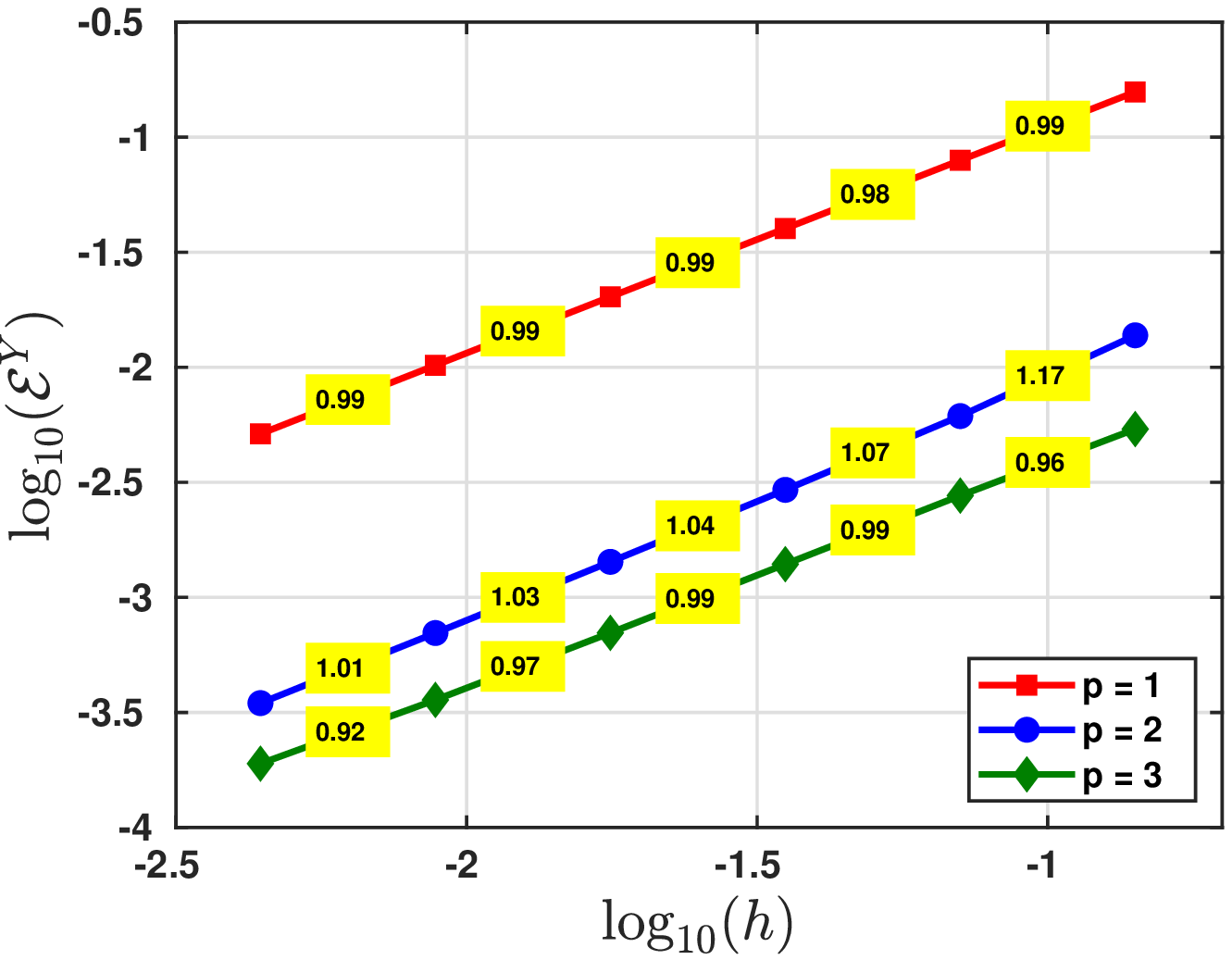}
}
\hspace{0.6cm}
{
\includegraphics[width = 2.5in]{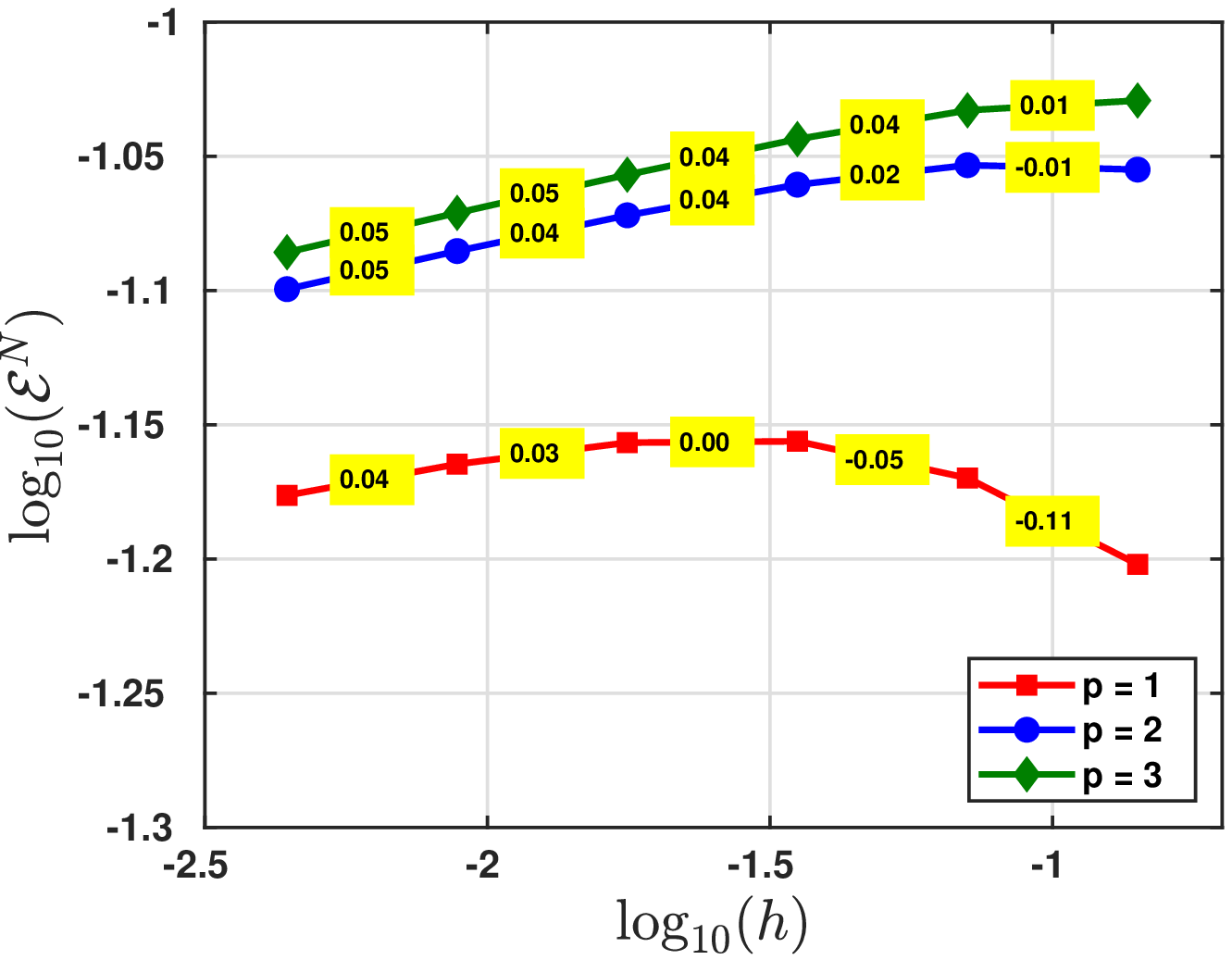}
} \\
{
\includegraphics[width = 2.5in]{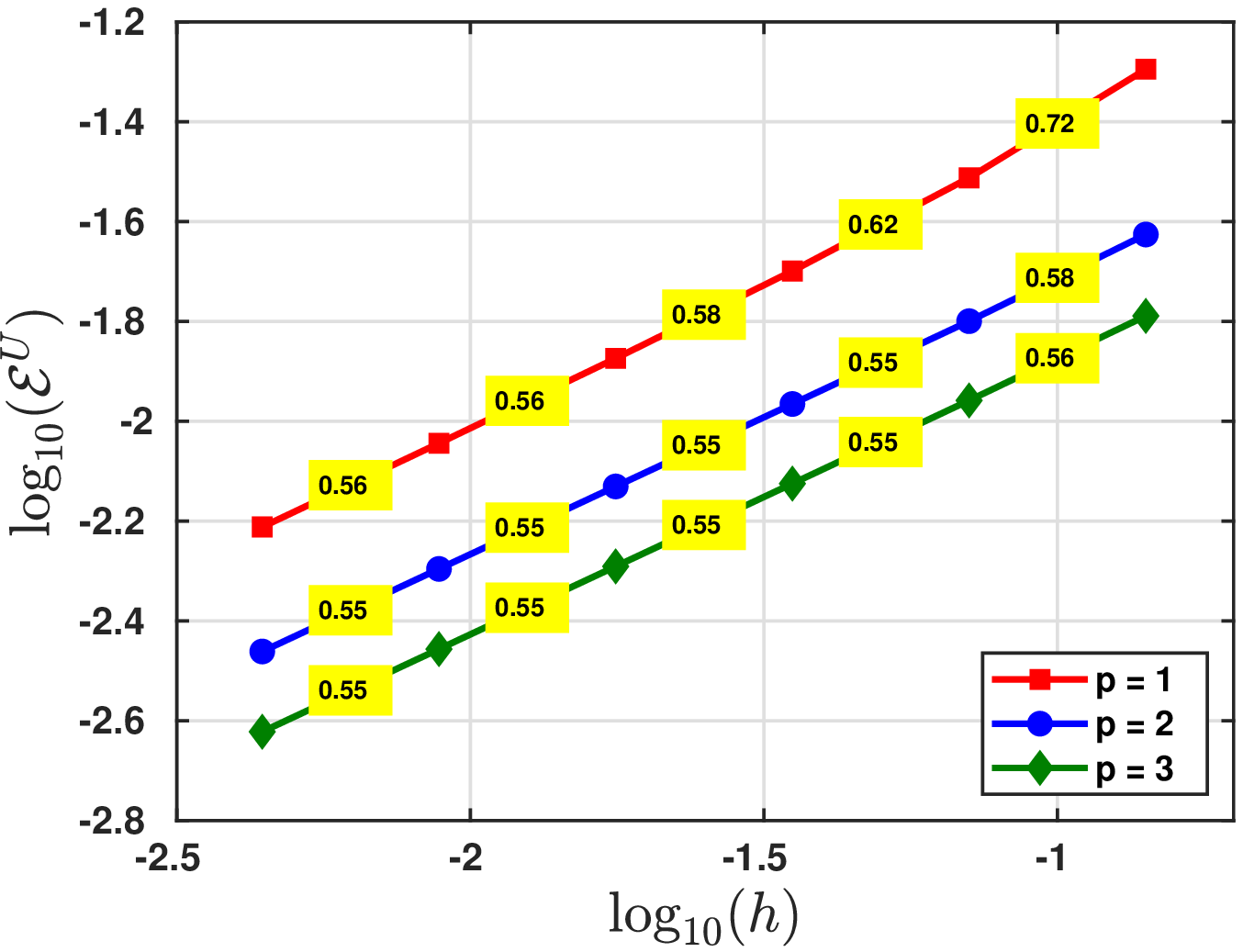}
}
\hspace{0.6cm}
{
\includegraphics[width = 2.4in, height = 2.0in]{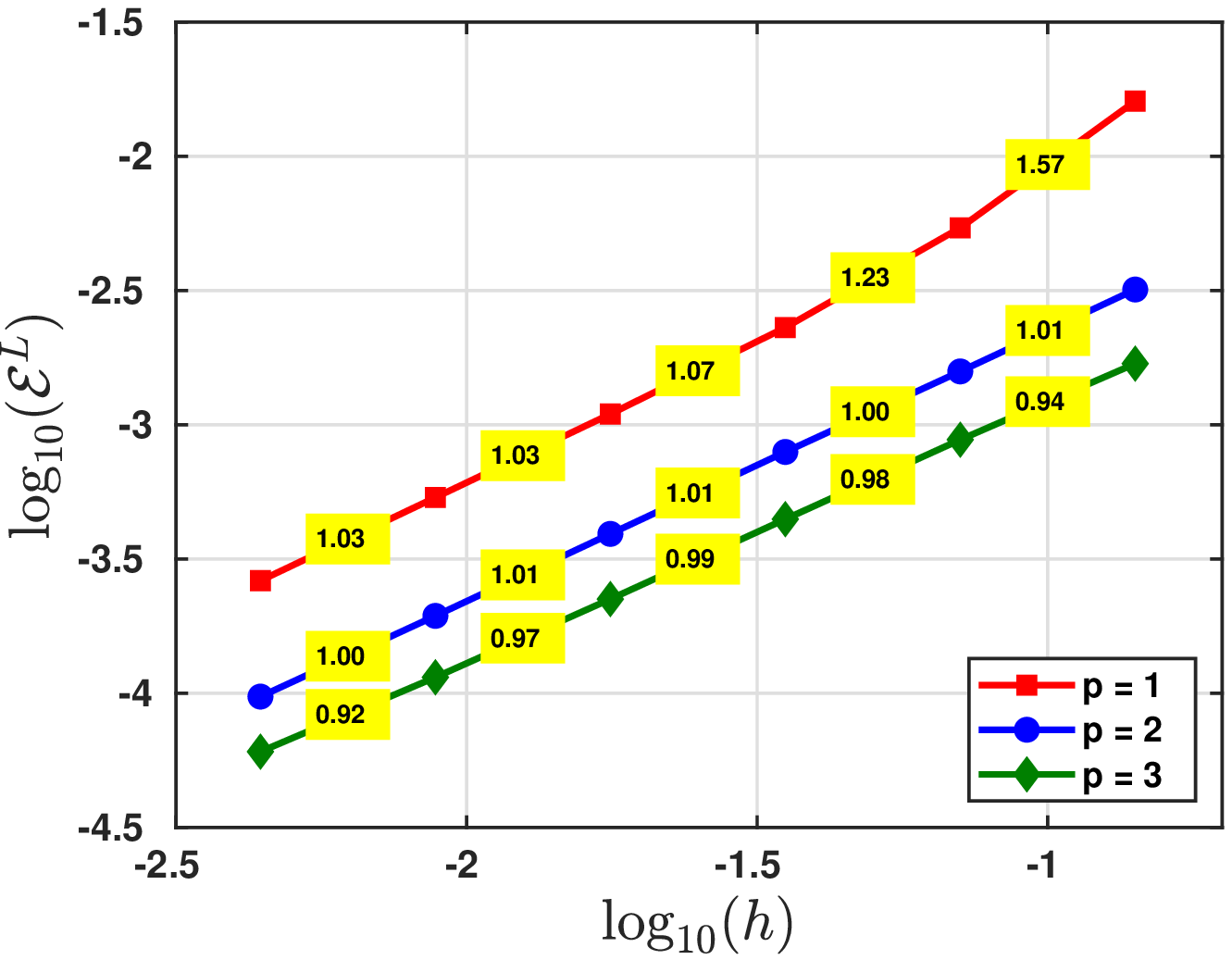}
}
\caption{$\h$-convergence of the errors in~\eqref{exact-errors} for the test case with} singular solution~$u_\alpha$ \eqref{singular-solution}
with~$\alpha = 0.55$.
\label{fig:non-smooth-1D-a55}
\end{figure}

\begin{figure}[ht]
\centering
{
\includegraphics[width = 2.5in]{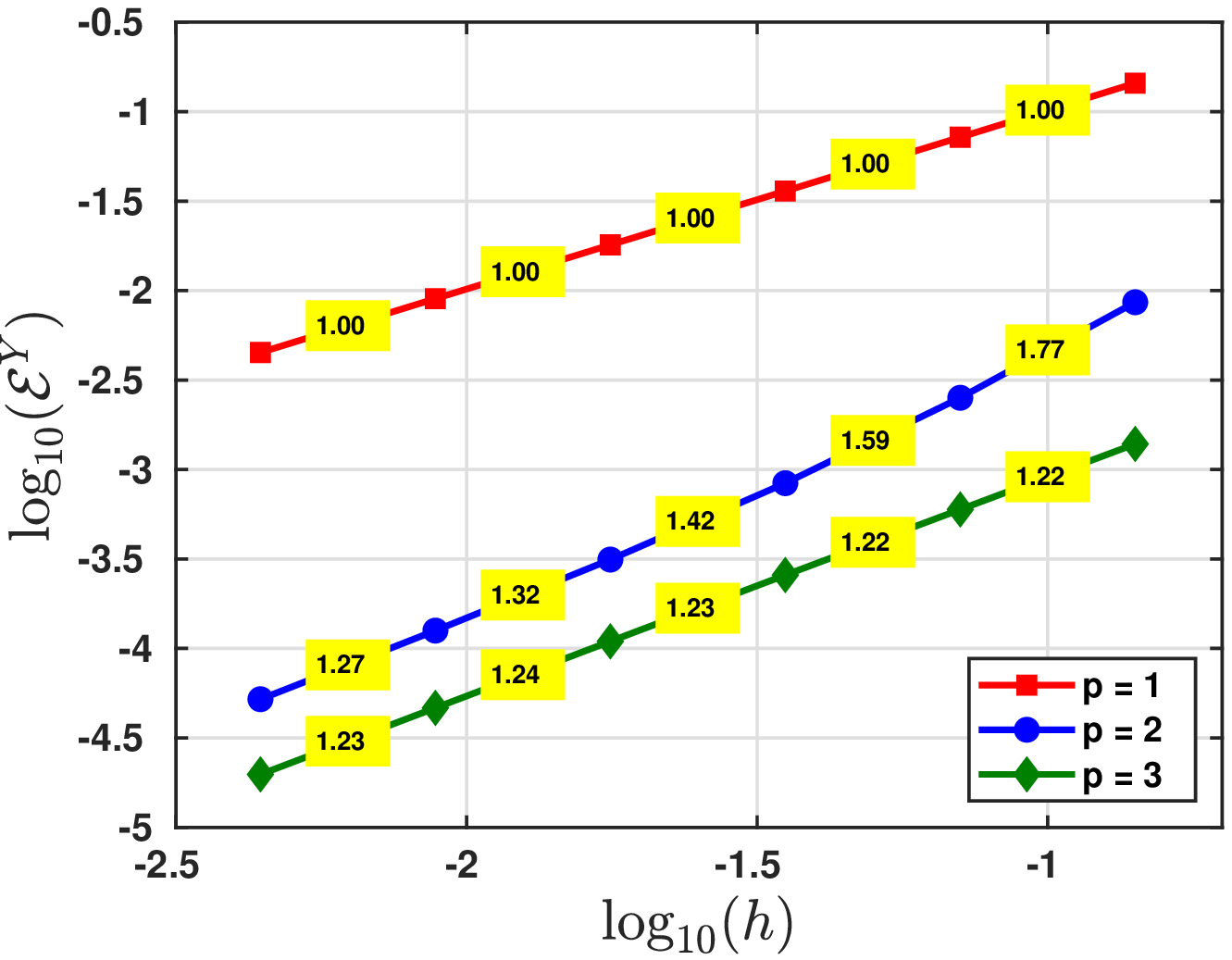}
}
\hspace{0.6cm}
{
\includegraphics[width = 2.5in]{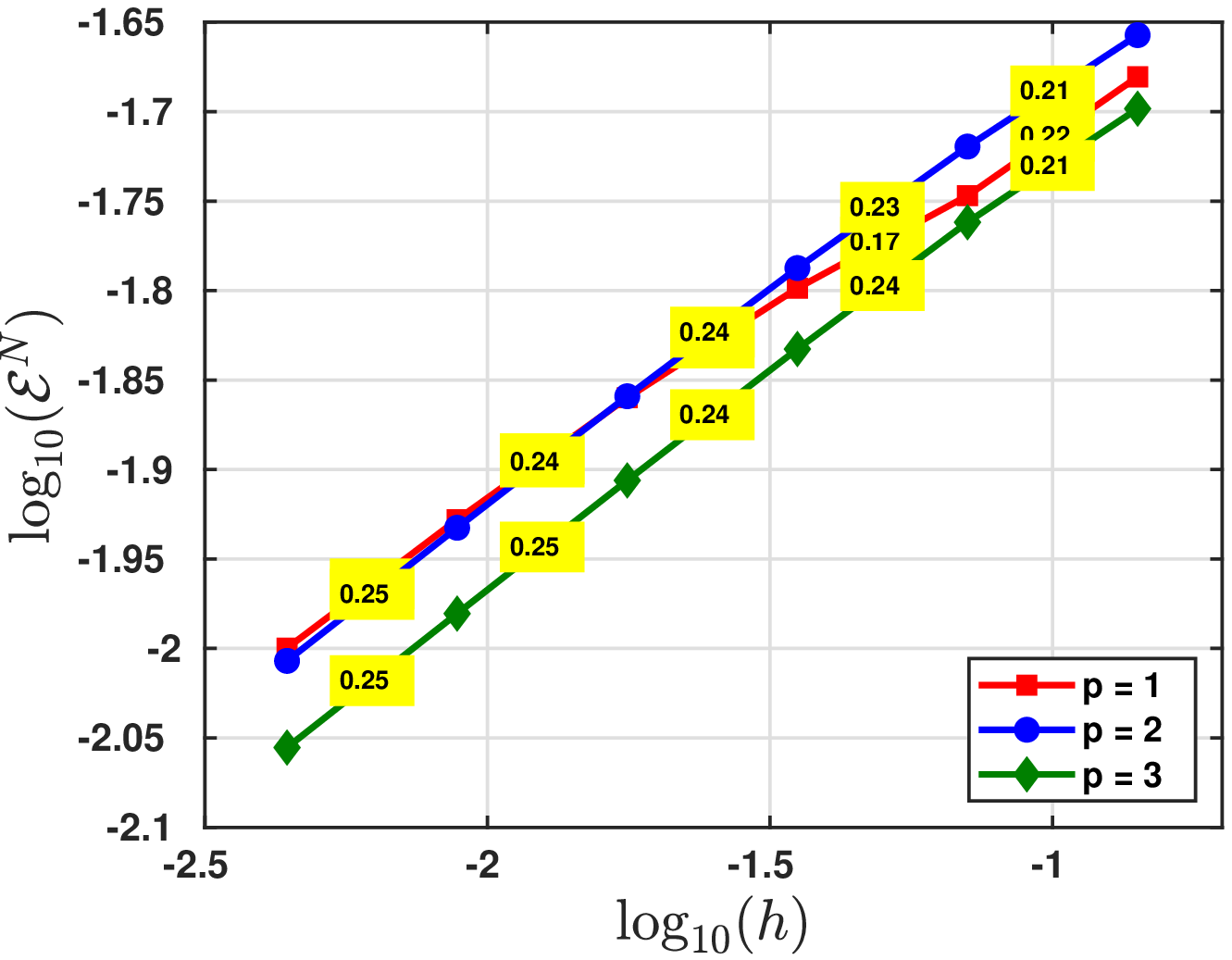}
} \\
{
\includegraphics[width = 2.5in]{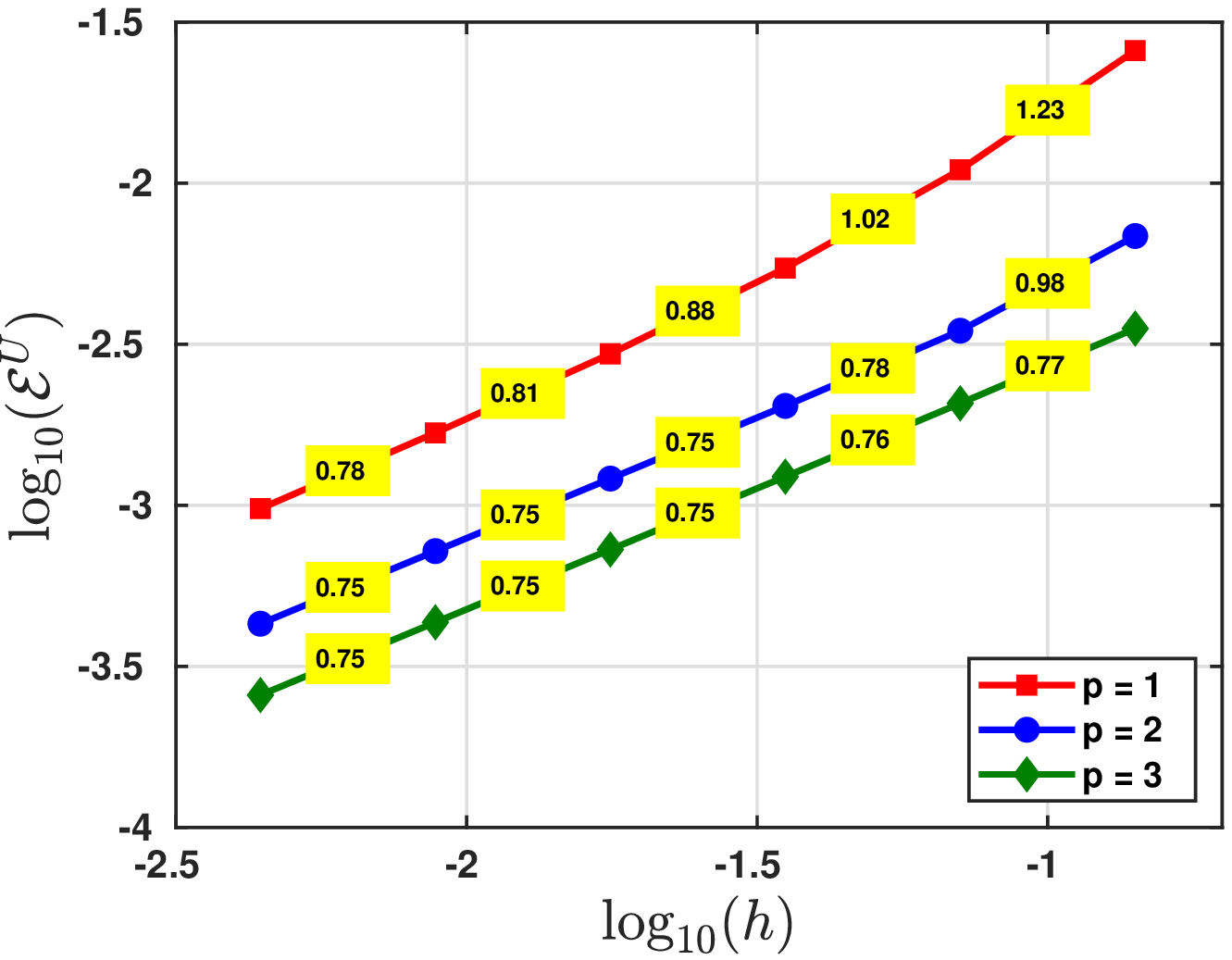}
}
\hspace{0.6cm}
{
\includegraphics[width = 2.5in]{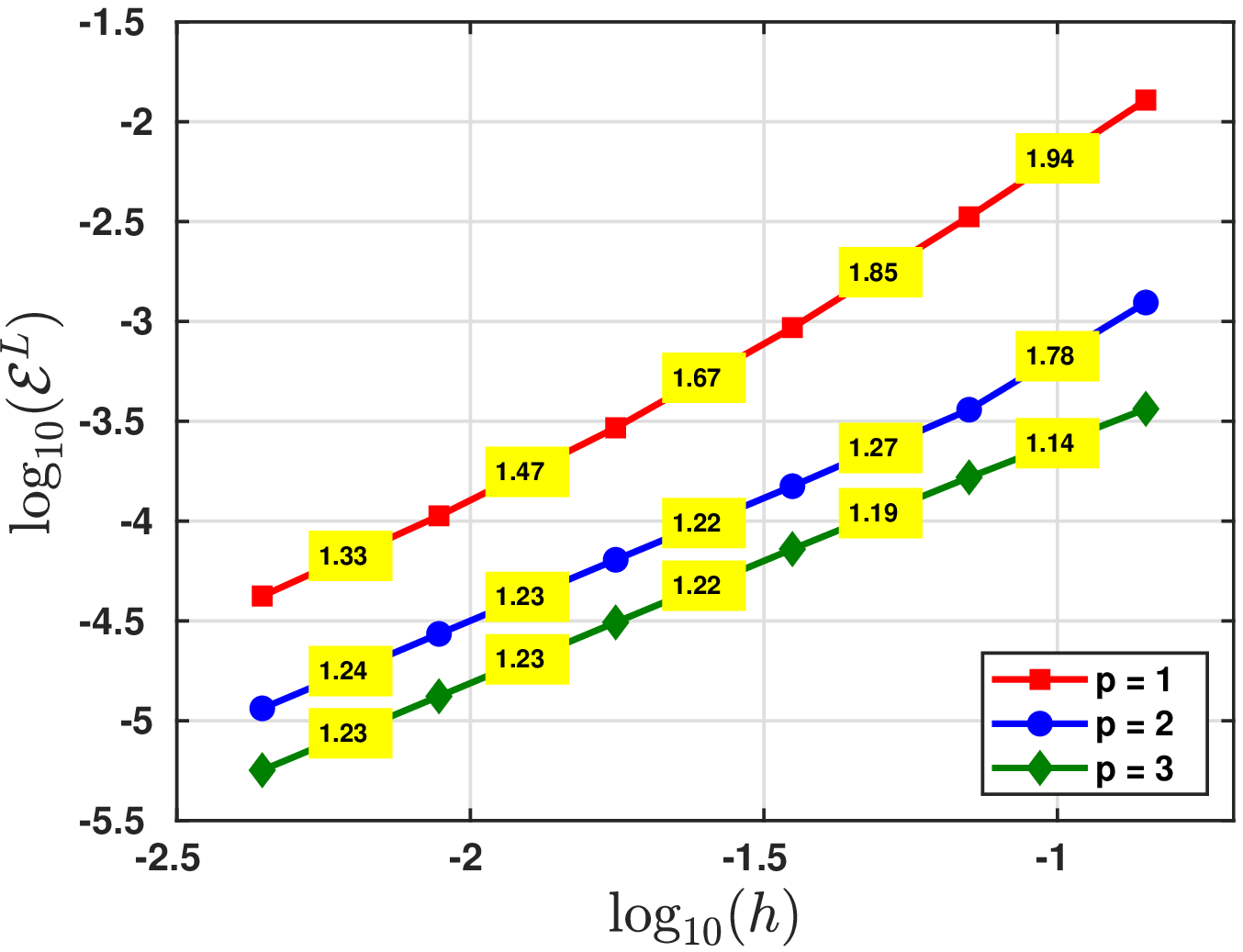}
}
\caption{$\h$-convergence of the errors in~\eqref{exact-errors} for the test case with singular solution~$u_\alpha$ \eqref{singular-solution}
with~$\alpha = 0.75$.}
\label{fig:non-smooth-1D-a75}
\end{figure}

\subsubsection{Incompatible initial and boundary conditions}
On the space-time domain~$\QT = (0, 1) \times (0, 1)$,
we consider the heat equation problem~\eqref{continuous-strong}
with zero source term~$(f = 0)$,
homogeneous Dirichlet boundary conditions~($u = 0$ on~$\partial \Omega \times (0,T)$),
and constant initial condition ($u = 1$ on~$\Omega \times \{0\}$).
The corresponding exact solution is given by
the Fourier series
\begin{equation} \label{incompatible-data-solution}
u(x, t) 
= \sum_{n = 0}^{\infty} \frac{4}{(2n + 1) \pi}
    \sin\left( (2n + 1) \pi x \right)
    \exp\left(-(2n + 1)^2 \pi^2 t\right).
\end{equation}
Due to the incompatibility of the initial and boundary conditions,
$u$ is discontinuous at~$(0,0)$ and~$(1,0)$, 
and does not belong to~$H^1(\QT)$
but belongs to $H^s\left(0,1;H^1_0(0,1)\right)$ for any $s<1/4$; see~\cite[Sect.~7.1]{SchoetzauSchwab}.
Therefore, the rates of convergence obtained cannot be predicted by Theorem~\ref{theorem:a-priori:error-estimates}.

In Figure~\ref{fig:VEM-FEM-comparison},
we show the errors obtained with~$p = 1, 2$
on a sequence of uniform Cartesian meshes
for the proposed VEM
and on a sequence of structured triangular meshes
for the continuous finite element method in~\cite{steinbach2015CMAM}.
The continuous finite element method does not converge in the~$Y$-norm,
while the error~$\EcalY$ of the proposed VEM converges with order~$\mathcal{O}(h^{1/4})$.
For the computation of the error,
we truncate the series~\eqref{incompatible-data-solution} at~$n = 250$.
\begin{figure}[ht]
\centering
\subfloat[Proposed VEM]{
\includegraphics[width = 2.5in]{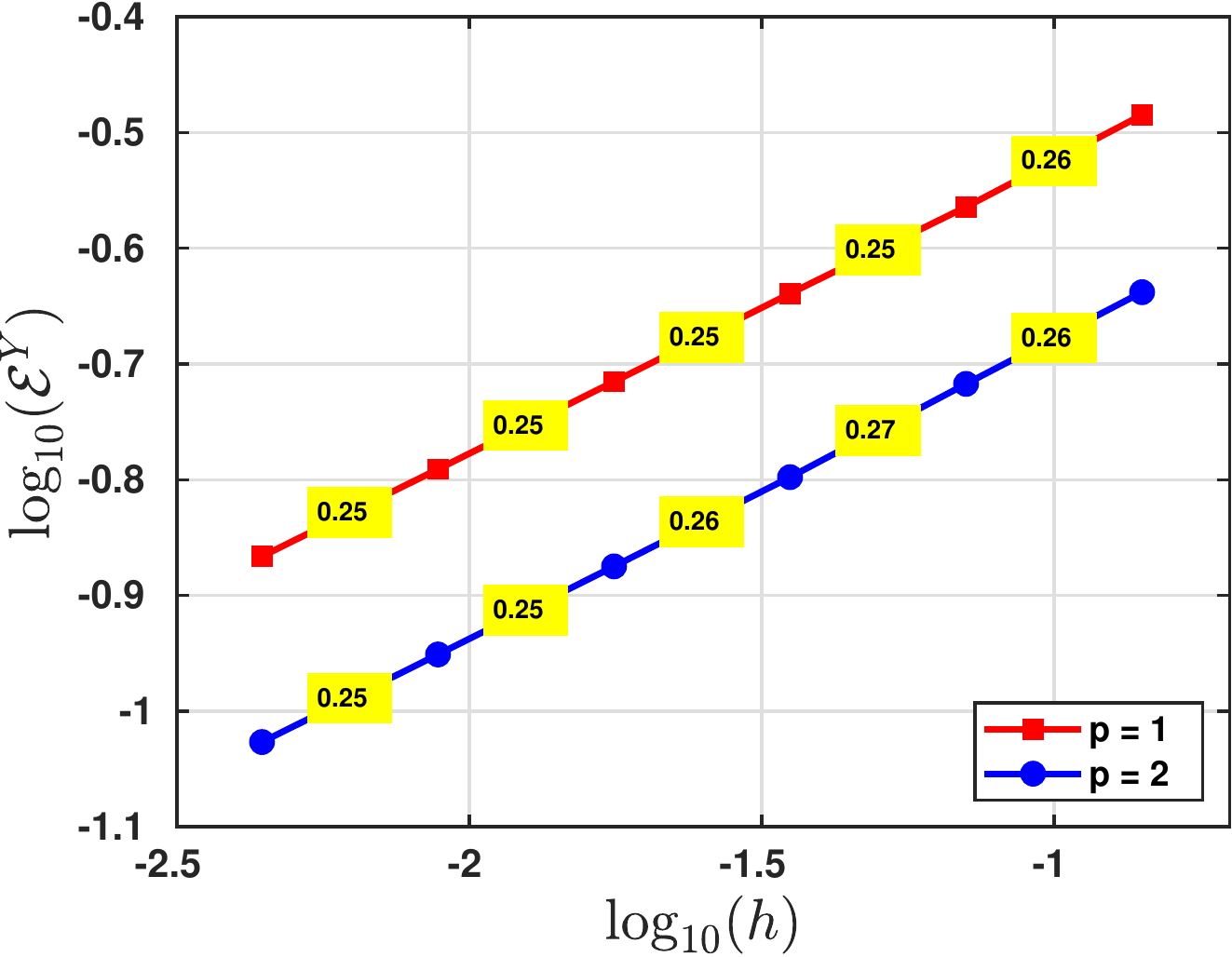}
}
\hspace{0.6cm}
\subfloat[Continuous finite element method]{
\includegraphics[width = 2.5in]{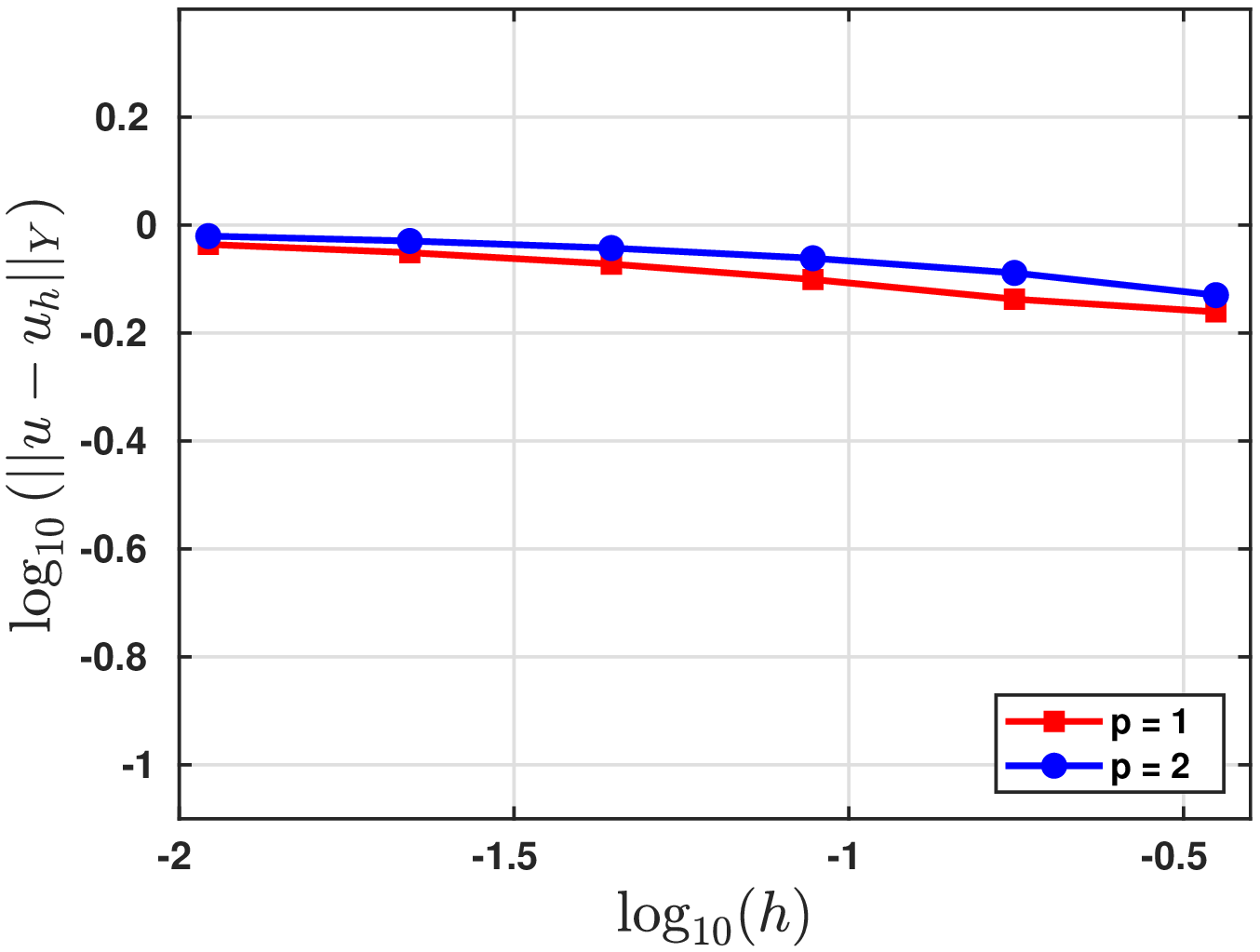}
}
\caption{$\h$-convergence for the test case with exact solution~\eqref{incompatible-data-solution} 
with incompatible initial and boundary conditions.}
\label{fig:VEM-FEM-comparison}
\end{figure}
%

\subsubsection{Increasing the degree of approximation}
We are also interested in the performance of the $\p$-version of the method,
i.e., we fix a mesh and increase the degree of approximation.
This is worth investigating 
also in view of the design of~$\h\p$ refinements.
We consider the smooth solution test case from Figure~\ref{exp:smooth-solution}
with a fixed mesh with~$\htime = \hEx = 0.1$.
The results shown in Figure~\ref{fig:p-convergence} in \emph{semilogy} scale.
We observe the expected exponential convergence 
in terms of the square root of $N_{DoFs}$ for all the VEM errors.
\begin{figure}[ht]
\centering
{
\includegraphics[width = 2.5in]{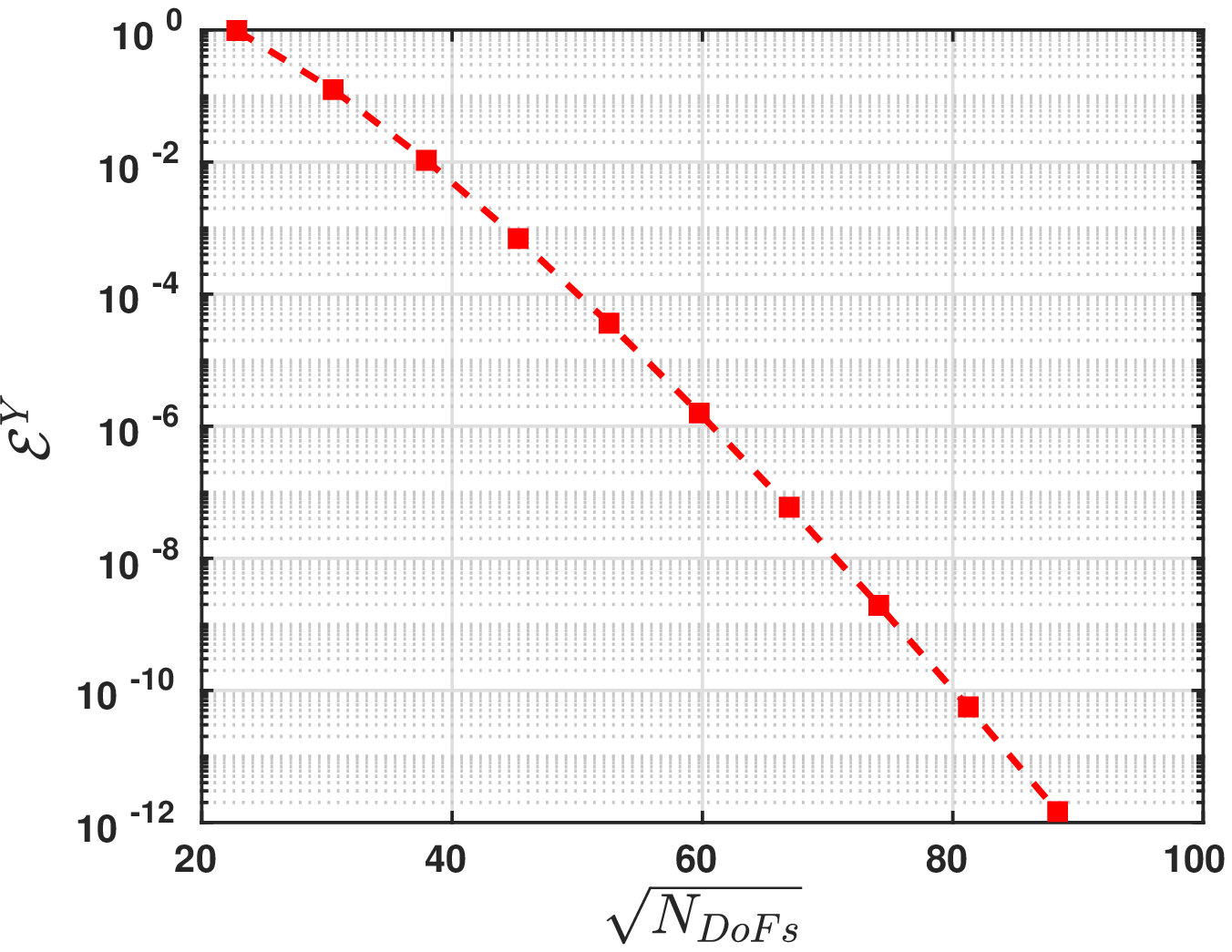}}
\hspace{0.6cm}
{
\includegraphics[width = 2.5in]{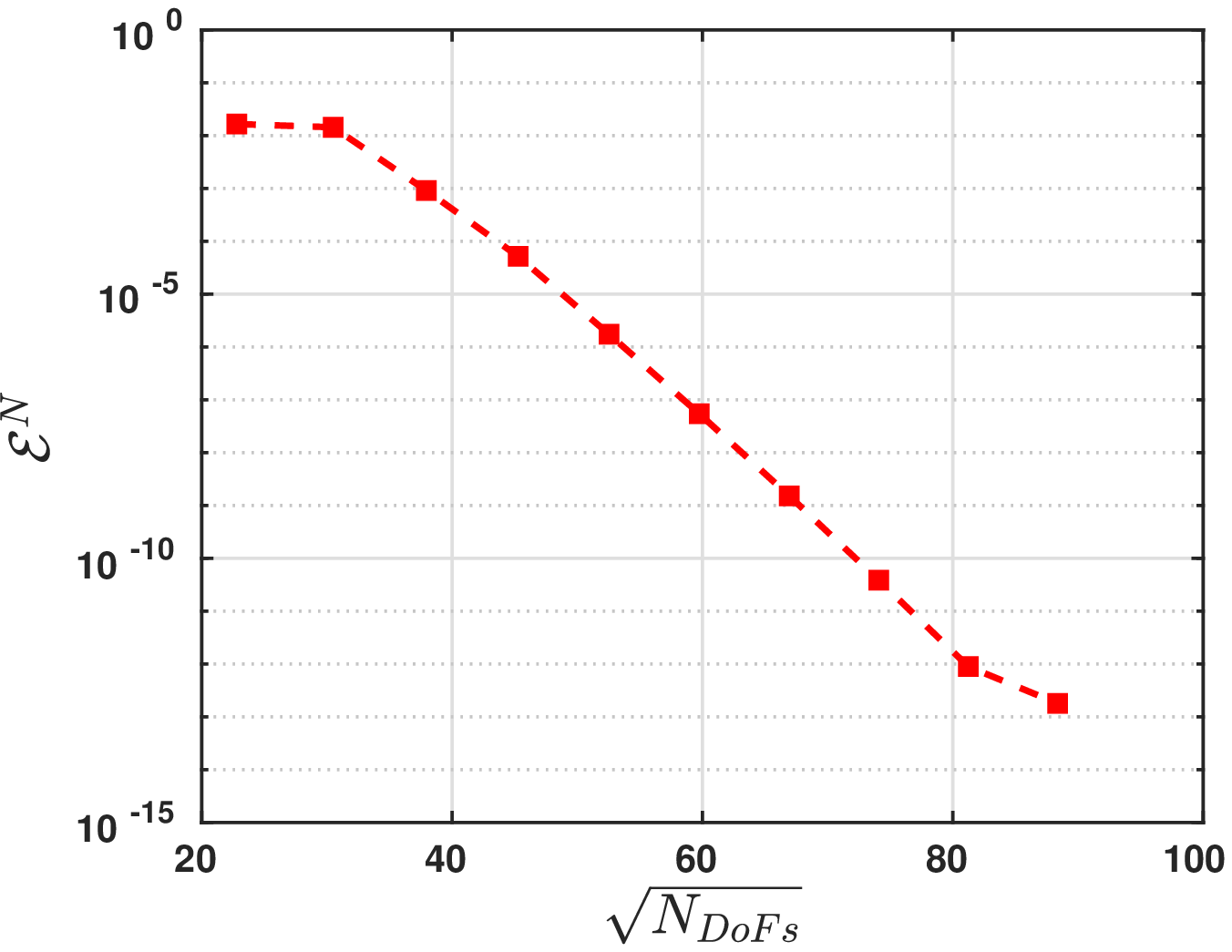}}\\
{
\includegraphics[width = 2.5in]{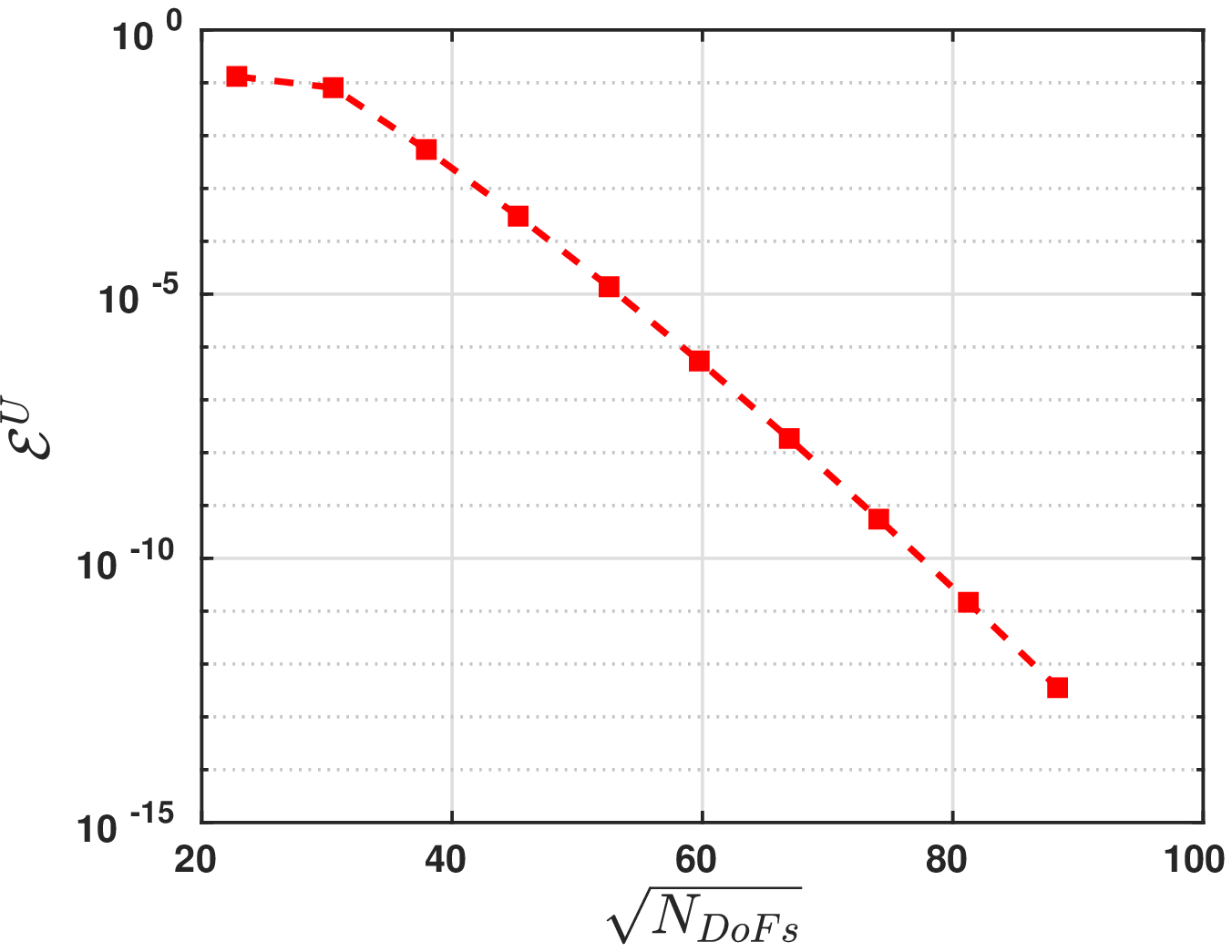}}
\hspace{0.6cm}
{
\includegraphics[width = 2.5in]{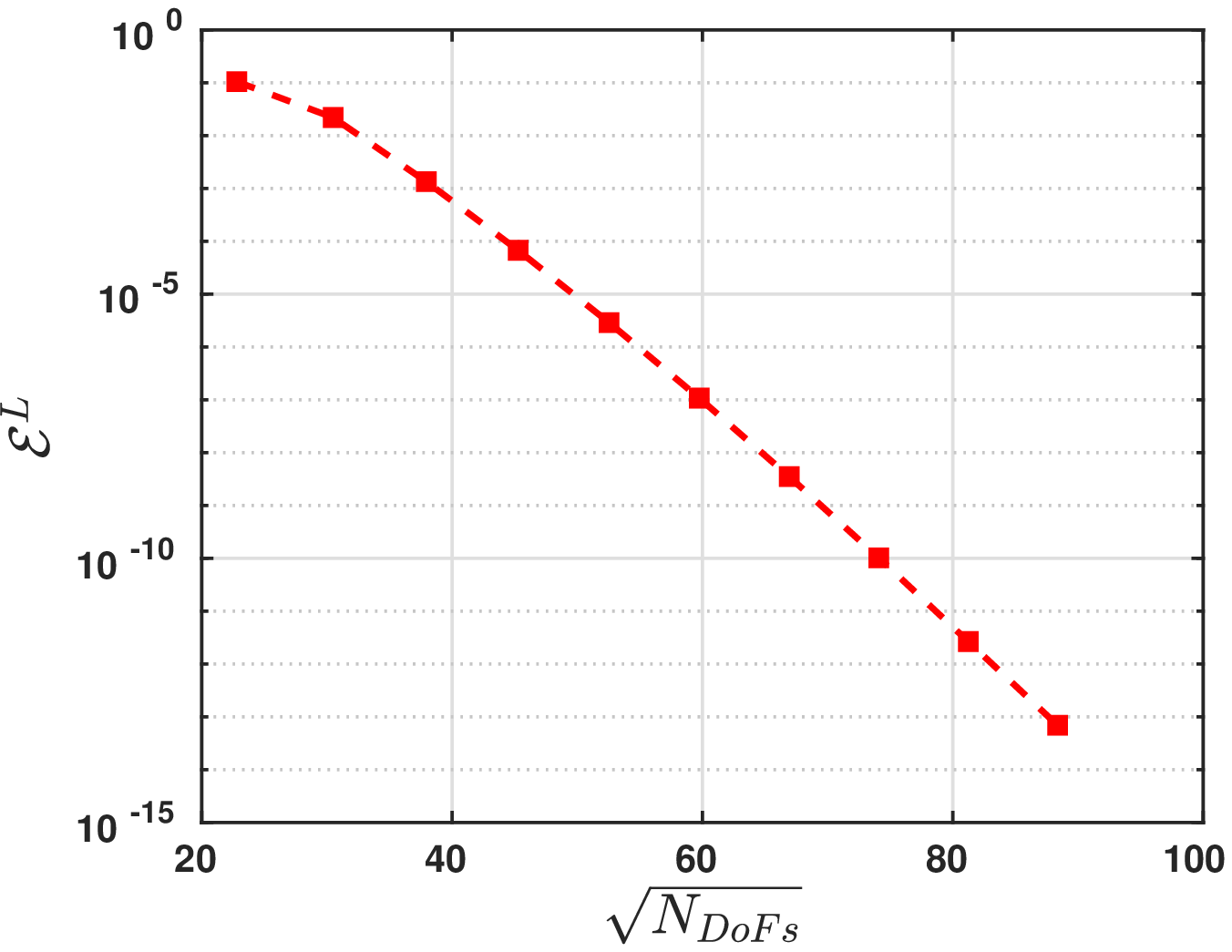}}
\caption{$p$-convergence of the errors in~\eqref{exact-errors} for the test case with smooth solution in~\eqref{smooth-test-case}.}
\label{fig:p-convergence}
\end{figure}

\subsection{Results in (2+1)-dimension}  \label{subsection:2+1}
We use tensor-product-in-time meshes and uniform partitions of the time interval~$(0, T)$,
and discretize the spatial domain~$\Omega$
with sequences of quadrilateral meshes
such as that in Figure~\ref{fig:last} (left panel).
We checked that the method passes the patch test also in the (2+1) dimensional case.
We do not report the results for the sake of brevity.

On~$\QT = (0, 1)^2 \times (0, 1)$,
we consider 
\begin{equation} \label{smooth-test-case-2+1}
u(\x, t) = \exp(-t) \sin(\pi x_1) \sin(\pi x_2).
\end{equation}
In Figure~\ref{fig:last} (right panel),
we display the rates of convergence using different values of~$p$
and observe the expected rates of convergence for the error~$\EcalY$.

\begin{figure}[!ht]
\centering
{
\includegraphics[width = 2.5in, height = 2.1in]{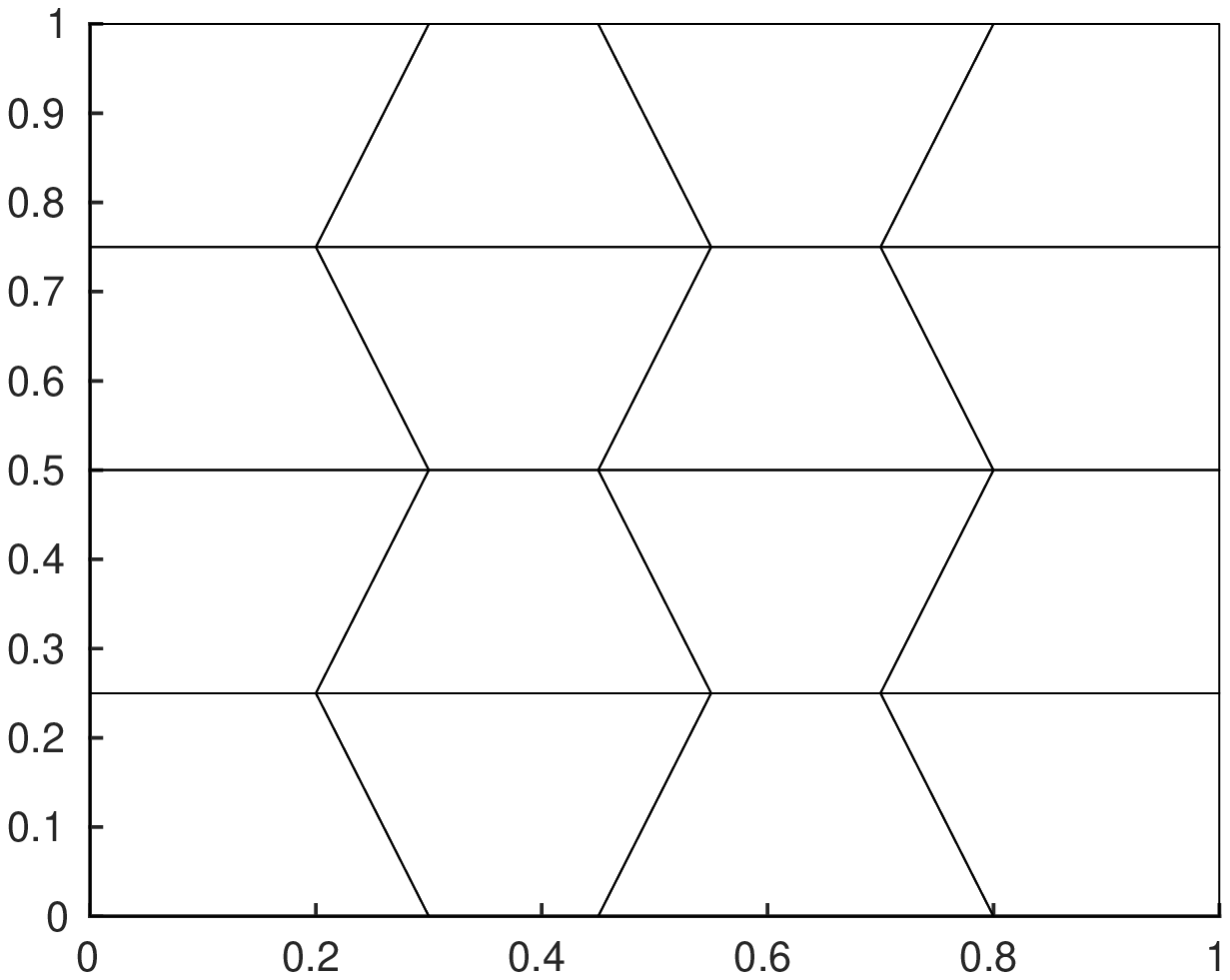} 
}
{ 
\includegraphics[width = 2.5in]{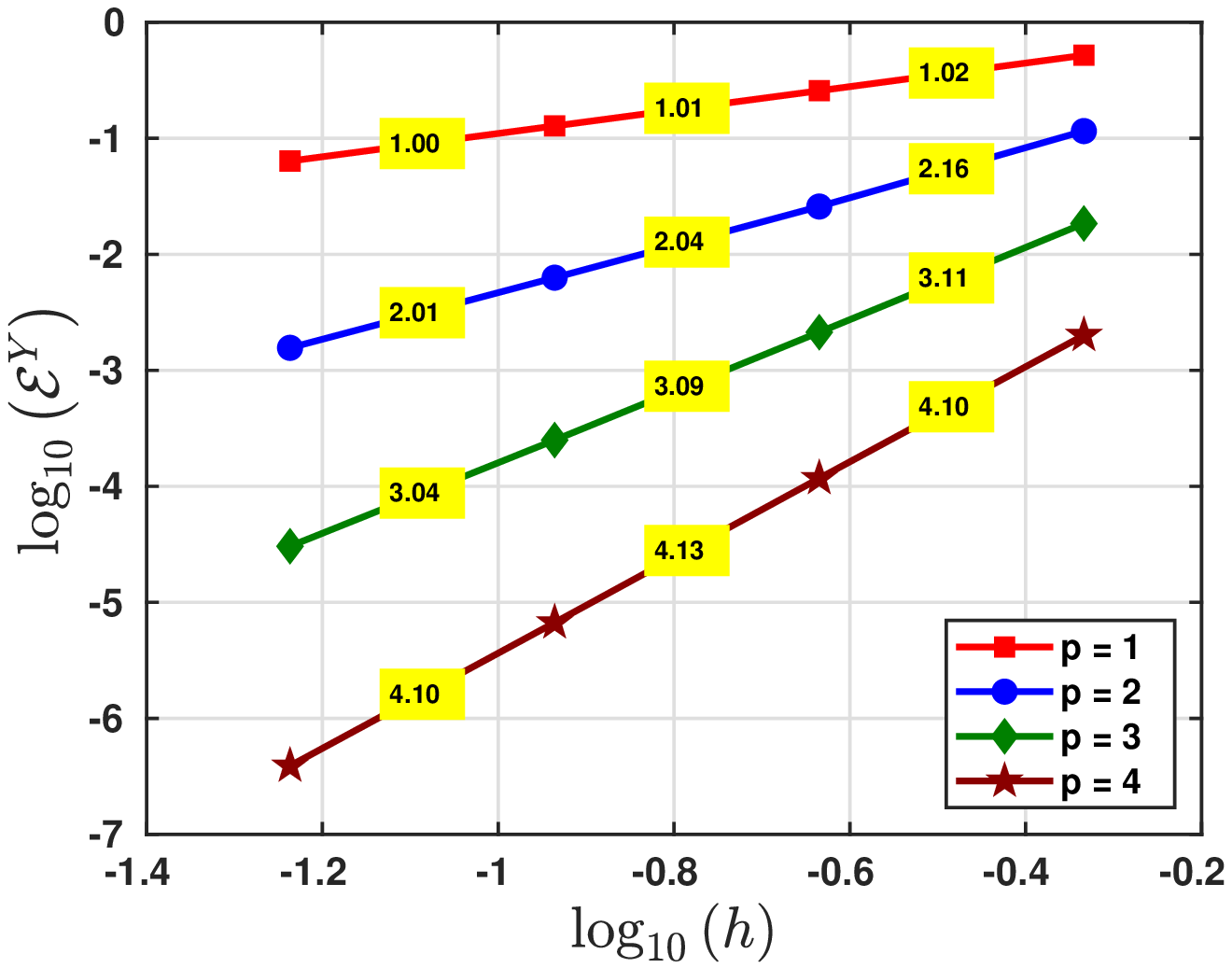}
}
\caption{\emph{Left panel}:
example of mesh for the 2-dimensional spatial domain used in the numerical experiments.
\emph{Right panel}: $\h$-convergence for the test case with smooth solution~\eqref{smooth-test-case-2+1}.}
\label{fig:last}
\end{figure}

\section{Conclusions} \label{section:conclusion}
We designed and analyzed a space-time virtual element method for the heat equation
based on a standard Petrov-Galerkin variational formulation.
The advantages of using the proposed space-time VEM
over standard space-time finite element methods
are that it 
allows for decomposing the linear system stemming from the method
into smaller systems associated with different time slabs;
can be modified into a Trefftz variant;
permits the treatment of incompatible initial and boundary data.
We proved well posedness of the method and optimal \textit{a priori} error estimates.
Numerical results validate the expected rates of convergence.

In~\cite{Gomez-Mascotto-Perugia:2023}, 
the method introduced in this paper has been extended to more general prismatic meshes
with hanging facets and variable degrees of accuracy,
enabling the implementation of $\h\p$-adaptive mesh refinements.
Tests of an adaptive procedure driven by a residual-type error indicator are also presented there.

\section*{Acknowledgements}
The authors have been funded by the Austrian Science Fund (FWF) through the projects F~65 (I.~Perugia)
and P~33477 (I. Perugia, L. Mascotto),
by the 
Italian 
Ministry of University and Research 
through the PRIN project ``NA-FROM-PDEs''
(A. Moiola, S. G\'omez),
and the ``Dipartimenti di Eccellenza'' Program (2018-2022) - Dept. of Mathematics, University of Pavia (A. Moiola).


\bibliographystyle{plain}

\end{document}